\documentclass{elsarticle}

\usepackage{lineno,times}
\modulolinenumbers[5]

\usepackage{hyperref}
\usepackage{amsmath, amsthm, amscd, amsfonts, amssymb, graphicx, color}
\newtheorem{theorem}{Theorem}[section]
\newtheorem{main}{Main Theorem}
\newtheorem{lemma}[theorem]{Lemma}
\newtheorem{assumption}{Assumption}
\newtheorem{proposition}[theorem]{Proposition}

\theoremstyle{definition}
\newtheorem{definition}[theorem]{Definition}

\newtheorem{remark}[theorem]{Remark}
\numberwithin{equation}{section}
\usepackage{latexsym}
\usepackage{amsmath}
\usepackage{amssymb}
\allowdisplaybreaks[4]

 \usepackage[text={135mm,210mm},centering]{geometry}

\newcommand{\be}{\begin{equation}\begin{aligned}}
\newcommand{\ee}{\end{aligned}\end{equation}}
\newcommand{\ben}{\begin{equation}\nonumber\begin{aligned}}
\newcommand{\dist}{{\rm dist}}
\newcommand{\R}{\mathbb{R}}
\newcommand{\N}{\mathbb{N}}
\newcommand{\D}{\mathcal{D}}
\newcommand{\A}{\mathcal{A}}
\newcommand{\B}{\mathfrak{B}}

\newcommand{\TT}{\mathbb{T}}
\newcommand{\LL}{\mathbb{L}_{per}}

\newcommand{\HH}{\mathbb{H}_{per}}
\renewcommand{\d}{{\rm{d}}}

    \usepackage{mathabx}

\begin{document}

\begin{frontmatter}

\title{   $(H,H^3)$-smoothing effect and convergence of solutions of  stochastic two-dimensional anisotropic Navier-Stokes equations driven by colored noise
}

\author[lh]{Hui Liu}

\ead{ss\_liuhui@ujn.edu.cn}

\author[sd]{Dong Su}
\ead{dsu224466@163.com}

\author[scf]{Chengfeng Sun\corref{mycorrespondingauthor}}
\cortext[mycorrespondingauthor]{Corresponding author}
\ead{sch200130@163.com}

\author[xj]{Jie Xin}

\ead{fdxinjie@sina.com}

\address[lh]{School of Mathematical Sciences, University of Jinan, Jinan 250022, PR China}
\address[sd]{School of Mathematics, Nanjing Audit University, Nanjing 211815, PR China}
\address[scf]{School of Applied Mathematics, Nanjing University of Finance and Economics, Nanjing 210023, PR China}
\address[xj]{School of Information Engineering, Shandong Youth University of Political Science, Jinan 250103, PR China}

\begin{abstract}

This paper is devoted to the higher regularity and convergence of solutions of anisotropic Navier-Stokes (NS) equations  with  additive colored noise and  white noise  on two-dimensional torus $\mathbb T^2$.  Under the conditions that the external force $f(\textbf{x})$ belongs to the phase space $ H$ and the noise intensity function $h(\textbf{x})$  satisfies $\|\nabla h\|_{L^\infty} \leq \sqrt{\pi\delta}  \frac{\nu \lambda_1}{2}$,  it was proved that the random anisotropic NS equations possess a tempered  $(H,H^2)$-random attractor  whose (box-counting) fractal  dimension in $H^2$ is finite. This was achieved by establishing, first, an $H^2$ bounded   absorbing set and, second,  an $(H,H^2)$-smoothing effect of the system which lifts the compactness and finite-dimensionality  of the attractor in $H$ to that in $H^2$.   Since the force $f$ belongs only to $H$, the $H^2$-regularity of solutions as well as the $H^2$-bounded absorbing set was constructed by an indirect approach of estimating the $H^2$-distance     between the solution  of the random anisotropic NS equations and that of the corresponding deterministic anisotropic NS equations. When the external force $f(\textbf{x})$  belongs to $H^2$ and the noise intensity function $h(\textbf{x})$ satisfies the Assumption \ref{assum1}, it was proved that the random anisotropic NS equations possess a tempered  $(H,H^3)$-random attractor  whose (box-counting) fractal  dimension in $H^3$ is finite. Finally, we prove the  upper semi-continuity of random attractors  and the convergence of solutions of \eqref{8.2}  as $\delta\rightarrow0$ in the spaces $(H,H)$, $(H,H^1)$, $(H^1,H^2)$ and $(H^2,H^3)$, respectively.
  \end{abstract}

\begin{keyword}
\texttt{Navier-Stokes equations;
   $H^3$ regularity;  Upper semi-continuity; Fractal dimension; Convergence.}
\MSC[2020] 35B40, 35B41, 60H15
\end{keyword}
\end{frontmatter}

\tableofcontents

\section{Introduction}

In this  paper we study the asymptotic dynamics of the following stochastic two-dimensional anisotropic Navier-Stokes (NS) equations driven by colored noise  on  $\mathbb{T}^2=[0,L]^2$, $L>0$:
 \begin{align} \label{1}
 \left \{
  \begin{aligned}
 &  \frac{\partial \textbf{u}_h}{\partial t}-  \nu \left(
\begin{array}{ccc}
\partial_{yy}u_h\\ \partial_{xx}v_h
\end{array}\right)+(\textbf{u}_h\cdot\nabla)\textbf{u}_h+\nabla p=  f(\textbf{x}) +h(\textbf{x})\zeta_{\delta}(\theta_t\omega),\\
  & \nabla\cdot \textbf{u}_h=0,\\
  &\textbf{ u}_h(0)=\textbf{u}_{h,0},
  \end{aligned} \right.
\end{align}
 endowed with periodic boundary conditions and $t>0$. Here $\textbf{u}_h=(u_h,v_h)$ is the velocity field of fluid and $\textbf{x}=(x,y)$, $ \nu >0$ is the kinematic viscosity, $p$ is the scalar pressure and $f \in H $  represents volume forces that are applied to the fluid, where $H= \left \{u\in \LL ^2:
  \int_{\mathbb T^2}u \, \d x=0,\,  {\rm div}\, u=0 \right \}$ is the phase space of the system.  $\zeta_{\delta}(\theta_t\omega)$ is the colored noise with correlation time $\delta>0$. $(\Omega,\mathcal{F},\mathbb{P})$ with  $\mathcal{F}$  the Borel $\sigma$-algebra induced by the compact open topology of $\Omega$, and $\mathbb{P}$  the corresponding Wiener measure on $(\Omega,\mathcal{F})$. The subscript $``_{h}"$ indicates the dependence on the noise intensity $h(\textbf{x})$.
When $h\equiv0$, we introduce the deterministic anisotropic  NS equations on  $\mathbb{T}^2$:
 \begin{align}\label{2}
 \left \{
  \begin{aligned}&
   \frac{\partial \textbf{u}}{\partial t}-  \nu \left(
\begin{array}{ccc}
\partial_{yy}u\\ \partial_{xx}v
\end{array}\right)+(\textbf{u}\cdot\nabla)\textbf{u}+\nabla p=f(\textbf{x}),\\
  & \nabla\cdot \textbf{u}=0,\\
   &\textbf{u}(0)=\textbf{u}_{0}.
  \end{aligned} \right.
\end{align}

The anisotropic  NS equations are fundamental mathematical models in fluid mechanics, see, e.g. \cite{bahouri,liang,paicu,temam}. Global well-posedness and long time behaviour of the two-dimensional or three-dimensional anisotropic   NS equations  have been investigated in \cite{bessaih,chemin,liujfa,liuarma,paicucmp,zhang}. In order to prove the asymptotic behavior of the  anisotropic  NS equations, we first introduce some results for the global attractor of NS equations, see, e.g. \cite{constantin,ju,liuvx1,rosa,sun,wang,zhao}. Meanwhile, asymptotic behavior of the non-autonomous NS equations has been investigated in \cite{cui24siads,cui24ma,hou,langa,lukaszewicz}. \cite{Brzezniak13,Brzezniak2018,Brzezniak,crauel97jdde,crauel,cui2025,dong,li,liu2018,mohammed,wang12,wang2023} studied the regularity of random attractors of the stochastic NS equations. Inspired by \cite{robinson01,robinson13} and the similar method, if $f\in H^s$, $s>0$, the regularity of the global attractor of the system \eqref{2} has been proved in $H^s$, and then the global attractor $\A_0$ of the system \eqref{2} is bounded in $H^{s+1}$. \cite{julia} proved the $H^2$-boundedness of the pullback attractors for non-autonomous  two-dimensional NS equations, and then we will prove the $H^2$-boundedness of the global attractor for the  two-dimensional anisotropic  NS equations (see Section \ref{sec3}).

Recently, the anisotropic  fluid equations have been investigated by many authors in \cite{chemin2000}. Under certain conditions, \cite{adhikari,adhikari2011,cao2013,danchin,larios,miao} studied the global existence and global regularity of  the two-dimensional or three-dimensional anisotropic Boussinesq equations. Meanwhile, \cite{cao,cao2014,dong2019,du,ren} proved the global regularity and decay for the two-dimensional anisotropic magnetohydrodynamics (MHD) equations.

By the Proposition 12.4 in \cite{robinson01}, we can deduce that the global attractor $\A_0$ of the two-dimensional anisotropic  NS equations \eqref{2} is a global $(H,H^1)-$attractor. That is to say, the global attractor $\A_0$ is a compact set in $H^1$ and attracts every bounded set in $H$ under the norm of $H^1$. The $H^2$-boundedness of the global attractor for the  two-dimensional anisotropic  NS equations implies that the strong solutions in $\A_0$ are smooth solutions, see, e.g. \cite{Ladyzenskaja,robinson01,temam}, up to now the attractor was not obtained to be $(H,H^2)$ and $(H,H^3)$.

Recently, under general settings that the $(H,H^1)$-attractor of the two-dimensional anisotropic  NS equations has been investigated. By using the asymptotic compactness, \cite{ju} proved the $(H,H^1)$-global attractor of the two-dimensional NS equations on some unbounded domains. For the non-autonomous and random cases, \cite{cui21jde} proved the finite fractal dimensional of $(H,H^1)$-uniform attractor of the two-dimensional NS equations, and \cite{zhao22} obtained the $(H,H^1)$-trajectory attractor of the three-dimensional modified NS equations, while \cite{li} proved the $(H,H^1)$-random uniform attractor of the non-autonomous stochastic two-dimensional NS equations. $(H,H^1)$-type attractors of the other fluids equations have been investigated in \cite{gu,wang2024,wang2023,xu}.

By proving the $H^1$ boundedness of the time derivative $u_t$ of system \eqref{2}, we prove that the global attractor $\A_0$ of the  two-dimensional anisotropic  NS equations is $(H,H^2)$-global attractor, see, e.g. \cite{song}. When $f\in H^2$, we also prove that the  two-dimensional anisotropic  NS equations in fact is $(H,H^2)$-global attractor.  In this paper, when $f\in H$, under certain conditions that the random dynamical system generated by \eqref{1} has  an $(H,H^2)$-random attractor, and, in addition, this random attractor  has finite (box-counting) fractal dimension in $H^2$.  Since the  idea of estimating $u_t$  does not apply to stochastic equations,  the $H^2$-random absorbing sets need be constructed differently here from \cite{song}. when $f\in H^2$, under certain conditions that the random dynamical system generated by \eqref{1} has  an $(H,H^3)$-random attractor, and, in addition, this random attractor  has finite (box-counting) fractal dimension in $H^3$. Moreover, for $f\in H^2$, we prove that the convergence of solutions of \eqref{8.2}  as $\delta\rightarrow0$ in the spaces $(H,H)$, $(H,H^1)$, $(H^1,H^2)$ and $(H^2,H^3)$, respectively.

\subsection*{Main results. Main techniques}

In this subsection, we introduce some main results and main techniques as follows. By virtue of the \cite{crauel}, we reformulated the  stochastic two-dimensional anisotropic  NS equations via a scalar Ornstein-Uhlenbeck process
as the following random equation
\begin{align}\label{3}
\frac{\d \textbf{v}}{\d t}+ \nu A_1\textbf{v}+B\big( \textbf{v}(t)+hy_\delta(\theta_t\omega)\big)=f(\textbf{x})-  \nu A_1hy_\delta(\theta_t\omega)+hy_\delta (\theta_t\omega),
\end{align}
where $A_1\textbf{u}=-P\left(
\begin{array}{ccc}
\partial_{yy}u\\ \partial_{xx}v
\end{array}\right)=-\left(
\begin{array}{ccc}
\partial_{yy}u\\ \partial_{xx}v
\end{array}\right)$, $B(u)=P((u\cdot \nabla)u)$ with $P:\LL^2\to H$ the Helmholtz-Leray orthogonal projection, and   $y_\delta$ is a tempered random variable.  The intensity  function of the noise $h(\textbf{x})$ is supposed to satisfy  the following assumption. In this paper, $C$ is a positive constant whose value may vary from line to line and even in the same line.

\begin{assumption} \label{assum}
Let    $ h \in H^3$, and $\nabla\cdot h=0$ and
 \begin{align*}
  \|\nabla h\|_{L^\infty} <  \sqrt {\pi\delta} \frac{1}{2}\nu \lambda_1 ,
 \end{align*}
where  $\lambda_1 = 4\pi^2 /L^2$ is the first eigenvalue of the Stokes operator $A$.
  \end{assumption}
In order to present the main results in the Section \ref{sec7}, we need to introduce the following assumption.

\begin{assumption} \label{assum1}
Let  $f\in H^2$,  $ h \in H^4$ and $\nabla\cdot h=0$ and
 \begin{align*}
  \|\nabla h\|_{L^\infty} <  \sqrt {\pi\delta} \frac{1}{2}\nu \lambda_1 ,
 \end{align*}
where  $\lambda_1 = 4\pi^2 /L^2$ is the first eigenvalue of the Stokes operator $A$.
  \end{assumption}

In  Theorem \ref{theorem4.1} of the Section \ref{sec4}, we  will introduce the following an $H^2$-random absorbing set.
\begin{main}\label{main}[$H^2$ absorbing set]
 Let Assumption \ref{assum} hold and $f\in H$. The  RDS $\phi$ generated by the random  anisotropic  NS equations \eqref{3} driven by colored noise has a  random absorbing set $\mathfrak{B}_{H^2}$, defined as a random $H^2$ neighborhood of the global attractor $\mathcal{A}_0$ of the deterministic anisotropic  NS equations \eqref{2}:
\begin{align}
\mathfrak{B}_{H^2}(\omega)=\left \{\textbf{v}\in H^2: \, {\rm dist}_{H^2}(\textbf{v},\mathcal{A}_{0}) \leq \sqrt{\rho(\omega)} \, \right\}, \quad \omega\in\Omega, \nonumber
\end{align}
where $\rho(\cdot)$ is the tempered random variable defined by \eqref{rho}.  As a consequence, the $\mathcal{D}_H$-random attractor $\mathcal{A}$ of \eqref{3} is a bounded and tempered random set in $H^2$.

\end{main}

Since the  idea of estimating $u_t$  does not apply to random equations \eqref{3},  the $H^2$-random absorbing sets need be constructed differently, see the Section \ref{sec4}. When $f$ belongs only to $H$, we can not prove the estimation the $H^2$ norm of $\textbf{v}$ of the  random equations \eqref{3}.

In order to prove the Main Theorem \ref{main}, we used a comparison approach to prove the difference $w:=\textbf{v}-\textbf{u}$ between the solutions of the random and the deterministic  anisotropic NS equations instead of estimating the solution $\textbf{v}$ itself, where $\textbf{u}$ is taken as a solution corresponding to an initial value lying  in the global attractor $\A_0$. By Lemma \ref{lemma4.4}, then we obtained that $w$ is asymptotically bounded in $H^2$. Inspired by the Proposition 12.4 in \cite{robinson01} and $ \cup_{t\geq 0} u(t) \subset \A_0$  bounded  in $ H^2$, the $H^2$ random absorbing set of $\textbf{v}$ is proved.

In Section \ref{sec6}, by Theorem \ref{theorem5.1},  we will show that   the random attractor $\A$
  is not only compact but also finite-dimensional in $H^2$, and the attraction  happens in $H^2$ as well. The second  Main Theorem is stated as the following.
\begin{main}\label{main1}
Let Assumption \ref{assum} hold and $f\in H$. Then the  RDS $\phi$ generated by the the random anisotropic  NS equations \eqref{3} driven by colored noise  has  a tempered $(H,H^2)$-random attractor $\mathcal A $. In addition, $\mathcal A$ has finite fractal dimension in $H^2$.
\end{main}
By the standard bi-spatial random attractor theory \cite{cui18jdde} and the $(H,H^2)$-smoothing  property of the system \eqref{3}, we proved the Main Theorem \ref{main1}. The smoothing property is essentially a local $(H,H^2)$-Lipschitz continuity in initial values. Moreover, we obtained the finite-dimensionality of the attractor in $H^2$.  Therefore,  in Theorem \ref{theorem5.6}, we obtained an $(H,H^2)$-smoothing property as stated below.
\begin{main}[$(H, H^2)$-smoothing]
 Let Assumption \ref{assum} hold and $f\in H$. For   any tempered set $\mathfrak D \in \D_H$, there
  are  random variables $T_{{\mathfrak D}} (\cdot)  $   and $ L_{\mathfrak D}(\cdot )$ such that  any two solutions $\textbf{v}_1$ and $\textbf{v}_2$ of the random anisotropic  NS equations \eqref{3} driven by colored noise corresponding to initial values   $\textbf{v}_{1,0},$ $ \textbf{v}_{2,0}$ in $\mathfrak D \left (\theta_{-T_{\mathfrak D}(\omega)}\omega \right)$, respectively, satisfy
  \begin{align}
  &
  \big\|\textbf{v}_1 \!  \big (T_{\mathfrak D}(\omega),\theta_{-T_{\mathfrak D}(\omega)}\omega,\textbf{v}_{1,0}\big)
  -\textbf{v}_2 \big (T_{\mathfrak D}(\omega),\theta_{-T_{\mathfrak D}(\omega)}\omega,\textbf{v}_{2,0} \big) \big \|{^2_{H^2}}  \notag \\[0.8ex]
&\quad \leq  L_{\mathfrak D} (\omega)\|\textbf{v}_{1,0}-\textbf{v}_{2,0}\|^2,\quad \omega\in \Omega.\label{sep5.1}
\end{align}

\end{main}

In Section \ref{sec7}, by Theorem \ref{theorem7.2},  we showed that   the random attractor $\A$
  is not only compact but also finite-dimensional in $H^3$, and the attraction  happens in $H^3$ as well. The fourth  Main Theorem is stated as the following.
\begin{main}\label{main2}
Let Assumption \ref{assum1} hold and $f\in H^2$. Then the  RDS $\phi$ generated by the the random anisotropic  NS equations \eqref{3} driven by colored noise  has  a tempered $(H,H^3)$-random attractor $\mathcal A $. In addition, $\mathcal A$ has finite fractal dimension in $H^3$.
\end{main}

By the $H^3$-absorbing set and the $(H,H^3)$-smoothing  property of the system \eqref{3}, we proved the Main Theorem \ref{main2}. The smoothing property is essentially a local $(H,H^3)$-Lipschitz continuity in initial values. Moreover, we obtained the finite-dimensionality of the attractor in $H^3$.  Therefore,  in Theorem \ref{theorem7.2}, we obtained an $(H,H^3)$-smoothing property as stated below.

\begin{main}[$(H, H^3)$-smoothing]
 Let Assumption \ref{assum1} hold and $f\in H^2$. For   any tempered set $\mathfrak D \in \D_H$, there
  are  random variables $T_{{\mathfrak D}}^* (\cdot)  $   and $ L_{\mathfrak D}^*(\cdot )$ such that  any two solutions $\textbf{v}_1$ and $\textbf{v}_2$ of the random anisotropic  NS equations \eqref{3} driven by colored noise corresponding to initial values   $\textbf{v}_{1,0},$ $ \textbf{v}_{2,0}$ in $\mathfrak D \left (\theta_{-T_{\mathfrak D}^*(\omega)}\omega \right)$, respectively, satisfy
  \begin{align}
  &
  \big\|\textbf{v}_1 \!  \big (T_{\mathfrak D}^*(\omega),\theta_{-T_{\mathfrak D}^*(\omega)}\omega,\textbf{v}_{1,0}\big)
  -\textbf{v}_2 \big (T_{\mathfrak D}^*(\omega),\theta_{-T_{\mathfrak D}^*(\omega)}\omega,\textbf{v}_{2,0} \big) \big \|{^2_{H^3}}  \notag \\[0.8ex]
&\quad \leq  L_{\mathfrak D}^* (\omega)\|\textbf{v}_{1,0}-\textbf{v}_{2,0}\|^2,\quad \omega\in \Omega.\label{sep5.1}
\end{align}
\end{main}

In Section \ref{sec8}, we firstly proved the upper semi-continuity of random attractors as $\delta\rightarrow0$. Finally, we proved the convergence of solutions of \eqref{8.2}  as $\delta\rightarrow0$ in the spaces $(H,H)$, $(H,H^1)$, $(H^1,H^2)$ and $(H^2,H^3)$, respectively. Then we introduce the main result as stated below.

\begin{main} Let Assumption \ref{assum1} hold and $f\in H^2$. There  exists a  random variable $\bar{T}_\omega^*$  such that   any two solutions $\textbf{v}_1$ and $\textbf{v}_2$ of the random anisotropic  NS equations \eqref{3} driven by colored noise corresponding to initial values  $\textbf{v}_{\delta,0},$ $\textbf{v}_{0}$ in $\mathfrak{B}(\theta_{-\bar{T}_\omega^*})$, respectively,   satisfy
\begin{align*}
 \lim\limits_{\delta\rightarrow0}\|\textbf{u}_\delta(\bar{T}_\omega^*,\theta_{-\bar{T}_\omega^*}\omega,\textbf{u}_{\delta,0})
 -\textbf{u}(\bar{T}_\omega^*,\theta_{-\bar{T}_\omega^*}\omega,\textbf{u}_{0}) \big \|^2_{H^3}=0, \quad \omega\in \Omega.
\end{align*}
\end{main}

\subsection*{Organization of the paper}

In Section \ref{sec2}, we introduce the  $(X,Y)$-random attractor theory and their fractal dimensions.  In Section \ref{sec3}, we first  introduce the  functional spaces that are necessary to study two-dimensional anisotropic  NS equations, and then associate an RDS to the random two-dimensional anisotropic  NS equations. Under  new   conditions, we prove an $(H,H^1)$-random attractor of finite  fractal dimension in $H$.    In Section \ref{sec4}, by an indirect approach of  estimating the $H^2$-distance between the random and the deterministic solution trajectories  we construct an  $H^2$ random absorbing set. In Section \ref{sec5}, we  show the crucial  $(H,H^2)$-smoothing  of the RDS from which it follows in Section \ref{sec6}  that the random attractor    in $H$ is in fact an $(H,H^2)$-random attractor of finite  fractal dimension in $H^2$.  In Section \ref{sec7}, we  first show the   $(H,H^3)$-smoothing  of the RDS,  and then the random attractor    in $H$ is in fact an $(H,H^3)$-random attractor of finite  fractal dimension in $H^3$. In Section \ref{sec8}, we prove the  upper semi-continuity of random attractors  and the convergence of solutions of \eqref{8.2}  as $\delta\rightarrow0$ in the spaces $(H,H)$, $(H,H^1)$, $(H^1,H^2)$ and $(H^2,H^3)$, respectively.

\section{Preliminaries:   $(X,Y)$-random attractors and  their fractal dimension}\label{sec2}

In this section, we first introduce the basic  theory of $(X,Y)$-random attractors and their fractal dimension needed in this paper.   The readers are referred to \cite{arnold,crauel,cui18jdde,gu2020,gu2019,langa2} and references therein.

Let $(\Omega,\mathcal{F},\mathbb{P})$ be the standard probability space, where $\Omega=\{\omega\in C(\mathbb{R},\mathbb{R}):\omega(0)=0\}$, with the open compact topology, $\mathcal{F}$ is  its Borel $\sigma$-algebra and $\mathbb{P}$ is  the Wiener measure on $(\Omega,\mathcal{F})$.  We consider  the Wiener shift $\{\theta_t\}_{t\in\mathbb{R}}$ defined on the probability space $(\Omega,\mathcal{F},\mathbb{P})$ as  stated below
\begin{align*}
\theta_t\omega(\cdot)=\omega(t+\cdot)-\omega(t)
\end{align*}
for any $\omega\in\Omega$ and $t\in\mathbb{R}$.

Let $(X,\|\cdot\|_X)$ denote a Banach space  and   be  a probability space with a group $\{\theta_t\}_{t\in\mathbb{R}}$ of measure-preserving self-transformations on $\Omega$.
Denote by $\mathcal{B}(\cdot)$ the Borel $\sigma$-algebra of a metric space, and let
   $\R^+:=[0,\infty)$.

\begin{definition}
A mapping $\phi:\mathbb{R}^+\times\Omega\times X\rightarrow X$ is defined  a {\em random dynamical system (RDS)}   in $X$, if
\begin{enumerate}[(i)]
\item $\phi$ is $(\mathcal{B}(\mathbb{R}^+)\times\mathcal{F}\times\mathcal{B}(X),\mathcal{B}(X))$-measurable;

\item  for all $\omega\in\Omega$, $\phi(0,\omega,\cdot)$ is the identity on $X$;

\item $\phi$ satisfies  the cocycle property as  stated below
\begin{align*}
\phi(t+s,\omega,x)=\phi(t,\theta_s\omega,\phi(s,\omega,x)), \quad \forall t,s\in \mathbb{R}^+,~\omega\in\Omega,~x\in X;
\end{align*}
\item the mapping $x\mapsto \phi(t,\omega, x)$ is continuous.
\end{enumerate}
\end{definition}

A  set-valued mapping $\mathfrak D$: $\Omega\mapsto 2^X\setminus \emptyset $, $ \omega\mapsto \mathfrak D(\omega) $, is called a \emph{random set}  in $X$ if it is   \emph{measurable} in the sense that the mapping $\omega\to \dist_X(x, \mathfrak D(\omega))$    is  $(\mathcal F,\mathcal B(\R))$-measurable for each $x\in X$. If each its image $\mathfrak D(\omega)$ is closed (or  bounded, compact, etc.) in $X$, then $\mathfrak D$ is called a  closed (or  bounded, compact, etc.) random set in $X$.

A random variable $\zeta:\Omega\mapsto X$ is called {\em tempered} if
\[
  \lim_{t\to \infty} e^{-\varepsilon t} |\zeta(\theta_{-t}\omega)| =0, \quad \varepsilon >0.
\]
A random set $\mathfrak D$ in $X$  is called {\emph{tempered}}  if there exists a   tempered random variable $\zeta$ such that $\|\mathfrak D(\omega)\|_X:=\sup_{x\in \mathfrak D(\omega)} \|x\| \leq \zeta(\omega)$ for any $\omega\in \Omega$.

\smallskip
Denote by $\mathcal{D}_X$  the collection of all the tempered random sets in $X$.

 \begin{definition}
A    random set $\mathfrak{B}$ in $X$ is called a {\emph{(random)  absorbing set}} for $\phi$ if  it pullback absorbs any tempered random set in $X$, i.e., if for every  $\mathfrak D\in\mathcal{D}_X$   there exists a random variable $T_{\mathfrak D} (\cdot)$ such that
\begin{align*}
\phi(t,\theta_{-t}\omega, \mathfrak D(\theta_{-t}\omega))\subset \mathfrak{B}(\omega),\quad t\geq T_{\mathfrak D}(\omega),\  \omega\in\Omega.
\end{align*}
\end{definition}

\begin{definition}
A random set $\mathcal{A}$  is called  the \emph{random attractor}   for  $\phi$ in $X$ if
\begin{enumerate}[(i)]
\item  $\mathcal{A}$ is a tempered and compact random set in $X$;

 \item  $\mathcal{A}$ is invariant under $\phi$, that is,
\begin{align*}
\mathcal{A}(\theta_t\omega)=\phi(t,\omega,\mathcal{A}(\omega)), \quad t\geq0,~\omega\in\Omega;
\end{align*}

 \item $\mathcal{A}$  pullback attracts any tempered set in $X$, that is,  for all $\mathfrak D\in  \mathcal{D}_X$,
 \[
\lim\limits_{t\rightarrow\infty} {\rm dist}_{X} \big(\phi(t,\theta_{-t}\omega,\mathfrak D(\theta_{-t}\omega)), \, \mathcal{A}(\omega) \big)=0,\quad \omega\in\Omega ,
 \]
where
$\dist _{X}(\cdot, \cdot) $  is  the Hausdorff semi-metric between   subsets  of  $ X$, that is,
\[
 \dist_{X}(A,B)= \sup_{a\in A} \inf_{b\in B} \|a-b\|_X, \quad \forall A,B\subset X.
\]
\end{enumerate}
\end{definition}

For the colored noise, given $\delta>0$, then we introduce one-dimensional stochastic differential equation
\begin{align}\label{4}
\d \zeta_{\delta}+\frac{1}{\delta}\zeta_{\delta}\d t= \frac{1}{\delta}dW,
\end{align}
where $W$ is a two-sided real valued Wiener process defined on $(\Omega,\mathcal{F},\mathbb{P})$. Let $\zeta_{\delta}:\Omega\rightarrow\mathbb{R}$ be a random variable defined by
\begin{align*}
\zeta_{\delta}(\omega)= \frac{1}{\delta}\int_{-\infty}^0e^{ \frac{s}{\delta}}dW,~~\forall \omega\in\Omega.
\end{align*}
By \cite{gu2020,gu2019}, we have that $\zeta_{\delta}(\theta_t\omega)$ is the unique stationary solution of \eqref{4}. Then the process $\zeta_{\delta}(\theta_t\omega)$ is defined a real valued colored noise which is a stationary Wiener process with mean zero and variance $\frac{1}{2\delta}$. Moreover, there exists a $\{\theta_t\}_{t\in\mathbb{R}}$-invariant subset $\Omega$ of full measure, such that for any $\omega\in \Omega$,
 \begin{align}\label{5}
\lim\limits_{t\rightarrow\pm\infty}\frac{1}{t}\int_0^ty_{\delta}(\theta_r\omega)\d r=\mathbb{E}[y_{\delta}]=0~~uniformly~~for~~0<\delta\leq1.
  \end{align}

Assume that $Y$ is a Banach space with continuous embedding $Y \hookrightarrow X$.
\begin{definition}
A random set $\mathcal{A}$   is called  the {\emph{$(X,Y)$-random attractor}}  for $\phi$ in $X$,  if
\begin{enumerate}[(i)]
\item  $\mathcal{A}$ is a tempered and compact random set in $Y$;

 \item  $\mathcal{A}$ is invariant under $\phi$, that is,
\begin{align*}
\mathcal{A}(\theta_t\omega)=\phi(t,\omega,\mathcal{A}(\omega)), \quad t\geq0,~\omega\in\Omega;
\end{align*}

 \item $\mathcal{A}$  pullback attracts any tempered set in $X$ in the metric of $Y$, that is,  for all $\mathfrak D\in  \mathcal{D}_X$,
 \[
\lim\limits_{t\rightarrow\infty} {\rm dist}_{Y} \big(\phi(t,\theta_{-t}\omega, \mathfrak D(\theta_{-t}\omega)), \, \mathcal{A}(\omega) \big)=0,\quad \omega\in\Omega .
 \]
\end{enumerate}
\end{definition}

Note that  in addition to the temperedness, compactness and pullback attraction in $Y$,  the measurability in $Y$ of an $(X,Y)$-random attractor was also  required by definition.

 Then we introduce  the following criterion  as  stated below, see, e.g. \cite[Theorem 19]{cui18jdde}.
  \begin{lemma}(\cite[Theorem 19]{cui18jdde}) \label{lem:cui18}
 Suppose that $\phi$ is  an RDS.  If
  \begin{enumerate}[(i)]
  \item  $\phi$ has a   random absorbing set $\mathfrak B$ which is a tempered and closed random set in $Y$;
  \item $\phi$  is $(X,Y)$-asymptotically compact, that is, for any $ \mathfrak D\in \D_X$ and $t_n\to \infty$,   any sequence
  $\{\phi(t_n, \theta_{-t_n} \omega, x_n)\}_{n\in \N}$ with $x_n \in \mathfrak D(\theta_{-t_n} \omega)$, $\omega\in \Omega$, has a convergent subsequence in $Y$,
  \end{enumerate}
    then  $\phi$ has a unique   $(X,Y)$-random attractor $\A$ defined by
    \[
    \A(\omega) = \bigcap_{s\geq 0} \overline{ \bigcup_{t\geq  s} \phi(t,\theta_{-t} \omega, \mathfrak  B(\theta_{-t}\omega)) }^Y ,
    \quad \omega \in \Omega.
    \]
  \end{lemma}

 \medskip

 Let $E$ be a precompact  set in $X$, and suppose that $N(E, \varepsilon)$ is the minimum number of balls of radius $\varepsilon$ in $X$ required to cover $E$. Then the {\emph{(box-counting) fractal dimension of $E$}}  is given by
\[
 d_f^X(E)=\limsup_{\varepsilon\to 0} \frac{ \log N(E,\varepsilon)}{ -\log \varepsilon} ,
\]
where the superscript ``$\ ^X$''   indicates the space to which  the fractal dimension is referred.

\section{The anisotropic Navier-Stokes equations and the   attractors in $H$}\label{sec3}
In this section, we first introduce the basic of functional spaces. As $f$ belongs only to $H$, we can prove that the global attractor of the deterministic two-dimensional anisotropic  NS equations \eqref{2} is $(H,H^2)$-attractor. When $f\in H^2$, we also can prove that the global attractor of the deterministic two-dimensional anisotropic  NS equations \eqref{2} is $(H,H^3)$-attractor. We transform the stochastic anisotropic  NS equations \eqref{3.2} to a random  anisotropic  NS equations  which generate an RDS $\phi$ on $H$. Finally, under  new   conditions, we will prove an $(H,H^1)$-random attractor of finite  fractal dimension in $H$.

\subsection{Functional spaces}

In this subsection, we  first introduce Sobolev spaces of periodic functions to study the periodic boundary conditions. The following settings are standard, the readers are referred to, e.g., \cite{bahouri,robinson01,temam1,temam}. etc.

 Let $(C_{per}^\infty(\mathbb T^2))^2 $ be the space of two-component infinitely differentiable functions that are $L$-periodic in each direction, and by  $\LL^2$, $\HH^1$ and $\HH^2$ the completion of $(C_{per}^\infty(\mathbb T^2))^2 $  with respect to the $(L^2(\TT^2))^2$, $(H^1(\TT^2))^2$ and $(H^2(\TT^2))^2$ norm, respectively.  $\LL^2$ and $\HH^1$ are endowed with the inner products as follows
 \begin{align*}
 (u,v)&=\int_{\mathbb{T}^2}u\cdot v\d x,~~u,v\in\LL^2,\\
  ((u,v))&=\int_{\mathbb{T}^2}\sum\limits_{j=1}^2\nabla u_j\cdot \nabla v_j\d x,~~u,v\in\HH^1.
 \end{align*}

 The phase space $(H,\|\cdot\|)$ denotes a subspace of  $\LL^2$ of functions with zero mean and free divergence,  that is,
\begin{align*}
H= \left \{u\in \LL ^2:
\, \int_{\mathbb T^2}u \ \d x=0,\  {\rm div}\, u=0 \right \},
\end{align*}
endowed with the norm of $\LL^2$, i.e.,
  $\|\cdot\|=\|\cdot\|_{\LL^2}$.   For $p>2$,  we write $(L^p(\TT^2))^2$  simply as $L^p$ and its norm as $\|\cdot\|_{L^p}$.

Note that any $u\in \LL^2 $ has an expression $u=\sum_{j\in \mathbb{Z}^2} \hat u_j e^{{\rm i}j\cdot x}$, where  $ {\rm i}= \sqrt{-1}$ and $\hat u_j$ are Fourier coefficients.  Hence, $\int_{\mathbb T^2}    u \,  \d x=0$ is equivalent to $\hat u_0=0$. For any  $u\in H $, we have
\[
 u =\sum_{j\in \mathbb Z^2\setminus\{0\}} \hat u_j e^{{\rm i}j\cdot x} .
\]
Suppose that $P: \LL ^2\rightarrow H$ is the Helmholtz-Leray orthogonal projection operator. In  this periodic space, we define the Stokes operator $A\textbf{u}=-P\Delta\textbf{u}==-\Delta\textbf{u}$ for any $\textbf{u}\in D(A) $. Then we set  $A_1\textbf{u}=-P\left(
\begin{array}{ccc}
\partial_{yy}u\\ \partial_{xx}v
\end{array}\right)=-\left(
\begin{array}{ccc}
\partial_{yy}u\\ \partial_{xx}v
\end{array}\right)$. In  this bounded domain, the projector $P$ does not commute with derivatives. The operator $A^{-1}$  is a self-adjoint positive-definite compact operator from $H$  to $H$.

   For $s >0 $,    the Sobolev space  $H^s:=D(A^{s/2})$  is given  by as stated below
 \[
H^s =\left \{u\in H:\|u\|^2_{H^s}=\sum\limits_{j\in \mathbb{Z}^{2}\backslash \{0\}}|j|^{2s}|\hat{u}_j|^2<\infty \right \},
 \]
endowed with the norm $\|\cdot\|_{H^s}=\|A^{s/2}\cdot \|$. Then the dual space of $H^s$ is defined by $H^{-s}$.

For any $u,v\in H^1$,  we define the bilinear form
\[
B(u,v)=P((u\cdot\nabla)v) ,
\]
and, in particular, $B(u):=B(u,u)$. For any  $u,v,w\in H^1$, we set
\begin{align*}
\langle B(u,v),w\rangle=b(u,v,w)=\sum\limits_{i,j=1}^2\int_{\mathbb{T}^2}u_i\frac{\partial v_j}{\partial x_i}w_j\d x.
\end{align*}
 Then we introduce the Poincar\'e's inequality as stated below, see, e.g., \cite[lemma 5.40]{robinson01}.

\begin{lemma}  If $u\in H^1$, then we have
\begin{equation} \label{poin}
 \|u\| \leq \left( \frac L{2\pi} \right) \|    u\|_{H^1}.
\end{equation}
\end{lemma}
Thus, let
\[
 \lambda_1 :=   \frac{4\pi^2 } {L^2}  ,
\]
we can deduce that $\| u\|_{H^1}^2 \geq\lambda_1 \|u\|^2$. For later purpose we  let $\lambda:= \alpha \nu \lambda_1/4$.

Assume that  the noise  intensity $h$  satisfy Assumption  \ref{assum}, that is,     $ h \in H^3$ and
\begin{align}\label{2.3}
  \|\nabla h\|_{L^\infty} <   \sqrt{\pi\delta}\frac{1}{2} \nu \lambda_1 .
\end{align}
 (In particular, $\|\nabla h\|_{L^\infty} < \sqrt{\pi\delta} \frac{\nu}{2} $  for the scale $L=2\pi$).
Clearly, the  tolerance of the noise  intensity     increases as the kinematic viscosity $\nu$ increases.

Let $f\in H$, $\textbf{u}_{h}(0)\in H$ and $\omega\in\Omega$, we say that a mapping $\textbf{u}(\cdot,\omega,\textbf{u}_h(0)):[0,+\infty)\rightarrow H$ is defined  a solution of system \eqref{1} if for any $T>0$,
\begin{align*}
\textbf{u}(\cdot,\omega,\textbf{u}_h(0))\in C([0,+\infty);H)\cap L^2(0,T;H^1)
\end{align*}
and $u$ satisfies
\begin{align}\label{3.1}
(\textbf{u}_h(t),\xi)&+\nu\int_0^t(\partial_yu_h,\partial_y\xi_1)\d s+\nu\int_0^t(\partial_xv_h,\partial_x\xi_2)\d s+\int_0^tb(\textbf{u}_h,\textbf{u}_h,\xi)\d s\nonumber\\
&=(\textbf{u}_0,\xi)+\int_0^t( f(\textbf{x}),\xi)\d s+\int_0^t(h(\textbf{x})\zeta_{\delta}(\theta_s\omega),\xi)\d s
\end{align}
for any $t>0$ and $\xi=(\xi_1,\xi_2)\in H^1$. Then the system \eqref{1} is rewritten as follows
\begin{align}\label{3.2}
\frac{\d \textbf{u}_h}{\d t}+\nu
A_1\textbf{u}_h+B(\textbf{u}_h,\textbf{u}_h)=f+h\zeta_{\delta}(\theta_t\omega),
\end{align}
in $H^{-1}$.In \cite{temam1,temam}, applying the Galerkin method, we say that for any $\textbf{u}_h(0)\in H$ and $\omega\in\Omega$, the system \eqref{1} has a unique solution $\textbf{u}_h$ in the sense of \eqref{3.1}.

\subsection{Global attractor of the deterministic anisotropic Navier-Stokes equations}

In this subsection, we recall the global attractor of the deterministic two-dimensional anisotropic  NS equations \eqref{2}, i.e. for $h(\textbf{x})\equiv 0$,  which will be crucial for our analysis later. Under projection $P$, the      deterministic two-dimensional anisotropic  NS equations \eqref{2}  is rewritten as
 \begin{align}\label{2.1}
   \partial_t\textbf{u}+ \nu A_1\textbf{u}+B(\textbf{u})  =f,\quad u(0)=u_{0}.
\end{align}
The system \eqref{2.1}  generates an autonomous dynamical system $S$ in phase space $H$ with a global attractor $\A_0$. Inspired by the Proposition 12.4 in \cite{robinson01},    the following regularity result of the attractor $\A_0$ seems  optimal for $f\in H$  in the literature.

\begin{lemma}
For $f\in H$, the deterministic two-dimensional anisotropic  NS equations \eqref{2.1}  generate a semigroup $S$ which possesses a global attractor $\A_0$ in $H$ and has an
  absorbing set which is bounded in $H^2$. In particular,    the global attractor $\A_0$   is   bounded in $H^2$.
\end{lemma}

\begin{lemma}\label{lem:det}
For $f\in H^2$, the deterministic two-dimensional anisotropic  NS equations \eqref{2.1}  generate a semigroup $S$ which possesses a global attractor $\A_0$ in $H$ and has an
  absorbing set which is bounded in $H^3$. In particular,    the global attractor $\A_0$   is   bounded in $H^3$.
\end{lemma}
For $f\in H$, we will show that the global attractor $\mathcal A$ is not only bounded in $H^2$, but also has   finite fractal dimension (and thus compact) in $H^2$. Moreover, For $f\in H^2$, we also will show that the global attractor $\mathcal A$ is not only bounded in $H^3$, but also has   finite fractal dimension in $H^3$.

\begin{remark}  \label{rmk1}
 When $f$ belongs only to $H$, we cannot deduce an energy equation of  $\|Au\|$ directly and thus Lemma \ref{lem:det} was proved by estimating the time-derivatives $ \partial_t \textbf{u} $  of solutions    rather than estimating the solutions  $\textbf{u}$ themselves.       However, this   method  does not apply to the  random two-dimensional anisotropic  NS equations \eqref{2.2}  since the colored noise is not derivable in time.  Therefore, in Section \ref{sec4} we will employ a comparison approach  to prove $H^2$ random absorbing sets for \eqref{2.2}.  When $f\in H^2$, we can deduce an energy equation of  $\|A^\frac32u\|$ directly. And then this   method apply to the  random two-dimensional anisotropic  NS equations \eqref{2.2} to prove random absorbing set in $H^3$.
\end{remark}

\subsection{Generation of an RDS}\label{sec2.1}

For the colored noise $\zeta_{\delta}(\theta_t\omega)$ and the Wiener process $W(t,\omega)=\omega(t)$, we recall some main results as follows,  see, e.g., \cite{gu2020,gu2018}
\begin{lemma} \label{lemma3.4}
 Assume the correlation time $\delta\in (0,1]$. Then there is a $\{\theta_t\}_{t\in \mathbb{R}}$ invariant subset $\Omega$ of full measure,such that for $\omega\in\Omega$,

(i)\begin{align*}
\lim\limits_{t\rightarrow\pm\infty}\frac{\omega(t)}{t}=0;
\end{align*}
(ii) the mapping
\begin{align*}
(t,\omega)\mapsto\zeta_{\delta}(\theta_t\omega)=-\frac{1}{\delta^2}\int^0_{-\infty}e^{\frac{s}{\delta}}\theta_t\omega(s)\d s
\end{align*}
denotes a stationary solution (also called  a colored noise) of the one-dimensional stochastic differential equation
\begin{align*}
\d \zeta_{\delta}+\frac{1}{\delta}\zeta_{\delta}\d t= \frac{1}{\delta}dW
\end{align*}
with continuous trajectories satisfying
\begin{align}\label{3.3}
\lim\limits_{t\rightarrow\pm\infty}\frac{|\zeta_{\delta}(\theta_t\omega)|}{t}=0,~~for~~any~~0<\delta\leq1.
\end{align}
\end{lemma}
 In order to transform \eqref{3.2},   we introduce the following random equation driven by colored noise,
  \begin{align}\label{3.8}
\frac{\d y_{\delta}}{\d t}=- y_{\delta}+\zeta_{\delta}(\theta_t\omega).
 \end{align}
By virtue of \eqref{3.3}, we can get that for any $\omega\in\Omega$, the integral
  \begin{align}\label{3.7}
y_{\delta}(\omega)=\int_{-\infty}^0e^{ s}\zeta_{\delta}(\theta_s\omega)\d s
  \end{align}
is convergent. Meanwhile, $y_{\delta}(\omega)$ denotes a well defined random variable. By \cite{crauel97jdde,crauel}, we have that $z(\theta_t\omega)$ is the stationary solution of the one-dimensional Ornstein-Uhlenbeck. Then there exists a $\theta_t $-invariant  subset $\tilde \Omega\subset \Omega$ of full measure such that $z(\theta_t\omega)$ is continuous in $t$  for any $\omega \in \tilde\Omega$ and the random variable $|z(\cdot )|$ is tempered,

By virtue of the Lemma 3.2 in \cite{gu2020} and \eqref{5}, we introduce the following main proposition.
\begin{proposition}\label{proposition3.5}
Assume that $y_{\delta}$ is the random variable as \eqref{3.7}, we have that the mapping
 \begin{align}
(t,\omega)\mapsto y_{\delta}(\theta_t\omega)=e^{-t}\int_{-\infty}^te^{ s}\zeta_{\delta}(\theta_s\omega)\d s
  \end{align}
is a stationary solution of \eqref{3.8} with continuous trajectories. Moreover, we have $\mathbb{E}(y_{\delta})=0$ and for any $\omega\in\Omega$ and $T>0$,
 \begin{align}
 \lim\limits_{\delta\rightarrow0}y_{\delta}(\theta_t\omega)=z(\theta_t\omega)~~uniformly~~on~~[0,T],\label{2.4}\\
\lim\limits_{t\rightarrow\infty}\frac{|y_{\delta}(\theta_t\omega)|}{|t|}=0~~uniformly~~for~~0<\delta\leq\frac{1}{2},\\
\lim\limits_{t\rightarrow\infty}\frac{1}{t}\int_0^ty_{\delta}(\theta_r\omega)\d r=0~~uniformly~~for~~0<\delta\leq\frac{1}{2},\\
\lim\limits_{\delta\rightarrow0}\mathbb{E}(|y_{\delta}(\omega)|)=\mathbb{E}(|z(\omega)|)\label{2.5}.
  \end{align}

\end{proposition}

In  \cite{arnold,zhu}, for any given $\delta>0$, then we get
\begin{align} \label{erg}
\lim\limits_{t\rightarrow \pm \infty}\frac{1}{t}\int_0^t|y_{\delta}(\theta_s\omega)|^m\, \d s=\mathbb{E}  \big (|y_{\delta}(\theta_t\omega)|^m
\big) =\frac{\Gamma  (\frac{1+m}{2} )}{\sqrt{\pi\delta^m}}
,\quad \omega\in \tilde \Omega,
\end{align}
for  all $m\geq 1$, where $\Gamma$ is a Gamma function.  Hence, we will not distinguish $\tilde \Omega$ and $\Omega$.

Then we introduce the following a new variable
 \begin{align}\label{3.4}
\textbf{u}_h(t)=\textbf{v}(t)+hy_\delta(\theta_t\omega),\quad t>0,\ \omega\in \Omega,
\end{align}
here, $\textbf{v}(t)=(v_1,v_2)$ and $0=\nabla\cdot\textbf{u}_h=\nabla\cdot(\textbf{v}+hy_\delta(\theta_t\omega))=\nabla\cdot\textbf{v}$. Then we have
\begin{align}\label{3.9}
\|\nabla \textbf{v}\|^2=\|\nabla\times \textbf{v}\|^2\leq 2\|(\partial_yv_1,\partial_xv_2)\|^2.
\end{align}
 By \eqref{3.2} and \eqref{3.4}, we have the following   abstract evolution equations
\begin{align}\label{2.2}
\left\{
\begin{aligned}
&
\frac{\d \textbf{v}}{\d t}+ \nu A_1\textbf{v}+B\big( \textbf{v}(t)+hy_\delta(\theta_t\omega)\big)=f(\textbf{x})-  \nu A_1hy_\delta(\theta_t\omega)+hy_\delta (\theta_t\omega), \\
& \textbf{v}(0)=\textbf{u}_h(0)-hy_\delta (\omega).
\end{aligned}
\right.
\end{align}

By the classic Galerkin method in \cite{arnold,temam1}, the system \eqref{2.2} has a unique weak solution as stated below.
\begin{lemma}
Let $f\in H$ and Assumption \ref{assum} hold.  For any $\textbf{v}(0)\in H$ and $\omega\in\Omega$, then there exists a unique weak solution
\begin{align*}
\textbf{v} (t) \in C_{loc}([0,\infty);H)\cap L^2_{loc} (0,\infty;H^1 )
\end{align*}
satisfying \eqref{2.2} in distribution sense with $\textbf{v}|_{t=0}=\textbf{v}(0)$.  In addition, this solution is continuous in initial data $\textbf{v}(0)$ in $H$.
\end{lemma}

Therefore, by
\[
 \phi(t,\omega, v_0) =\textbf{v}(t,\omega,v_0),\quad  \forall t\geq 0,\, \omega\in \Omega, \, v_0\in H,
\]
where $\textbf{v}$ is the solution of \eqref{2.2},
we associated an RDS $\phi$ in $H$ to \eqref{2.2}.

\subsection{The random attractor in $H$}

 In this subsection, we first will introduce the following two    basic lemmas,  see, e.g., \cite{temam1,temam}.

\begin{lemma}[{Gronwall's lemma}]
Assume that $x(t)$ is  a function from $\R$ to $\R^+$ such that
\[
\dot x +a(t)x\leq b(t) .
\]
 Then for any $t\geq s $,
\[
 x(t)\leq e^{ -\int^t_s a(\tau) \, \d \tau} x(s)+\int_s^t e^{-\int^t_\eta a(\tau) \, \d \tau } b(\eta) \, \d \eta.
\]
\end{lemma}

\begin{lemma}[Gagliardo-Nirenberg inequality]  If $u\in L^q(\mathbb{T}^2)$, $D^{m}u\in L^r(\mathbb{T}^2)$, $1\leq q,r\leq\infty$, then there exists a positive constant $C$ such that
\[
 \|D^ju\|_{L^p} \leq  C \|D^mu\|_{L^r}^a\|u\|^{1-a}_{L^q},
\]
where
\begin{equation*}
\frac{1}{p}-\frac{j}{2}=a \left (\frac{1}{r}-\frac{m}{2} \right)+ \frac{1-a}{q}, \quad 1\leq p\leq\infty,~0\leq j\leq m,~\frac{j}{m}\leq a\leq1,
\end{equation*}  and
$C$ depends only on $\{m, \, j, \, a, \, q, \,r \}$.
\end{lemma}

\begin{lemma}[$H$ bound]\label{lemma4.1}
Let  $f\in H$ and Assumption \ref{assum} hold.
There exist random variables  $T_1(\omega)$ and $\zeta_1(\omega)$  such that any
  solution $\textbf{v}$  of the system \eqref{2.2} corresponding to initial  values  $\textbf{v}(0) \in H $ satisfies
  \begin{align}
 \|\textbf{v}(t,\theta_{-t} \omega, \textbf{v}(0))\|^2
 \leq e^{-\lambda t}\|\textbf{v}(0)\|^2
 + \zeta_1(\omega) ,\quad t\geq T_1(\omega).  \nonumber
\end{align}
\end{lemma}
\begin{proof}
By \eqref{2.3}, $  {\|\nabla h\|_{L^\infty} }/ \sqrt{\pi\delta} < \frac{\nu  \lambda_1}{2}$, then there exist  $  \alpha \in  (0,\frac{1}{2}]$ and $\beta>0$  such that
\begin{gather}
  \frac{\|\nabla h\|_{L^\infty} }{\sqrt {\pi\delta}}  =  (\frac{1}{2}- \alpha) \nu  \lambda_1,  \label{c1} \\
   \frac{  \|\nabla h\|_{L^\infty}  }{ \sqrt{\pi\delta}}  \left( 1 +  \beta \right )  =   (\frac{1}{2}-\frac  1 2 \alpha)\nu \lambda_1  .  \label{c2}
  \end{gather}
Taking the inner product of the first equation   \eqref{2.2} with $\textbf{v}$ in $H$, by integration by parts yields
\begin{align}
\frac{\d}{\d t}\|\textbf{v}\|^2+ 2\nu  \|\partial_yv_1\|^2+2\nu  \|\partial_xv_2\|^2
 &=-2(B(\textbf{v}+hy_\delta(\theta_t\omega)) ,\textbf{v})\nonumber\\
 &+ 2\big (f- \nu A_1hy_\delta(\theta_t\omega) +hy_\delta(\theta_t\omega),\textbf{v}\big),
\end{align}
by \eqref{3.9} and Poincar\'e's inequality \eqref{poin}, it yields
\begin{align}
 & \frac{\d}{\d t}\|\textbf{v}\|^2+ \left(1-\frac \alpha 2\right)  \nu  \lambda_1  \| \textbf{v}\|^2  + \frac {\alpha \nu } 2   \|A^{\frac12}\textbf{v}\|^2 \nonumber  \\
 &\quad \leq \frac{\d}{\d t}\|\textbf{v}\|^2  + \nu  \|A^{\frac12}\textbf{v}\|^2 \nonumber\\
&\quad  \leq 2 \big |\big(B(\textbf{v}+hy_\delta(\theta_t\omega)) ,\textbf{v }\big) \big|
+2\big| \big (f- \nu A_1hy_\delta(\theta_t\omega) +hy_\delta(\theta_t\omega),\textbf{v} \big)\big|  . \label{4.3}
\end{align}
Applying  the divergence free conditions and the H\"{o}lder inequality, Young's inequality and Poincar\'{e}'s inequality, we can deduce that
\begin{align}\label{3.5}
   2\big | \big(B(\textbf{v}+hy_\delta(\theta_t\omega)) ,\textbf{v} \big) \big|
  &= 2\big | \big(B(\textbf{v}+hy_\delta(\theta_t\omega),hy_\delta(\theta_t\omega)) ,\textbf{v} \big) \big| \nonumber\\
   & \leq  2 \big | \big(B(\textbf{v},hy_\delta(\theta_t\omega)),\textbf{v} \big) \big| +2 \big| \big (B(hy_\delta(\theta_t\omega),hy_\delta(\theta_t\omega)),\textbf{v} \big) \big| \nonumber\\
&\leq 2 |y_\delta(\theta_t\omega)|\int_{\mathbb{T}^2}|\textbf{v}|^2|\nabla h|\ \d x+ 2|y_\delta(\theta_t\omega)|^2\int_{\mathbb{T}^2}|\textbf{v}||h||\nabla h|\ \d x\nonumber\\
&\leq 2 |y_\delta(\theta_t\omega)|\|\nabla h\|_{L^\infty}\|\textbf{v}\|^2+ 2|y_\delta(\theta_t\omega)|^2\|\nabla h\|_{L^\infty}\|\textbf{v}\|\|h\|\nonumber\\
&\leq 2 |y_\delta(\theta_t\omega)|\|\nabla h\|_{L^\infty}\|\textbf{v}\|^2+  \frac{ \alpha\nu \lambda_1} {8} \|\textbf{v}\|^2+ C|y_\delta(\theta_t\omega)|^4\|\nabla h\|_{L^\infty}^2\|h\|^2 ,
\end{align}
and
\begin{align}\label{3.6}
 2\big| \big (f-\nu A_1hy_\delta(\theta_t\omega) +hy_\delta(\theta_t\omega),\textbf{v} \big)\big|  \leq \frac{ \alpha\nu \lambda_1}{8}\|\textbf{v}\|^2+C\|f\|^2
 + C|y_\delta(\theta_t\omega)|^2 \left(\|Ah\|^2 + \|h\|^2\right) .
\end{align}
Inserting \eqref{3.5}  and \eqref{3.6} into \eqref{4.3} yields
\ben
&
\frac{\d}{\d t}\|\textbf{v}\|^2+ \left(1-\frac {3 \alpha} 4\right) \nu \lambda_1  \| \textbf{v}\|^2  + \frac {\alpha \nu }  2 \|A^{\frac12}\textbf{v}\|^2  \\
 &\quad \leq 2|y_\delta(\theta_t\omega)|\|\nabla h\|_{L^\infty}\|\textbf{v}\|^2+ C |y_\delta(\theta_t\omega)|^4\|\nabla h\|_{L^\infty}^2\|h\|^2\nonumber\\
&\qquad + C\|f\|^2+ C |y_\delta(\theta_t\omega)|^2 \left (\|h\|^2+\|Ah\|^2 \right)\nonumber\\
&\quad \leq 2|y_\delta(\theta_t\omega)|\|\nabla h\|_{L^\infty}\|\textbf{v}\|^2+C \left (1+|y_\delta(\theta_t\omega)|^4 \right),
\ee
and then,
\begin{align*}
\frac{\d}{\d t}\|\textbf{v}\|^2+ \left[ \left(1-\frac{3\alpha} 4\right ) \nu \lambda_1 -2\|\nabla h\|_{L^\infty}|y_\delta(\theta_t\omega)| \right ]
\|\textbf{v}\|^2+\frac{ \alpha \nu } 2\|A^{\frac12}\textbf{v}\|^2
\leq C  \left (1+|y_\delta(\theta_t\omega)|^4 \right)  .
\end{align*}
By  Gronwall's lemma, we can deduce that
\begin{align}
&\|\textbf{v}(t)\|^2+\frac {\alpha \nu } 2\int_0^te^{ \left(1-\frac{3\alpha} 4\right )\nu  \lambda_1(s-t)+2\|\nabla h\|_{L^\infty}\int_{s}^t|z(\theta_\tau\omega)|\, \d \tau}\|A^{\frac12}\textbf{v}(s)\|^2\, \d s\nonumber\\
&\quad \leq e^{-\left(1-\frac{3\alpha} 4\right ) \nu  \lambda_1t+2\|\nabla h\|_{L^\infty}\int_{0}^t|z(\theta_\tau\omega)|\, \d \tau}\|\textbf{v}(0)\|^2  \nonumber \\
& \qquad
+C\int_0^te^{ \left(1-\frac{3\alpha} 4\right ) \nu \lambda_1(s-t)+2\|\nabla h\|_{L^\infty}\int_{s}^t|z(\theta_\tau\omega)|\, \d \tau}  \left (1+|z(\theta_s\omega)|^4 \right) \d s, \nonumber
\end{align}
and then
\begin{align}
&\|\textbf{v}(t)\|^2+\frac {\alpha \nu} 2\int_0^te^{ \left(1-\frac{3\alpha} 4\right ) \nu \lambda_1(s-t) }\|A^{\frac12}\textbf{v}(s)\|^2\, \d s\nonumber\\
&\quad \leq e^{-\left(1-\frac{3\alpha} 4\right ) \nu \lambda_1t+2\|\nabla h\|_{L^\infty}\int_{0}^t|z(\theta_\tau\omega)|\, \d \tau}\|\textbf{v}(0)\|^2  \nonumber \\
& \qquad
+C\int_0^te^{ \left(1-\frac{3\alpha} 4\right ) \nu \lambda_1(s-t)+2\|\nabla h\|_{L^\infty}\int_{s}^t|z(\theta_\tau\omega)|\, \d \tau}  \left (1+|z(\theta_s\omega)|^4 \right)  \d s. \label{sep6.5}
\end{align}

Applying  \eqref{erg} for $m=1$, we can deduce that
\begin{align*}
\lim\limits_{t\rightarrow \pm \infty}\frac{1}{t}\int_0^t|y_{\delta}(\theta_s\omega)| \, \d s=\mathbb{E}(|y_{\delta}(\theta_t\omega)| )=\frac{ 1}{\sqrt{\pi\delta}}
, \quad \omega\in \Omega,
\end{align*}
so    there exists a  random variable  $ T_1(\omega)\geq1$ such that, for every $t\geq T_1(\omega)$,
\begin{align*}
 \frac 1t \int_{0}^t|y_{\delta}(\theta_s\omega)|\, \d s
 \leq  \frac{  1 }{\sqrt{\pi\delta}} + \frac{\beta}{\sqrt{\pi\delta}}     ,
\end{align*}
and then
\begin{align}
2\|\nabla h\|_{L^\infty}\int_{0}^t|y_{\delta}(\theta_s\omega)|\, \d s &\leq
   \frac{2\|\nabla h\|_{L^\infty}  }{ \sqrt{\pi\delta}}  \left( 1 +  \beta  \right ) t  \nonumber \\
   &= \left (1 -  \alpha\right) \nu \lambda_1t \quad \text{(by \eqref{c2})} .\label{erg1}
\end{align}
Then by  \eqref{sep6.5} and \eqref{erg1}, it is easy to deduce that
\begin{align}
&\|\textbf{v}(t)\|^2+\frac {\alpha \nu }2\int_0^te^{ \left(1-\frac{3\alpha} 4\right )\nu  \lambda_1(s-t) }\|A^{\frac12}\textbf{v}(s)\|^2\, \d s\nonumber\\
&\quad \leq e^{ - \frac{\alpha} 4  \nu  \lambda_1t }\|\textbf{v}(0)\|^2
+C\int_0^te^{ \left(1-\frac{3\alpha} 4\right )\nu  \lambda_1(s-t)+2\|\nabla h\|_{L^\infty}\int_{s}^t|y_{\delta}(\theta_\tau\omega)|\, \d \tau}  \left (1+|y_{\delta}(\theta_s\omega)|^4 \right)   \d s. \nonumber
\end{align}
Let  $\lambda:= \alpha \nu    \lambda_1 /4, $
 this estimate is then rewritten as
\begin{align*}
&\|\textbf{v}(t)\|^2+\frac {\alpha \nu } 2\int_0^te^{ \left( \frac 4 \alpha -3  \right) \lambda (s-t)}\|A^{\frac12}\textbf{v}(s)\|^2\, \d s\nonumber\\
 &\quad\leq e^{-\lambda t}\|\textbf{v}(0)\|^2
 +C\int_0^te^{ \left( \frac 4 \alpha -3  \right) \lambda (s-t)+2\|\nabla h\|_{L^\infty}\int_{s}^t|y_{\delta}(\theta_\tau\omega)|\, \d \tau}  \left (1+|y_{\delta}(\theta_s\omega)|^4 \right)  \d s. \nonumber
\end{align*}
Replacing $\omega$ with $\theta_{-t}\omega$, we can deduce that
 \begin{align*}
&\|\textbf{v}(t, \theta_{-t}\omega, \textbf{v}(0))\|^2+\frac {\alpha  \nu }2\int_0^te^{ \left( \frac 4 \alpha -3  \right) \lambda (s-t)}\|A^{\frac12}\textbf{v}(s, \theta_{-t}\omega, \textbf{v}(0))\|^2\, \d s\nonumber\\
 &\quad\leq e^{-\lambda t}\|\textbf{v}(0)\|^2
 +C\int^0_{-t} e^{ \left( \frac 4 \alpha -3  \right) \lambda  s +2\|\nabla h\|_{L^\infty}\int_{s}^0 |y_{\delta}(\theta_\tau\omega)|\, \d \tau}  \left (1+|y_{\delta}(\theta_s\omega)|^4 \right) \d s. \nonumber
\end{align*}
The random variable is given by
\begin{align*}
 \zeta_1(\omega) :=C\int_{-\infty}^0  e^{ \left( \frac 4 \alpha -3  \right) \lambda  s +2\|\nabla h\|_{L^\infty}\int_{s}^0 |y_{\delta}(\theta_\tau\omega)|\, \d \tau}  \left (1+|y_{\delta}(\theta_s\omega)|^4 \right)   \d s,\quad \omega\in \Omega .
\end{align*}
Then $\zeta_1(\cdot)$ is a tempered random variable such that  $\zeta_1(\omega)\geq 1$ and
\begin{align}
&\|\textbf{v}(t,\theta_{-t} \omega, \textbf{v}(0))\|^2+ \frac {\alpha \nu } 2\int_0^t e^{ \left ( \frac{4}{\alpha} -3  \right) \lambda (s-t)}\|A^{\frac12}\textbf{v}(s, \theta_{-t}\omega, \textbf{v}(0))\|^2\, \d s \nonumber \\
 &\quad \leq e^{-\lambda t}\|\textbf{v}(0)\|^2
 + \zeta_1(\omega) ,\quad t\geq T_1(\omega).  \label{4.29}
\end{align}
This completes the proof of Lemma \ref{lemma4.1}.
\end{proof}

\begin{lemma}[$H^1$ bound] \label{lem:H1bound}
Let  $f\in H$ and Assumption \ref{assum} hold.
Then for any bounded set $ B$ in $ H$ there exists a random variable $ T_B(\omega)>T_1(\omega)$ such that,  for any $t\geq T_B(\omega)$,
\begin{align*}
\sup_{\textbf{v}(0)\in B} \left( \big\|A^{\frac12}\textbf{v}(t,\theta_{-t}\omega,\textbf{v}(0))  \big\|^2 + \int_{t-\frac 12 }^t \|A\textbf{v}(s,\theta_{-t}\omega,\textbf{v}(0))\|^2  \ \d s\right)
\leq  \zeta_2(\omega),
\end{align*}
where $\zeta_2(\omega)$  is a tempered random variable given by \eqref{mar8.8} such that $\zeta_2(\omega) > \zeta_1(\omega)\geq 1$, $ \omega\in\Omega$.
\end{lemma}
\begin{proof}
Taking the inner product of the first equation   \eqref{2.2} with $A\textbf{v}$ in $H$, by integration by parts we can deduce that
\begin{align}\label{4.6}
&\frac12\frac{\d}{\d t}\|A^{\frac12}\textbf{v}\|^2+ \nu \int_{\mathbb{T}^2}\partial_{yy}v_1\Delta v_1dx+\nu \int_{\mathbb{T}^2}\partial_{xx}v_2\Delta v_2dx\nonumber\\
&=-\big(B(\textbf{v}+hy_{\delta}(\theta_t\omega),\textbf{v}+hy_{\delta}(\theta_t\omega)),\, A\textbf{v} \big)\nonumber\\
&\quad \ +\big(f-\nu A_1hy_{\delta}(\theta_t\omega) +hy_{\delta}(\theta_t\omega),A\textbf{v} \big),
\end{align}
here by direct computations, it yields
\begin{align}\label{3.10}
\nu \int_{\mathbb{T}^2}\partial_{yy}v_1\Delta v_1dx&=\nu \int_{\mathbb{T}^2}(\partial_{yy}v_1)^2dx
+\nu \int_{\mathbb{T}^2}\partial_{xx}v_1\partial_{yy}v_1dx\nonumber\\
&=\nu \int_{\mathbb{T}^2}(\partial_{yy}v_1)^2dx
+\nu \int_{\mathbb{T}^2}\partial_{xy}v_1\partial_{xy}v_1dx\nonumber\\
&=\nu \int_{\mathbb{T}^2}(\partial_{yy}v_1)^2dx
+\nu \int_{\mathbb{T}^2}(\partial_{yy}v_2)^2dx.
\end{align}
Similarly, it yields
\begin{align}\label{3.11}
\nu \int_{\mathbb{T}^2}\partial_{xx}v_2\Delta v_2dx&=\nu \int_{\mathbb{T}^2}(\partial_{xx}v_2)^2dx
+\nu \int_{\mathbb{T}^2}\partial_{xx}v_2\partial_{yy}v_2dx\nonumber\\
&=\nu \int_{\mathbb{T}^2}(\partial_{xx}v_2)^2dx
+\nu \int_{\mathbb{T}^2}\partial_{xy}v_2\partial_{xy}v_2dx\nonumber\\
&=\nu \int_{\mathbb{T}^2}(\partial_{xx}v_2)^2dx
+\nu \int_{\mathbb{T}^2}(\partial_{xx}v_1)^2dx.
\end{align}
Combining \eqref{3.10} and \eqref{3.11} yields
\begin{align}\label{3.12}
&\nu \int_{\mathbb{T}^2}\partial_{yy}v_1\Delta v_1dx+\nu \int_{\mathbb{T}^2}\partial_{xx}v_2\Delta v_2dx\nonumber\\
&=\nu \int_{\mathbb{T}^2}(\partial_{xx}v_1)^2dx+\nu \int_{\mathbb{T}^2}(\partial_{yy}v_1)^2dx
+\nu \int_{\mathbb{T}^2}(\partial_{xx}v_2)^2dx+\nu \int_{\mathbb{T}^2}(\partial_{yy}v_2)^2dx\nonumber\\
&\geq \frac12\nu \int_{\mathbb{T}^2}(\partial_{xx}v_1+\partial_{yy}v_1)^2dx
+\frac12\nu \int_{\mathbb{T}^2}(\partial_{xx}v_2+\partial_{yy}v_2)^2dx
\nonumber\\
&=\frac12\nu\|\Delta \textbf{v}\|^2=\frac12\nu\|A \textbf{v}\|^2.
\end{align}
By the H\"{o}lder inequality,  the Gagliardo-Nirenberg inequality, Poincar\'{e}'s inequality and the Young inequality, we can deduce that
\begin{align}\label{4.4}
&\big| \big(B(\textbf{v}+hy_{\delta}(\theta_t\omega),\textbf{v}+hy_{\delta}(\theta_t\omega)),\, A\textbf{v }\big)\big|
= \big| \big(B(\textbf{v}+hy_{\delta}(\theta_t\omega),\textbf{v}+hy_{\delta}(\theta_t\omega)),Ahy_{\delta}(\theta_t\omega) \big) \big| \nonumber\\
& \quad \leq \|Ahy_{\delta}(\theta_t\omega)\|\|\textbf{v}+hy_{\delta}(\theta_t\omega)\|_{L^4}\|A^{\frac12}(\textbf{v}+hy_{\delta}(\theta_t\omega))
\|_{L^4}\nonumber\\
& \quad \leq C\|Ahy_{\delta}(\theta_t\omega)\|\|\textbf{v}+hy_{\delta}(\theta_t\omega)\|^{\frac12}
\|A^{\frac12}(\textbf{v}+hy_{\delta}(\theta_t\omega))\|
\|A(\textbf{v}+hy_{\delta}(\theta_t\omega))\|^{\frac12}\nonumber\\
& \quad \leq C\|Ahy_{\delta}(\theta_t\omega)\|\|A^{\frac12}(\textbf{v}+hy_{\delta}(\theta_t\omega))\|^{\frac32}
\|A(\textbf{v}+hy_{\delta}(\theta_t\omega))\|^{\frac12}\nonumber\\
& \quad \leq  C\|Ahy_{\delta}(\theta_t\omega)\| \left (\|A^{\frac12}\textbf{v}\|^{\frac32}+\|A^{\frac12}hy_{\delta}(\theta_t\omega)\|^{\frac32}\right) \left
(\|A\textbf{v}\|^{\frac12}+\|Ahy_{\delta}(\theta_t\omega)\|^{\frac12} \right)\nonumber\\
& \quad \leq \frac  \nu 8\|A\textbf{v}\|^2+C\left (1+\|Ah\|^2|y_{\delta}(\theta_t\omega)|^2 \right )\|A^{\frac12}\textbf{v}\|^2+C\left (1+|y_{\delta}(\theta_t\omega)|^6\right),
\end{align}
and
\begin{align}
\big (f-\nu A_1hy_{\delta}(\theta_t\omega) +hy_{\delta}(\theta_t\omega),A\textbf{v} \big)
& \leq \frac  \nu 8 \|A\textbf{v}\|^2+C\|f\|^2+ C  |y_{\delta}(\theta_t\omega)|^2   .\label{4.5}
\end{align}
Inserting \eqref{3.12}, \eqref{4.4} and \eqref{4.5} into \eqref{4.6}, it is easy to get
\begin{align}\label{4.27}
\frac{\d}{\d t}\|A^{\frac12}\textbf{v}\|^2 +  \frac{\nu }{2} \|A\textbf{v}\|^2
&\leq  C\left ( 1 + \|Ah\|^2|y_{\delta}(\theta_t\omega)|^2 \right)\|A^{\frac12}\textbf{v}\|^2  +
C \left(1 +|y_{\delta}(\theta_t\omega)|^6+\|f\|^2 \right)\nonumber \\
&\leq  C\left ( 1 +  |y_{\delta}(\theta_t\omega)|^2 \right)  \|A^{\frac12}\textbf{v}\|^2  +
C \left(1 +|y_{\delta}(\theta_t\omega)|^6  \right).
\end{align}
Applying Gronwall's lemma on $(\eta, t)$ with $\eta\in  (t-1,t )$, $t >1$, it yields
\begin{align*}
&\|A^{\frac12}\textbf{v}(t)\|^2
+ \frac\nu 2 \int_\eta ^te^{ \int_s^tC(1+|y_{\delta}(\theta_\tau\omega)|^2)\, \d \tau}\|A\textbf{v}(s)\|^2 \, \d s \nonumber\\
&\quad \leq e^{ \int_\eta^tC(1+|y_{\delta}(\theta_\tau\omega)|^2)\, \d \tau}\|A^{\frac12}\textbf{v}( \eta)\|^2 +C\int_\eta^te^{ \int_s^t C(1+ |y_{\delta} (\theta_\tau\omega)|^2)\, \d \tau} \left (1 +|y_{\delta}(\theta_s\omega)|^6 \right)\d s,
\end{align*}
and then integrating over $\eta\in (t-1,t-\frac 12) $, we have
\begin{align*}
& \frac12  \|A^{\frac12}\textbf{v}(t )\|^2
+\frac  \nu 4\int_{t-\frac 12 }^t e^{ \int_s^tC(1+|y_{\delta}(\theta_\tau\omega)|^2)\, \d \tau}\|A\textbf{v}(s)\|^2\, \d s \nonumber\\
&\quad \leq \int_{t-1}^{t-\frac 12}  e^{ \int_\eta^tC(1+|y_{\delta}(\theta_\tau\omega)|^2)\, \d \tau} \|A^{\frac12}\textbf{v}( \eta)\|^2 \, \d \eta
 \nonumber\\
&\qquad +C\int_{t-1}^te^{ \int_s^t C(1+ |y_{\delta}(\theta_\tau\omega)|^2)\, \d \tau} \left (1 +|y_{\delta}(\theta_s\omega)|^6 \right) \d s \nonumber\\
&\quad \leq Ce^{   C \int_{ t-1}^t (1+|y_{\delta}(\theta_\tau\omega)|^2) \, \d \tau}  \left[  \int_{t-1}^{t  }  \|A^{\frac12}\textbf{v}( \eta)\|^2 \, \d \eta  + \int_{t-1}^t \left (1 +|y_{\delta}(\theta_s\omega)|^6 \right)\d s\right]  .
\end{align*}
Replacing $\omega$ with $\theta_{-t}\omega$ we can deduce that
\begin{align*}
&   \|A^{\frac12}\textbf{v}(t,\theta_{-t}\omega, \textbf{v}(0) )\|^2
+   \frac{\nu}{2} \int_{t-\frac 12 }^te^{ C(t-s) + \int_{s-t}^0C |y_{\delta}(\theta_\tau\omega)|^2 \, \d \tau}\|A\textbf{v}(s ,\theta_{-t}\omega, \textbf{v}(0))\|^2\, \d s\nonumber\\
&\quad  \leq  C e^{ C \int_{-1}^0 |y_{\delta}(\theta_\tau\omega)|^2\, \d \tau}
\left [ \int_{t-1}^t   \|A^{\frac12}\textbf{v}(s,\theta_{-t}\omega, \textbf{v}(0) )\|^2\, \d s+\int_{-1}^0 \left(1+ |y_{\delta}(\theta_{s} \omega)|^6  \right)  \d s\right].
\end{align*}
By     \eqref{4.29},
\begin{align*}
 \nu  \int_{t-1}^t  \|A^{\frac12}\textbf{v}(s,\theta_{-t}\omega, \textbf{v}(0) )\|^2  \, \d s
 & \leq   \nu e^{\left( \frac 4\alpha -3 \right)  \lambda }   \int_{t-1}^t  e^{\left( \frac 4\alpha -3 \right)  \lambda (s-t) } \|A^{\frac12}\textbf{v}(s,\theta_{-t}\omega, \textbf{v}(0) )\|^2 \, \d s
 \\
 & \leq  \frac{2}{\alpha    }   e^{\left( \frac 4\alpha -3 \right)  \lambda }   \left( e^{-\lambda t}\|\textbf{v}(0)\|^2
 +   \zeta_1(\omega)\right) , \quad    t\geq T_1(\omega) , \nonumber
\end{align*}
so
we can deduce that
\begin{align*}
&   \|A^{\frac12}\textbf{v}(t,\theta_{-t}\omega, \textbf{v}(0) )\|^2
 +  \frac{\nu}{2} \int_{t-\frac 12 }^t  \|A\textbf{v}(s ,\theta_{-t}\omega, \textbf{v}(0))\|^2\, \d s \nonumber\\
 &\quad  \leq  \|A^{\frac12}\textbf{v}(t,\theta_{-t}\omega, \textbf{v}(0) )\|^2
+  \frac{\nu}{2} \int_{t-\frac 12 }^te^{ C(t-s) + \int_{s-t}^0C |y_{\delta}(\theta_\tau\omega)|^2 \, \d \tau}\|A\textbf{v}(s ,\theta_{-t}\omega, \textbf{v}(0))\|^2\, \d s\nonumber\\
&\quad \leq   C e^{ C \int_{-1}^0 |y_{\delta} (\theta_\tau\omega)|^2\, \d \tau} \left(   e^{-\lambda t}\|\textbf{v}(0)\|^2
 +   \zeta_1(\omega)  +
\int_{-1}^0 |y_{\delta}(\theta_{s} \omega)|^6  \, \d s \right) , \quad    t\geq T_1(\omega) ,
\end{align*}
where  we have used the fact  that $\zeta_1(\omega)\geq 1$.
Hence, let
\begin{align}\label{mar8.8}
 \zeta_2(\omega) :=   C e^{ C \int_{-1}^0 |y_{\delta}(\theta_\tau\omega)|^2\, \d \tau} \left(      \zeta_1(\omega)  +
\int_{-1}^0 |y_{\delta}(\theta_{s} \omega)|^6   \ \d s \right) ,
\quad \omega\in \Omega,
\end{align}
here, $\zeta_2(\omega)$ is a tempered random variable.  This completes the proof of Lemma \ref{lem:H1bound}.

\end{proof}

Let a random set
$\mathfrak{B}=\{\mathfrak{B}(\omega)\}_{\omega\in\Omega}$ in $H^1$ is defined by
\begin{align}\label{4.19}
\mathfrak{B}(\omega):= \left \{ \textbf{v} \in H^1:\|A^{\frac12}\textbf{v}\|^2\leq \zeta_2(\omega)
\right \}, \quad \omega\in\Omega,
\end{align}
where $\zeta_2(\cdot)$ is the tempered random variable given by \eqref{mar8.8}. By Lemma \ref{lem:H1bound}, we have that  $\mathfrak{B}$ is a  random absorbing set of the RDS $\phi$ generated by the system  \eqref{2.2}.  In addition,  the compact embedding $H^1 \hookrightarrow H$ gives the compactness of $\mathfrak B(\omega)$ in $H$, hence the RDS $\phi$ has a    random attractor $\A$ in $H$. Moreover, by the Langa \& Robinson \cite{langa2}, this random attractor has finite fractal dimension in $H$.  Then we introduce the following main result.

\begin{theorem}\label{theorem4.6}  Let Assumption \ref{assum} hold and $f\in H$. The  RDS $\phi$ generated by  the random  anisotropic NS equations  \eqref{2.2} driven by colored noise has a  random absorbing set $\mathfrak{B}$ which is a tempered and bounded random set in $H^1 $, and has also a   random attractor $\mathcal A $ in $H$. In addition, $\mathcal A$ has finite fractal dimension in $H$:
\[
d_f^H(\A(\omega))  \leq d,\quad \omega \in \Omega,
\]
for some positive constant $d$.
\end{theorem}

 In the Section \ref{sec4}, Section \ref{sec5} and Section \ref{sec6}, we will prove that this random attractor $\A$ is in fact an $(H,H^2)$-random attractor of $\phi$.

\section{Construction of  an $H^2$ random absorbing set}\label{sec4}

In this section, we will prove an  $H^2$ random absorbing set of system  \eqref{2.2} as $f\in H$.   This will be done by estimating  the difference between the   solutions of the random anisotropic NS equations \eqref{2.2} and that of the  deterministic anisotropic NS equations \eqref{2.1} within the global attractor $\A_0$, see Remark \ref{rmk1}.  This comparison approach seems first employed in Cui \& Li \cite{cui}.

\subsection{$H^2$-distance between random and deterministic trajectories}

By the Theorem \ref{theorem4.6},  we have that the random anisotropic NS equations  \eqref{2.2} driven by colored noise  has an $H^1$ random absorbing set $\mathfrak B$. Meanwhile, by the Lemma \ref{lem:det},  we have that  the deterministic anisotropic NS equations  \eqref{2.1} has a global attractor $\mathcal A_0$ bounded in $H^2$, it suffices to restrict ourselves to the random absorbing set  $\mathfrak B$ and the  global attractor $\mathcal A_0$. Note that    the $\omega$-dependence  of each  random time in the following estimates will be crucial for later analysis.

\begin{lemma}[$H^1$-distance]
\label{lemma4.3}
 Let Assumption \ref{assum} hold and $f\in H$. There   exist   random variables $T_{\mathfrak B}(\cdot)$ and    $\zeta_4(\cdot)$, where  $\zeta_4(\cdot)$ is tempered,  such that the solutions $\textbf{v}$  of the random  anisotropic NS equations \eqref{2.2} driven by colored noise and $\textbf{u}$  of the deterministic anisotropic NS equations  \eqref{2.1} satisfy
\begin{align}
 \big \|A^{\frac 12} \textbf{v}(t,\theta_{-t}\omega,\textbf{v}(0))-A^{\frac 12} \textbf{u}(t,\textbf{u}(0)) \big\| ^2 \leq \zeta_4(\omega),\quad \forall t\geq T_{\mathfrak B}(\omega), \nonumber
\end{align}
for any $\textbf{v}(0)\in\mathfrak{B}(\theta_{-t}\omega)$
 and $\textbf{u}(0)\in\mathcal{A}_0$ for $\omega\in \Omega$.
\end{lemma}
\begin{proof}
Let  $w=\textbf{v}-\textbf{u}$ denote the difference between the two solutions, here $w=(w_1,w_1)$. Then we introduce the following system
\begin{align}\label{4.8}
\frac{\d w}{\d t}+ \nu  A_1w+B( \textbf{v}+hz(\theta_t\omega))-B(\textbf{u}) = hy_\delta(\theta_t\omega) - \nu A_1hy_\delta(\theta_t\omega).
\end{align}
Taking the inner product of the system   \eqref{4.8} with $Aw$ in $H$,  by integration by parts we can deduce that
\begin{align}\label{4.11}
 &\frac12\frac{\d}{\d t}\|A^{\frac12}w\|^2+ \nu \int_{\mathbb{T}^2}\partial_{yy}w_1\Delta w_1dx+\nu \int_{\mathbb{T}^2}\partial_{xx}w_2\Delta w_2dx \nonumber \\
&= \big (hy_\delta(\theta_t\omega)-\nu A_1hy_\delta(\theta_t\omega) ,Aw \big )  - \big(B(\textbf{v}+hy_\delta(\theta_t\omega))-B(\textbf{u}),Aw\big) \nonumber \\
&= \big  (hy_\delta(\theta_t\omega)-\nu A_1hy_\delta(\theta_t\omega) ,Aw \big ) - \big (B(w+hy_\delta(\theta_t\omega),\textbf{u}),Aw \big ) \nonumber\\
&\quad -  \big (B(\textbf{v}+hy_\delta(\theta_t\omega),w+hy_\delta(\theta_t\omega)),Aw  \big ) \nonumber \\
&=: I_{1}(t) + I_{2}(t)+I_{3}(t) .
\end{align}
By direct computations and \eqref{3.10}, \eqref{3.11} and \eqref{3.12}, it yields
\begin{align}\label{4.1}
&\nu \int_{\mathbb{T}^2}\partial_{yy}w_1\Delta w_1dx+\nu \int_{\mathbb{T}^2}\partial_{xx}w_2\Delta w_2dx\nonumber\\
&\geq \frac12\nu \int_{\mathbb{T}^2}(\partial_{xx}w_1+\partial_{yy}w_1)^2dx
+\frac12\nu \int_{\mathbb{T}^2}(\partial_{xx}w_2+\partial_{yy}w_2)^2dx
\nonumber\\
&=\frac12\nu\|\Delta w\|^2=\frac12\nu\|A w\|^2.
\end{align}
For three terms of $I_{1}(t)$-$I_{3}(t)$, by H\"{o}lder's inequality, the Gagliardo-Nirenberg inequality,  Poincar\'{e}'s inequality and Young's inequality, we obtain the estimates
\begin{align}\label{4.10}
 I_{1}(t)   &\leq \nu\|A_1hy_\delta(\theta_t\omega)\|\|Aw\|+\|hy_\delta(\theta_t\omega)\|\|Aw\|\nonumber\\
&\leq \frac  \nu 8\|Aw\|^2+C|y_\delta(\theta_t\omega)|^2,
\end{align}
\begin{align}
I_{2}(t)&\leq \|w+hy_\delta(\theta_t\omega)\|_{L^\infty}\|A^{\frac12}\textbf{u}\|\|Aw\|\nonumber\\
&\leq C\|A^{\frac12}\textbf{u}\|\|Aw\| \left(\|Aw\|^{\frac12}+\|Ahy_\delta(\theta_t\omega)\|^{\frac12} \right) \left
(\|A^{\frac12}w\|^{\frac12}+\|A^{\frac12}hy_\delta(\theta_t\omega)\|^{\frac12} \right )\nonumber\\
&\leq C\|A^{\frac12}\textbf{u}\|\|Aw\|^{\frac32}\|A^{\frac12}w\|^{\frac12}
+C\|A^{\frac12}\textbf{u}\|\|Aw\|^{\frac32}|y_\delta(\theta_t\omega)|^{\frac12}\nonumber\\
&\quad +C\|A^{\frac12}\textbf{u}\|\|Aw\|\|A^{\frac12}w\|^{\frac12}|y_\delta(\theta_t\omega)|^{\frac12}+
C\|A^{\frac12}\textbf{u}\|\|Aw\||y_\delta(\theta_t\omega)|\nonumber\\
&\leq \frac{\nu}{16}\|Aw\|^2+C\|A^{\frac12}\textbf{u}\|^4\|A^{\frac12}w\|^2
+C \left (\|A^{\frac12}\textbf{u}\|^2+\|A^{\frac12}\textbf{u}\|^4 \right) |y_\delta(\theta_t\omega)|^2,
\end{align}
and
\begin{align}
I_{3}(t)&\leq \|\textbf{v}+hy_\delta(\theta_t\omega)\|_{L^4}\|A^{\frac12}(w+hy_\delta(\theta_t\omega))\|_{L^4}\|Aw\|\nonumber\\
&\leq C \left (\|\textbf{v}\|_{L^4}+\|hy_\delta(\theta_t\omega)\|_{L^4} \right ) \left(\|A^{\frac12}w\|_{L^4}+\|A^{\frac12}hy_\delta(\theta_t\omega)\|_{L^4} \right)\|Aw\|\nonumber\\
&\leq C\|A^{\frac12}\textbf{v}\|\|A^{\frac12}w\|^{\frac12}\|Aw\|^{\frac32}
+C\|A^{\frac12}hy_\delta(\theta_t\omega)\|\|A^{\frac12}w\|^{\frac12}\|Aw\|^{\frac32}\nonumber\\
&\quad +C\|A^{\frac12}\textbf{v}\|\|Ahy_\delta(\theta_t\omega)\|\|Aw\|+C\|A^{\frac12}hy_\delta(\theta_t\omega)\|
\|Ahy_\delta(\theta_t\omega)\|\|Aw\|
\nonumber\\
&\leq \frac{\nu}{16}\|Aw\|^2+C \! \left (\|A^{\frac12}\textbf{v}\|^4+|y_\delta(\theta_t\omega)|^4 \right )\|A^{\frac12}w\|^2
+C \! \left (\|A^{\frac12}\textbf{v}\|^4+|y_\delta(\theta_t\omega)|^4 \right).
\nonumber
\end{align}
 Since $\textbf{u}$ is a trajectory within the global attractor $\mathcal{A}_0$, we deduce $\|A^{\frac12}\textbf{u}(t)\|\leq C$ for any $t\geq0$, it yields
\begin{align}\label{4.9}
I_{2}(t)
\leq \frac{\nu}{16}\|Aw\|^2+C\|A^{\frac12}w\|^2+C|y_\delta(\theta_t\omega)|^2.
\end{align}
Inserting \eqref{4.10}-\eqref{4.9} into \eqref{4.11} it is easy to deduce that
\begin{align}
\frac{\d}{\d t}\|A^{\frac12}w\|^2+  \frac{\nu}{2} \|Aw\|^2
&\leq C \left (1+\|A^{\frac12}\textbf{v}\|^4+|y_\delta(\theta_t\omega)|^4 \right )\|A^{\frac12}w\|^2  \nonumber \\
&\quad
+C \left(1+\|A^{\frac12}\textbf{v}\|^4+|y_\delta(\theta_t\omega)|^4\right) .
\label{sep6.1}
\end{align}

For any $t >1$ and $s\in[t-1,t]$, we apply Gronwall's lemma to \eqref{sep6.1} yields
\begin{align}
\|A^{\frac12}w(t)\|^2&\leq e^{\int_s^tC(1+\|A^{\frac12}\textbf{v}(\eta) \|^4+|y_\delta(\theta_\eta\omega)|^4)\, \d \eta}\|A^{\frac12}w(s)\|^2\nonumber\\
&
 \quad +C\int_s^te^{\int_\tau^tC(1+\|A^{\frac12}\textbf{v}(\eta) \|^4+|y_\delta(\theta_\eta\omega)|^4)\, \d \eta} \left(1+\|A^{\frac12}\textbf{v}(\tau)\|^4+|y_\delta(\theta_\tau\omega)|^4 \right) \d \tau\nonumber\\
&\leq Ce^{\int_s^tC(\|A^{\frac12}\textbf{v}(\eta) \|^4+|y_\delta(\theta_\eta\omega)|^4)\, \d \eta} \nonumber  \\
&\quad \times
 \left[\|A^{\frac12}w(s)\|^2
+\int_{t-1}^t    \left(1+\|A^{\frac12}\textbf{v}\|^4+|y_\delta(\theta_\tau\omega)|^4\right) \d \tau \right]. \nonumber
\end{align}
Integrating w.r.t$.$  $s$ over $[t-1,t]$, we can deduce that
\begin{align*}
\|A^{\frac12}w(t)\|^2
&\leq Ce^{\int_{t-1}^tC(\|A^{\frac12}\textbf{v}(\eta)\|^4+|y_\delta(\theta_\eta\omega)|^4)\, \d \eta}\nonumber\\
&\quad \times \left[\int_{t-1}^t\|A^{\frac12}w(s)\|^2\, \d s
+\int_{t-1}^t  \left(1+\|A^{\frac12}\textbf{v}\|^4+|y_\delta(\theta_\tau \omega)|^4\right) \d \tau\right]. \nonumber
\end{align*}
   Since $\textbf{u}$ is a trajectory within the global attractor, $\|A^{\frac 12} \textbf{u}\|^2$ is uniformly bounded. Hence, we have
\begin{align*}
\|A^{\frac 12} w(s)\|^2
 =\|A^{\frac 12} (\textbf{v}-\textbf{u}) (s)\|^2
  \leq \|A^{\frac 12} \textbf{v} (s)\|^4 + C,
\end{align*}
and then
\begin{align*}
\|A^{\frac12}w(t)\|^2
 \leq Ce^{\int_{t-1}^tC(\|A^{\frac12}\textbf{v}(\eta)\|^4+|y_\delta(\theta_\eta\omega)|^4)\, \d \eta}  \int_{t-1}^t  \left(1+\|A^{\frac12}\textbf{v}(\tau)\|^4+|y_\delta(\theta_\tau\omega)|^4 \right)\d \tau .
\end{align*}
Replacing $\omega$ with $\theta_{-t}\omega$, for any $t>1$  it can deduce that
\begin{align}\label{4.14}
  \|A^{\frac12}w(t,\theta_{-t}\omega,w(0))\|^2
&\leq Ce^{\int_{t-1}^tC \|A^{\frac12}\textbf{v}(\eta,\theta_{-t}\omega,\textbf{v}(0))\|^4 \d \eta+  C\int_{-1}^0|y_\delta(\theta_{\eta}\omega)|^4 \, \d \eta}  \nonumber\\
 &\quad
 \times \left(\int_{t-1}^t  \|A^{\frac12}\textbf{v}(s,\theta_{-t}\omega,\textbf{v}(0))\|^4 \, \d s + \int_{-1}^0  |y_\delta(\theta_{s}\omega)|^4 \, \d s +1 \right) .
\end{align}

Recall from \eqref{4.27} that
\begin{align}\label{4.2}
\frac{\d}{\d t}\|A^{\frac12}\textbf{v}\|^2 + \frac{\nu}{2} \|A\textbf{v}\|^2
 \leq  C\left ( 1 +  |y_\delta(\theta_t\omega)|^2 \right)  \|A^{\frac12}\textbf{v}\|^2  +
C \left(1 +|y_\delta(\theta_t\omega)|^6  \right).
\end{align}
For any $t>3$,  integrating both sides of \eqref{4.2} over $(\tau, t)$ with $\tau\in (t-3,t-2)$, then we can deduce that
\begin{align*}
&
\|A^{\frac12}\textbf{v} (t) \|^2 -\|A^{\frac12}\textbf{v}(\tau) \|^2  + \frac{\nu }{2}\int^t_{t-2} \|A \textbf{v}(s)\|^2 \, \d s \\
&\quad
 \leq  C \int^t_{t-3} \left(1+  |y_\delta(\theta_s\omega)|^2 \right) \|A^{\frac12}\textbf{v}(s) \|^2\, \d s
+ C\int^t_{t-3} \left(1 +|y_\delta(\theta_s\omega)|^6 \right )  \d s,
\end{align*}
and then integrating over $\tau\in (t-3,t-2)$ yields
\begin{align*}
\|A^{\frac12}\textbf{v} (t) \|^2+  \frac{\nu}{2} \int^t_{t-2} \|A \textbf{v}(s)\|^2\, \d s
& \leq   C \int^t_{t-3} \left (|y_\delta(\theta_s\omega)|^2+1\right)  \|A^{\frac12}\textbf{v}(s) \|^2 \, \d s \\
&\quad + C\int^t_{t-3} \left (1 +|y_\delta(\theta_s\omega)|^6 \right)  \d s ,\quad t> 3.
\end{align*}
For $\varepsilon\in  [0,2]$, we replace $\omega$ with $\theta_{-t-\varepsilon} \omega$  to deduce
\begin{align}
 & \|A^{\frac12}\textbf{v} (t,\theta_{-t-\varepsilon} \omega, \textbf{v}(0)) \|^2 + \frac{ \nu}{2} \int_{t-2}^t  \|A\textbf{v} (s,\theta_{-t-\varepsilon} \omega, \textbf{v}(0)) \|^2 \, \d s \nonumber \\
& \quad \leq   C \int^t_{t-3} \! \left (|y_\delta(\theta_{s-t-\varepsilon} \omega)|^2+1\right)  \|A^{\frac12}\textbf{v}(s,\theta_{-t-\varepsilon} \omega, \textbf{v}(0)) \|^2 \, \d s + C\int^{-\varepsilon}_{-\varepsilon-3} \! \left(1 +|y_\delta(\theta_s\omega)|^6 \right )  \d s\nonumber\\
& \quad \leq   C\left (  \sup_{\tau\in (-5,0)}|y_\delta(\theta_\tau\omega)|^2+1\right) \int^{t+\varepsilon}_{t+\varepsilon-5}  \|A^{\frac12}\textbf{v}(s,\theta_{-t-\varepsilon} \omega, \textbf{v}(0)) \|^2 \, \d s \nonumber\\
&\qquad + C\int^{0}_{-5}    \left(1 +|y_\delta(\theta_s\omega)|^6 \right)   \d s ,\quad t>5 . \label{mar8.1}
\end{align}
By \eqref{4.29}, it yields
\begin{align*}
  \frac {\alpha  \nu }2\int_0^te^{ \left( \frac 4 \alpha -3  \right) \lambda (s-t)}\|A^{\frac12}\textbf{v}(s, \theta_{-t}\omega, \textbf{v}(0))\|^2\, \d s
 \leq e^{-\lambda t}\|\textbf{v}(0)\|^2
 + \zeta_1(\omega) ,\quad t\geq T_1(\omega) ,
 \end{align*}
where $\zeta_1(\omega)\geq 1$ is a tempered random variable,
so we can deduce that
\begin{align*}
  \int^{t+\varepsilon}_{t+\varepsilon-5}  \|A^{\frac12}\textbf{v}(s,\theta_{-(t+\varepsilon)} \omega, \textbf{v}(0)) \|^2 \, \d s
  &\leq  C \int^{t+\varepsilon}_{t+\varepsilon-5}  e^{ \left( \frac 4 \alpha -3  \right) \lambda (s-t)}\|A^{\frac12}\textbf{v}(s,\theta_{-t-\varepsilon} \omega, \textbf{v}(0)) \|^2 \, \d s \\[0.8ex]
  & \leq Ce^{-\lambda t}\|\textbf{v}(0)\|^2
 + C \zeta_1(\omega) ,\quad t\geq T_1(\omega)+5 ,
\end{align*}
uniformly for $\varepsilon \in [0,2]$.  Hence, by coming back to  \eqref{mar8.1} we can deduce that
\begin{align}
 &\sup_{\varepsilon \in [0,2]} \left(\|A^{\frac12}\textbf{v }(t,\theta_{-t-\varepsilon} \omega, \textbf{v}(0)) \|^2  +  \frac{\nu }{2} \int_{t-2}^t  \|A\textbf{v} (s,\theta_{-t-\varepsilon} \omega, \textbf{v}(0)) \|^2 \, \d s\right ) \nonumber \\
&\quad \leq   C\left (  \sup_{\tau\in (-5,0)}|y_\delta(\theta_\tau\omega)|^2+1\right)\left(e^{-\lambda t}\|v(0)\|^2
 +   \zeta_1(\omega)   \right) + C\int^{0}_{-5} \left (1 +|y_\delta(\theta_s\omega)|^6 \right )    \d s \nonumber \\
 & \quad \leq   C\left (  \sup_{\tau\in (-5,0)}|y_\delta(\theta_\tau\omega)|^6+1\right)\left(e^{-\lambda t}\|v(0)\|^2
 +   \zeta_1(\omega)   \right)  , \quad t\geq T_1(\omega)+5.  \nonumber
\end{align}
 In other words,  for any  $t\geq T_1(\omega)+7$ and $\eta\in [t-2,t ]$ $($so that $\eta\geq T_1(\omega)+5)$, it can obtain that
\begin{align*}
& \|A^{\frac12}\textbf{v} (\eta ,\theta_{-t } \omega, \textbf{v}(0)) \|^2 + \frac{\nu}{2} \int_{\eta-2}^\eta  \|A\textbf{v } (s,\theta_{-t} \omega, \textbf{v}(0)) \|^2 \, \d s\\
  & \quad =  \|A^{\frac12}\textbf{v} (\eta ,\theta_{-\eta-(t-\eta) } \omega, \textbf{v}(0)) \|^2
  + \frac{\nu}{2}  \int_{\eta-2}^\eta  \|A\textbf{v} (s,\theta_{-\eta-(t-\eta)} \omega, \textbf{v}(0)) \|^2 \, \d s \\
&\quad \leq    C\left (  \sup_{\tau\in (-5,0)}|y_\delta(\theta_\tau\omega)|^6+1\right)\left(e^{-\lambda  \eta}\|\textbf{v}(0)\|^2
 +   \zeta_1(\omega)   \right)
  \\
&\quad  \leq   C\left (  \sup_{\tau\in (-5,0)}|y_\delta(\theta_\tau\omega)|^6+1\right)\left(e^{-\lambda t}\|\textbf{v}(0)\|^2
 +   \zeta_1(\omega)   \right) ,
\end{align*}
where in the last inequality we have used $\eta\geq t-2$.
Since the initial value $\textbf{v}(0)$ comes from within the absorbing set $\mathfrak B$, then there exists a random variable $T_{\mathfrak B}(\omega) \geq T_1(\omega)+ 7$
such that
\begin{align} \label{timeB}
 \sup_{\textbf{v}(0)\in \mathfrak B(\theta_{-t} \omega)}  e^{-\lambda t} \|\textbf{v}(0)\|^2 \leq 1,\quad t\geq T_{\mathfrak B}(\omega).
\end{align}
Hence, by \eqref{timeB}, we can get   the crucial estimate as follows
\begin{align}
 &\sup_{\eta\in [t-2,t]}
  \left(
 \|A^{\frac12}\textbf{v} (\eta ,\theta_{-t } \omega, \textbf{v}(0)) \|^2
 + \frac{\nu}{2} \int_{\eta-2}^\eta  \|A\textbf{v} (s,\theta_{-t} \omega, \textbf{v}(0)) \|^2 \, \d s\right) \notag  \\
& \quad \leq  C\left (  \sup_{\tau\in (-5,0)}|y_\delta(\theta_\tau\omega)|^6+1\right)\big(1
 +   \zeta_1(\omega)   \big)
 =:  \zeta_3(\omega) ,\quad t\geq T_{\mathfrak B} (\omega) ,\label{mar8.3}
\end{align}
where    $ \zeta_3 $ is a tempered random variable. By \eqref{mar8.3}, we can deduce that
an immediate  consequence is the following estimate
\begin{align}\label{mar8.2}
  \int_{t-1}^t\|A^{\frac12}\textbf{v}(\eta,\theta_{-t}\omega,\textbf{v}(0))\|^4 \, \d \eta
&  \leq    |\zeta_3(\omega)|^2 ,  \quad t\geq T_{\mathfrak B} (\omega) .
\end{align}

 Inserting \eqref{mar8.2} into \eqref{4.14}, we can deduce that
\begin{align}  \label{mar8.4}
\|A^{\frac12}w(t,\theta_{-t}\omega,w(0))\|^2
&\leq Ce^{ |\zeta_3 (\omega)|^2+ C\int_{-1}^0|y_\delta(\theta_{\eta}\omega)|^4   \d \eta}  \left( |\zeta_3(\omega)|^2 + \int_{-1}^0 |y_\delta(\theta_{s}\omega)|^4 \,  \d s  \right) \nonumber \\
&\leq Ce^{ 2|\zeta_3 (\omega)|^2+ C\int_{-1}^0|y_\delta(\theta_{\eta}\omega)|^4   \d \eta}
=: \zeta_4(\omega) , \quad t\geq T_{\mathfrak B} (\omega) .
\nonumber
\end{align}
This completes the proof of Lemma \ref{lemma4.3}.
\end{proof}

\begin{lemma}[$H^2$-distance] \label{lem:H2bd}  Let Assumption \ref{assum} hold and $f\in H$.
There exists a tempered random variable $\zeta_5(\cdot)$ such that the solutions $\textbf{v}$ of the random anisotropic NS equations \eqref{2.2} and $\textbf{u}$  of the deterministic anisotropic NS equations  \eqref{2.1} satisfy
\[
 \|A\textbf{v}(t,\theta_{-t}\omega,\textbf{v}(0))-A\textbf{u}(t,\textbf{u}(0))\|^2\leq \zeta_5(\omega),
\quad t\geq T_{\mathfrak B}(\omega),
\]
for any $\textbf{v}(0)\in\mathfrak{B}(\theta_{-t}\omega)$
 and $\textbf{u}(0)\in\mathcal{A}_0$,
where $T_{\mathfrak B}$ denotes the random variable defined in  Lemma \ref{lemma4.3}.
\end{lemma}
\begin{proof}

Taking the inner product of the system   \eqref{4.8} with $A^2w$ in $H$,  by integration by parts, we can deduce that
\begin{align}\label{4.17}
 &\frac12\frac{\d}{\d t}\|Aw\|^2- \nu \int_{\mathbb{T}^2}\partial_{yy}w_1\Delta^2 w_1dx-\nu \int_{\mathbb{T}^2}\partial_{xx}w_2\Delta^2 w_2dx\nonumber \\
&= \left(-\nu A_1hy_\delta(\theta_t\omega) +hy_\delta(\theta_t\omega),A^2w \right) \nonumber \\
&\quad -\left(B(\textbf{v}+hy_\delta(\theta_t\omega)) -B(\textbf{u}),A^2w \right)\nonumber\\
&=  \left(-\nu A_1hy_\delta(\theta_t\omega) +hy_\delta(\theta_t\omega),A^2w \right)\nonumber\\
&\quad -\left(B(\textbf{v}+hy_\delta(\theta_t\omega),w+hy_\delta(\theta_t\omega)),A^2w \right )\nonumber\\
&\quad
-\left(B(w+hy_\delta(\theta_t\omega),\textbf{u}),A^2w \right )\nonumber\\
&=: I_{4}(t)+I_{5}(t)+I_{6}(t).
\end{align}
By direct computations and $\nabla\cdot w=0$ and \eqref{3.9}, it also yields
\begin{align}
&-\nu \int_{\mathbb{T}^2}\partial_{yy}\Delta w_1\Delta w_1dx-\nu \int_{\mathbb{T}^2}\partial_{xx}\Delta w_2\Delta w_2dx\nonumber\\
&= \nu \int_{\mathbb{T}^2}(\partial_{y}\Delta w_1)^2dx
+\nu \int_{\mathbb{T}^2}(\partial_{x}\Delta w_2)^2dx
\nonumber\\
&\geq\frac12\nu\|A^{\frac32} w\|^2.
\end{align}

In order to get \eqref{sep6.3}, we only to estimate each terms on the right hand side of \eqref{4.17}.
By using H\"{o}lder's inequality, Young's inequality and $h\in H^3$,  we can deduce that
\begin{align}
I_{4}(t)&\leq  \nu \|A^{\frac32}hy_\delta(\theta_t\omega)\|\|A^{\frac32}w\|+ \|A^{\frac12}hy_\delta(\theta_t\omega)\|\|A^{\frac32}w\|\nonumber\\
&\leq \frac{ \nu }{16}\|A^{\frac32}w\|^2+C  |y_\delta(\theta_t\omega)|^2. \label{4.18}
\end{align}
By using H\"{o}lder's inequality and the Gagliardo-Nirenberg inequality, it also yields
\begin{align}
I_{5}(t)&\leq\|A^\frac12\textbf{u}_h\|_{L^4}\|\nabla(w+hy_\delta(\theta_t\omega))\|_{L^4}\|A^\frac32w\|
+\|\textbf{u}_h\|_{L^4}\|A(w+hy_\delta(\theta_t\omega))\|_{L^4}\|A^\frac32w\|
\nonumber\\
&\leq C\|A^\frac12\textbf{u}_h\|^\frac12\|A\textbf{u}_h\|^\frac12
\|\nabla(w+hy_\delta(\theta_t\omega))\|^\frac12\|A(w+hy_\delta(\theta_t\omega))\|^\frac12
\|A^\frac32w\|\nonumber\\
& \quad +C\|\textbf{u}_h\|^\frac12\|A^{\frac12}\textbf{u}_h\|^\frac12
\|A(w+hy_\delta(\theta_t\omega))\|^\frac12\|A^{\frac32}(w+hy_\delta(\theta_t\omega))\|^\frac12\|A^\frac32w\|\nonumber\\
&\leq  C\|A^\frac12\textbf{u}_h\|^\frac12\|A\textbf{u}_h\|^\frac12
\left(\|A^{\frac12}w\|^\frac12+ |y_\delta(\theta_t\omega)|^\frac12 \right) \left (\|Aw\|^\frac12+ |y_\delta(\theta_t\omega) |^\frac12 \right)
\|A^\frac32w\|\nonumber\\
&\quad +C\|A^{\frac12}\textbf{u}_h\| \left (\|Aw\|^\frac12+ |y_\delta(\theta_t\omega)|^\frac12\right)\left(\|A^{\frac32}w\|^\frac12
+ | y_\delta(\theta_t\omega)|^\frac12 \right)\|A^\frac32w\|\nonumber\\
&=: J_1(t) +J_2(t) . \nonumber
\end{align}

Since $\mathcal{A}_0$ is bounded in $H^2$,  $\|A\textbf{u}(t) \|^2+\|A^{1/2}\textbf{u}(t) \|^2\leq C$ for each $t\geq 0$, hence we can deduce that
\begin{align}\label{4.15}
\|A\textbf{u}_h\|^2&\leq C\|Aw\|^2+C\|A\textbf{u}\|^2
 +  C\|Ahy_\delta(\theta_t\omega)\|^2 \nonumber \\
&
\leq C\|Aw\|^2+  C  +C \|Ahy_\delta(\theta_t\omega)\|^2
,
\end{align}
and similarly, we also have
\begin{align}  \label{mar5.3}
  \|A^{\frac 12}w\|^2
   &\leq  C\|A^{\frac 12} \textbf{v}\|^2 +C
   \nonumber  \\
   &\leq C\|A^{\frac 12} \textbf{u}_h\|^2 + C\|A^{\frac 12} hy_\delta(\theta_t\omega) \|^2 +C .
\end{align}
Hence,  for  the term $J_1(t)$, we introduce the following inequalities
\begin{align*}
  &C\|A^\frac12\textbf{u}_h\|^\frac12\|A\textbf{u}_h\|^\frac12\|A^\frac12w\|^\frac12
\|Aw\|^\frac12\|A^\frac32w\|
& \nonumber\\
&\quad  \leq    \frac \nu {64}\|A^\frac32w\|^2  + C \|A^\frac12\textbf{u}_h\|  \|A \textbf{u}_h\|   \|A^{\frac 12} w\|  \|Aw\|
\nonumber
\\
&\quad  \leq   \frac \nu {64}\|A^\frac32w\|^2
+ C \|A^\frac12\textbf{u}_h\|^2  \|A \textbf{u}_h\|^2+      \|A^{\frac 12} w\|^2 \|Aw\|^2
\nonumber
\\
&\quad  \leq    \frac \nu {64} \|A^\frac32w\|^2
+ C\left( \|A^\frac12\textbf{u}_h\|^2 +  \|A^\frac 12 w   \|^2  \right) \|A w\|^ 2
+ C  \|A^\frac12\textbf{u}_h\|^2   \left(  1+ |y_\delta(\theta_t\omega) |^2  \right) ,
\end{align*}
\begin{align*}
  &C\|A^\frac12\textbf{u}_h\|^\frac12\|A\textbf{u}_h\|^\frac12\|A^\frac12w\|^\frac12
 |y_\delta(\theta_t\omega)|^\frac12\|A^\frac32w\|
& \nonumber\\
&\quad  \leq   \frac \nu {64}\|A^\frac32w\|^2  + C \|A^\frac12\textbf{u}_h\|  \|A \textbf{u}_h\|   \|A^{\frac 12} w\| |y_\delta(\theta_t\omega) |
\nonumber
\\
& \quad  \leq    \frac \nu {64}\|A^\frac32w\|^2
+ C \|A^\frac12\textbf{u}_h\|^2  \|A \textbf{u}_h\|^2+   C   \|A^{\frac 12} w\|^2 |y_\delta(\theta_t\omega) |^2
\nonumber
\\
&\quad \leq      \frac \nu {64}\|A^\frac32w\|^2
+ C  \|A^\frac12\textbf{u}_h\|^2  \|A w\|^2  + C \big(  \|A^\frac12\textbf{u}_h\|^2 +|y_\delta(\theta_t\omega) |^2 +1\big) \big(  1+|y_\delta(\theta_t\omega) |^2  \big)  ,
\end{align*}

\begin{align*}
  &C\|A^\frac12\textbf{u}_h\|^\frac12\|A\textbf{u}_h\|^\frac12 |y_\delta(\theta_t\omega)|^\frac12
\|Aw\|^\frac12\|A^\frac32w\|
& \nonumber\\
&\quad  \leq    \frac \nu {64} \|A^\frac32w\|^2  + C \|A^\frac12\textbf{u}_h\|  \|A \textbf{u}_h\|  |y_\delta(\theta_t\omega) |  \|Aw\|
\nonumber
\\
&\quad \leq     \frac \nu {64}\|A^\frac32w\|^2
+ C \|A^\frac12\textbf{u}_h\|^2  \|A \textbf{u}_h\|^2+    C| y_\delta(\theta_t\omega)|^2 \|Aw\|^2
\nonumber
\\
&\quad \leq     \frac \nu {64}\|A^\frac32w\|^2
+ C\left( \|A^\frac12\textbf{u}_h\|^2 +  |y_\delta(\theta_t\omega)|^2  \right) \|A w\|^2
+ C  \|A^\frac12\textbf{u}_h\|^2   \left(  1+ |y_\delta(\theta_t\omega) |^2  \right) ,
\end{align*}
and
\begin{align*}
  &C\|A^\frac12\textbf{u}_h\|^\frac12\|A\textbf{u}_h\|^\frac12 | y_\delta(\theta_t\omega)|  \|A^\frac32w\|
& \nonumber\\
&\quad \leq     \frac \nu {64}\|A^\frac32w\|^2  + C \|A^\frac12\textbf{u}_h\| \|A \textbf{u}_h\|   |y_\delta(\theta_t\omega)|^2
\nonumber
\\
&\quad \leq    \frac \nu {64}\|A^\frac32w\|^2
+ C \|A^\frac12\textbf{u}_h\|^2  \|A \textbf{u}_h\|^2   + C|y_\delta(\theta_t\omega) |^4
\nonumber
\\
&\quad \leq   \frac \nu {64}\|A^\frac32w\|^2
+ C  \|A^\frac12\textbf{u}_h\|^2     \left(  \|A w\|^2  +1+|y_\delta(\theta_t\omega) |^2  \right)
+ C|y_\delta(\theta_t\omega) |^4 ,
\end{align*}
it follows that
\begin{align}
J_1(t)
&\leq \frac  \nu {16}\|A^\frac32w\|^2
+C\left(\|A^\frac12\textbf{u}_h\|^2+\|A^\frac12w\|^2+|y_\delta(\theta_t\omega)|^2 \right)
\|Aw\|^2\nonumber\\
&\quad +C\left(1+\|A^\frac12\textbf{u}_h\|^4 \right)+C   |y_\delta(\theta_t\omega)|^4 .
\nonumber
\end{align}

For   the term $J_2(t)$, applying \eqref{4.15} and the Gagliardo-Nirenberg inequality and the Young inequality we can deduce that
\begin{align}
J_2(t)
&\leq C\|A^{\frac12}\textbf{u}_h\|\|Aw\|^\frac12\|A^{\frac32}w\|^\frac 32 +C\|A^{\frac12}\textbf{u}_h\|\|Aw\|^\frac12|y_\delta(\theta_t\omega)|^\frac12\|A^\frac32w\|
 \nonumber\\
&\quad+C\|A^{\frac12}\textbf{u}_h\| |y_\delta(\theta_t\omega)|^\frac12\|A^{\frac32}w\|^\frac 32
+C\|A^{\frac12}\textbf{u}_h\| |y_\delta(\theta_t\omega)|  \|A^\frac32w\|\nonumber\\
&\leq \frac \nu {16}\|A^\frac32w\|^2+C\left (1+\|A^\frac12\textbf{u}_h\|^4\right )\|Aw\|^2+C\left (1+\|A^\frac12\textbf{u}_h\|^8 \right)  +C |y_\delta(\theta_t\omega)|^4, \nonumber
\end{align}
and then  for $I_{5} (t)=J_1(t)+ J_2(t)$ we can get the following estimate
\begin{align}
I_{5}(t) &\leq  \frac \nu 8\|A^\frac32w\|^2
+C\left(1+ \|A^\frac12\textbf{u}_h\|^4+\|A^\frac12w\|^2+ |y_\delta(\theta_t\omega)|^2 \right)
\|Aw\|^2 \nonumber  \\
&\quad +C\left(1+\|A^\frac12\textbf{u}_h\|^8 \right)
+C  |y_\delta(\theta_t\omega)|^4. \label{sep6.2}
\end{align}

For the term $I_{6}(t)$, applying the similar method, we also can deduce that
\begin{align}\label{4.16}
I_{6}(t)&\leq\|A^{\frac12}(w+hy_\delta(\theta_t\omega))\|_{L^4}\|\nabla \textbf{u}\|_{L^4}\|A^\frac32w\|+\|w+hy_\delta(\theta_t\omega)\|_{L^\infty}\|A\textbf{u}\|\|A^\frac32w\|\nonumber\\
&\leq C\|A^{\frac12}(w+hy_\delta(\theta_t\omega))\|^\frac12\|A(w+hy_\delta(\theta_t\omega))\|^\frac12\|A^\frac12\textbf{u}\|^\frac12
\|A\textbf{u}\|^\frac12\|A^\frac32w\|\nonumber\\
& \quad +C\|A^\frac12(w+hy_\delta(\theta_t\omega))\|^\frac12\|A(w+hy_\delta(\theta_t\omega))\|^\frac12\|A\textbf{u}\|\|A^\frac32w\|\nonumber\\
&\leq C\|A^{\frac12}(w+hy_\delta(\theta_t\omega))\|^\frac12\|A(w+hy_\delta(\theta_t\omega))\|^\frac12\|A^\frac32w\|\nonumber\\
&\leq \frac \nu {16}\|A^\frac32w\|^2+C\left(\|A^{\frac12}w\|+\|A^{\frac12}hy_\delta(\theta_t\omega)\| \right) \big (\|Aw\|
+\|Ahy_\delta(\theta_t\omega)\| \big )\nonumber\\
&\leq \frac \nu {16}\|A^\frac32w\|^2+C\|Aw\|^2 +C |y_\delta(\theta_t\omega)|^2.
\end{align}
Inserting \eqref{4.18}, \eqref{sep6.2} and \eqref{4.16} into \eqref{4.17}, by \eqref{mar5.3}  we can get that
\begin{align}
 \frac{\d}{\d t}\|Aw\|^2+  \frac{\nu}{2}  \|A^{\frac32}w\|^2
&\leq C\left (1+\|A^\frac12\textbf{u}_h\|^4+\|A^\frac12w\|^2+|y_\delta(\theta_t\omega)|^2 \right )
\|Aw\|^2\nonumber\\
&\quad +C \left (1+\|A^\frac12\textbf{u}_h\|^8+|y_\delta(\theta_t\omega)|^4 \right) \nonumber \\
 &\leq C\left (1+\|A^\frac12 \textbf{v}\|^4 +|y_\delta(\theta_t\omega)|^4 \right )
\|Aw\|^2  \nonumber \\
&\quad +C \left (1+\|A^\frac12 \textbf{v}\|^8 +|y_\delta(\theta_t\omega)|^8 \right) .  \label{sep6.3}
\end{align}

 For any $t >1$ and $s\in[t-1,t]$, applying Gronwall's lemma to \eqref{sep6.3}   it is easy to get that
\begin{align*}
\|Aw(t)\|^2&\leq e^{C\int_s^t( 1+\|A^\frac12 \textbf{v}(\eta)\|^4 +|y_\delta(\theta_\eta\omega)|^4)\, \d \eta}\|Aw(s)\|^2 \nonumber\\
&\quad +\int_s^t Ce^{C\int_\tau^t  (1+\|A^\frac12 \textbf{v}\|^4+|y_\delta(\theta_\eta\omega)|^4)\, \d \eta}
 \Big (1+\|A^\frac12 \textbf{v}\|^8  +|y_\delta(\theta_\tau\omega)|^8\Big ) \, \d \tau\nonumber\\
&\leq Ce^{C\int_{t-1}^t(1+\|A^\frac12 \textbf{v}\|^8 +|y_\delta(\theta_\eta\omega)|^8)\, \d \eta} \nonumber \\
&\quad \times \left [\|Aw(s)\|^2+\int_{t-1}^t
\left(1+\|A^\frac12 \textbf{v}\|^8  +|y_\delta(\theta_\tau\omega)|^8\right)\d \tau \right]\\
& \leq Ce^{C\int_{t-1}^t(1+\|A^\frac12 \textbf{v}\|^8 +|y_\delta(\theta_\eta\omega)|^8)\, \d \eta} \Big(\|Aw(s)\|^2+ 1 \Big)
, \nonumber
\end{align*}
where in the last inequality we have used the basic relation $x\leq e^x$ for $x\geq 0$.
Integrating w.r.t. $s$ over   $  [t-1,t] $, we can get that
 \begin{align*}
\|Aw(t)\|^2
&\leq Ce^{C\int_{t-1}^t  (1+\|A^\frac12 \textbf{v}(\eta)\|^8+|y_\delta(\theta_\eta\omega)|^8)\, \d \eta}
 \left (\int_{t-1}^t\|Aw(s)\|^2\, \d s + 1\right). \nonumber
\end{align*}
Then we replace $\omega$ with $\theta_{-t}\omega$ to deduce
\begin{align}\label{4.22}
\|Aw(t,\theta_{-t}\omega,w(0))\|^2
&\leq Ce^{C\int_{t-1}^t \|A^\frac12 \textbf{v}(\eta,\theta_{-t}\omega, \textbf{v}(0))\|^8 \d \eta + C\int^0_{-1} |y_\delta(\theta_{\eta}\omega) |^8
\, \d \eta}\nonumber\\
&\quad \times \left(\int_{t-1}^t\|Aw(s,\theta_{-t}\omega,w(0))\|^2\, \d s +1  \right) .
\end{align}
By \eqref{mar8.3}, it can get that
\begin{align} \label{4.23}
  \int_{t-1}^t  \|A^\frac12 \textbf{v}(s,\theta_{-t}\omega,\textbf{v}(0))\|^8\, \d s \leq  |\zeta_3(\omega)|^4   ,  \quad t\geq T_{\mathfrak B}(\omega).
\end{align}
Moreover, since $\textbf{u}$ is a trajectory within the global attractor $\mathcal A_0$ which is   bounded  in $H^2$,  by \eqref{mar8.3}  again we can deduce that
\begin{align}
\int_{t-1}^t\|Aw(s,\theta_{-t}\omega,w(0))\|^2\, \d s
&\leq2\int_{t-1}^t\|A\textbf{v}(s,\theta_{-t}\omega,\textbf{v}(0))\|^2\, \d s+2\int_{t-1}^t\|A\textbf{u}\|^2\, \d s\nonumber\\
&\leq  2\int_{t-1}^t \|A\textbf{v}(s,\theta_{-t}\omega,\textbf{v}(0))\|^2\, \d s  +  2\|\mathcal A \|_{H^2} \nonumber\\
&\leq \frac 4  \nu \, \zeta_3(\omega) +C ,  \quad t\geq T_{\mathfrak B}(\omega).  \label{mar8.7}
\end{align}
Hence, inserting \eqref{4.23} and \eqref{mar8.7} into \eqref{4.22} we can deduce that
\begin{align}
\|Aw(t,\theta_{-t}\omega,w(0))\|^2
&\leq C e^{C|\zeta_3(\omega) |^4 + C\int^0_{-1} |y_\delta(\theta_{\eta}\omega) |^8
\, \d \eta} \left(   \zeta_3(\omega) +1
\right)  \nonumber \\
&=:\zeta_5(\omega),  \quad t\geq T_{\mathfrak B}(\omega).  \nonumber
\end{align}
 This completes the proof of Lemma \ref{lem:H2bd}.
\end{proof}

\subsection{$H^2$ random absorbing sets}
By Lemma \ref{lem:H2bd} we can construct an $H^2$ random absorbing set of the  random anisotropic  (NS) equations \eqref{2.2} driven by colored noise. Nevertheless, for   later purpose we  make  the following  stronger estimate than Lemma \ref{lem:H2bd}.
It will be crucial in the next section in deriving the local $(H,H^2)$-Lipschitz continuity of the equations.
\begin{lemma}\label{lemma4.4}
 Let Assumption \ref{assum} hold and $f\in H$. There exists a tempered random variable $\rho(\cdot)$ such that  the solutions $\textbf{v}$ of the random anisotropic  NS equations \eqref{2.2} driven by colored noise and the solutions $\textbf{u}$ of the deterministic anisotropic  NS equations \eqref{2.1} satisfy
\[
  \sup_{\varepsilon\in [0,1]} \|A\textbf{v}(t,\theta_{-t-\varepsilon}\omega,  \textbf{v}(0))- A\textbf{u}(t,\textbf{u}_0)\|^2     \leq  \rho(\omega) ,\quad t\geq T_{\mathfrak B} (\omega),
\]
whenever $\textbf{v}(0)\in \mathfrak B(\theta_{-t-\varepsilon} \omega)$ and $\textbf{u}(0)\in \mathcal A_0$, where $T_{\mathfrak B} (\cdot)$ is given  in \eqref{timeB}.
\end{lemma}

\begin{proof}
By \eqref{sep6.3}, we can deduce that
\begin{align}\label{4.0}
 \frac{\d}{\d t}\|Aw\|^2  &\leq C\left (1+\|A^\frac12 \textbf{v}\|^4 +|y_\delta(\theta_t\omega)|^4 \right )
\|Aw\|^2
+C \left (1+\|A^\frac12 \textbf{v}\|^8 +|y_\delta(\theta_t\omega)|^8 \right)\nonumber\\
&\leq     C\left (1+\|A^\frac12 \textbf{v}\|^8 +|y_\delta(\theta_t\omega)|^8 \right )\left( \|Aw\|^2+1\right)   .
\end{align}
For any $t \geq 1$, integrating the  inequality \eqref{4.0} from $(s,t)$ for $s\in (t-1,t)$  and then integrating over $s\in (t-1,t)$,  we can deduce that
\begin{align*}
 \|Aw(t)\|^2
& \leq    C\int_{t-1}^t   \left (1 +\|A^\frac12\textbf{v}(s)\|^8+|y_\delta(\theta_s\omega)|^8 \right ) \left( \|Aw(s)\|^2 +1\right)\d s.
\end{align*}
For any $\varepsilon\in [0,1]$, replacing $\omega$ with $\theta_{-t-\varepsilon}\omega$ yields
\begin{align*}
& \|Aw(t,\theta_{-t-\varepsilon}\omega, w(0))\|^2   \\
&\quad  \leq   C  \int_{t-1}^t   \left (1 +\|A^\frac12\textbf{v}(s,\theta_{-t-\varepsilon}\omega, \textbf{v}(0))\|^8+|y_\delta(\theta_{s-t-\varepsilon}\omega)|^8 \right ) \left( \|Aw(s)\|^2 +1\right) \d s   \\
&\quad  \leq   C  \int_{t+\varepsilon-2}^{t+\varepsilon}   \left (1 +\|A^\frac12\textbf{v}(s,\theta_{-t-\varepsilon}\omega, \textbf{v}(0))\|^8+|y_\delta(\theta_{s-t-\varepsilon}\omega)|^8 \right ) \left( \|Aw(s)\|^2 +1\right) \d s .
\end{align*}
Note that from  \eqref{mar8.3}  it can deduce that
\begin{align*}
\sup_{\eta\in (t-2,t)}
 \|A^{\frac12}\textbf{v} (\eta ,\theta_{-t } \omega, \textbf{v}(0)) \|^8
&\leq  |\zeta_3(\omega)|^4 ,\quad t\geq T_{\mathfrak B} (\omega) .
\end{align*}
Hence,  let
\begin{align*}
 \zeta_6(\omega):=   |\zeta_3(\omega)|^4 +\sup_{s\in(-2,0)} |z(\theta_s\omega)|^8+1
\end{align*}
we    given a tempered random variable $\zeta_6$ such that
\begin{align*}
  \|Aw(t,\theta_{-t-\varepsilon}\omega, w(0))\|^2     \leq   C \zeta_6(\omega)  \int_{t+\varepsilon-2}^{t+\varepsilon}   \! \big( \|Aw(s,\theta_{-t-\varepsilon } \omega, w(0))\|^2 +1\big)\, \d s
\end{align*}
for any $t\geq T_{\mathfrak B} (\omega)$ uniformly for $\varepsilon\in [0, 1]$.
By  \eqref{mar8.7}, we can have  that
\begin{align*}
 \int_{t+\varepsilon-2}^{t+\varepsilon}   \|Aw(s,\theta_{-t-\varepsilon } \omega, w(0))\|^2 \, \d s \leq \frac 4 \nu \, \zeta_3(\omega) +C
\end{align*}
for any $\varepsilon\in [0,1]$ and $t\geq  T_{\mathfrak B} (\omega) $, hence we can deduce that
\begin{align}\label{rho}
\sup_{\varepsilon\in [0,1]}
  \|Aw(t,\theta_{-t-\varepsilon}\omega, w(0))\|^2     \leq   C \zeta_6(\omega)\left(\zeta_3(\omega) +1\right)=: \rho(\omega)
\end{align}
for any $ t\geq T_{\mathfrak B} (\omega)$. This completes the proof of Lemma \ref{lemma4.4}.
\end{proof}

Let $\varepsilon =0$, then we  can construct an $H^2$ random absorbing set.
\begin{theorem}[$H^2$ absorbing set]
\label{theorem4.1}
 Let Assumption \ref{assum} hold and $f\in H$. The  RDS $\phi$ generated by the random  anisotropic  NS equations \eqref{2.2} driven by colored noise has a  random absorbing set $\mathfrak{B}_{H^2}$, defined as a random $H^2$ neighborhood of the global attractor $\mathcal{A}_0$ of the deterministic anisotropic  NS equations:
\begin{align}
\mathfrak{B}_{H^2}(\omega)=\left \{\textbf{v}\in H^2: \, {\rm dist}_{H^2}(\textbf{v},\mathcal{A}_{0}) \leq \sqrt{\rho(\omega)} \, \right\}, \quad \omega\in\Omega, \nonumber
\end{align}
where $\rho(\cdot)$ is the tempered random variable defined by \eqref{rho}.  As a consequence, the $\mathcal{D}_H$-random attractor $\mathcal{A}$ of \eqref{2.2} is a bounded and tempered random set in $H^2$.

\end{theorem}
\begin{proof}
This is a direct consequence of  Lemma \ref{lemma4.4}, see also   \cite[ Theorem 13]{cui}.\end{proof}

\section{The ($H,H^2$)-smoothing effect} \label{sec5}

In this section, we will prove the ($H,H^2$)-smoothing effect  of the random anisotropic  NS equations \eqref{2.2} driven by colored noise. This effect  is essentially a
local $(H, H^2)$-Lipschitz continuity in initial values. It will be shown  step-by-step by   proving the local Lipschitz continuity in $H$, the  $(H,H^1)$-smoothing and finally the  $(H,H^2)$-smoothing. For simplicity, we set
 $$ \bar v(t,\omega, \bar v(0)):= \textbf{v}_1(t,\omega, \textbf{v}_1(0))-\textbf{v}_2(t,\omega, \textbf{v}_2(0)) $$
the difference between two solutions.

\subsection{Lipschitz continuity in $H$}

\begin{lemma}\label{lemma5.0}
 Let Assumption \ref{assum} hold and $f\in H$.
 For  any tempered set $\mathfrak D\in \D_H $, then there is random variables $t_{\mathfrak D}(\cdot)$ and $L_1(\mathfrak D,\cdot) $ such that any
 two solutions $\textbf{v}_1$ and $\textbf{v}_2$ of the random anisotropic  NS equations \eqref{2.2} driven by colored noise corresponding to initial values  $\textbf{v}_{1,0},$ $\textbf{v}_{2,0}$ in $ \mathfrak D( \theta_{-t_{\mathfrak D}(\omega)}\omega),$   respectively,  satisfy
\begin{align*}
 & \left \| \textbf{v}_1 \! \left ( t_{\mathfrak D} (\omega) ,\theta_{-t_{\mathfrak D}(\omega)}\omega,  \textbf{v}_{1,0}\right)
 - \textbf{v}_2 \! \left ( t_{\mathfrak D} (\omega) ,\theta_{-t_{\mathfrak D}(\omega)}\omega, \textbf{v}_{2,0}\right) \right \|^2  \\[0.8ex]
 &\quad
 \leq L_1({\mathfrak D}, \omega) \|\textbf{v}_{1,0}-\textbf{v}_{2,0}\|^2  ,\quad \omega \in \Omega.
\end{align*}

\end{lemma}
\begin{proof}
The  difference $\bar{v}  =\textbf{v}_1-\textbf{v}_2$ of   solutions of system \eqref{2.2}  satisfies
\begin{align}\label{5.1}
\frac{\d \bar{v}}{
\d t}+  \nu  A_1\bar{v}+B( \textbf{v}_1+hy_\delta(\theta_t\omega))-B( \textbf{v}_2+hy_\delta(\theta_t\omega))=0.
\end{align}
Taking the inner product of the system   \eqref{5.1} with $\bar{v}$ in $H$,  by integration by parts, we can deduce that
\begin{align*}
\frac{1}{2}\frac{\d}{\d t}\|\bar{v}\|^2+ \nu  \|\partial_y\bar{v}_1\|^2+\nu  \|\partial_x\bar{v}_2\|^2&=- \big (B(\bar{v},\textbf{v}_1+hy_\delta(\theta_t\omega)),\bar{v} \big )
-\big (B(\textbf{v}_2+hy_\delta(\theta_t\omega),\bar{v}),\bar{v}\big )\nonumber\\
&=-\big (B(\bar{v},\textbf{v}_1+hy_\delta(\theta_t\omega)),\bar{v} \big )\nonumber\\
&\leq \|\bar{v}\|^2_{L^4}\|\nabla(\textbf{v}_1+hy_\delta(\theta_t\omega))\|\nonumber\\
&\leq C\|\bar{v}\|\|A^{\frac12}\bar{v}\| \left (\|A^\frac12\textbf{v}_1\|+\|A^\frac12hy_\delta(\theta_t\omega)\| \right)\nonumber\\
&\leq \frac\nu4\|A^{\frac12}\bar{v}\|^2+C \left(\|A^\frac12\textbf{v}_1\|^2+|y_\delta(\theta_t\omega)|^2 \right )\|\bar{v}\|^2.
\end{align*}
By direct computations and \eqref{3.9}, it yields
\begin{align*}
 \frac{\nu}{2}\|\bar{v}\|^2\leq\nu  \|\partial_y\bar{v}_1\|^2+\nu  \|\partial_x\bar{v}_2\|^2.
\end{align*}
Moreover,
\begin{align*}
\frac{\d}{\d t}\|\bar{v}\|^2+  \frac{\nu}{2} \|A^{\frac12}\bar{v}\|^2
\leq C\left (\|A^\frac12\textbf{v}_1\|^2+|y_\delta(\theta_t\omega)|^2 \right) \|\bar{v}\|^2.
\end{align*}
By using Gronwall's lemma, we can  get that
\begin{align} \label{mar9.1}
 & \|\bar{v}(t)\|^2+  \frac{\nu}{2} \int_0^te^{\int_s^tC(\|A^\frac12\textbf{v}_1(\tau) \|^2+|y_\delta(\theta_\tau\omega)|^2)\, \d \tau}\|A^{\frac12}\bar{v} (s) \|^2\, \d s \nonumber \\
& \quad
\leq e^{\int_0^tC(\|A^\frac12\textbf{v}_1(\tau)\|^2+|y_\delta(\theta_\tau\omega)|^2)\, \d \tau}\|\bar{v}(0)\|^2, \quad t>0.
\end{align}

By virtue of \eqref{4.29}, we   can deduce that
\begin{align*}
   \int_0^t \|A^\frac12\textbf{v}_1(\tau,\theta_{-t}\omega, \textbf{v}_{1,0})\|^2 \, \d \tau &\leq e^{ \left(\frac 4 \alpha -3  \right)\lambda t}  \int_0^t e^{ \left( \frac 4 \alpha -3  \right) \lambda (\tau-t)} \|A^\frac12\textbf{v}_1 (\tau,\theta_{-t}\omega, \textbf{v}_{1,0})\|^2  \,  \d \tau  \\
 &   \leq \frac {2}{\alpha\nu} e^{ \left(\frac 4 \alpha -3  \right)\lambda t}  \left( e^{-\lambda t} \|\textbf{v}_{1,0}\|^2 +\zeta_1(\omega) \right) ,\quad t\geq T_1(\omega),
 \end{align*}
so for \eqref{mar9.1} we also can  get
\begin{align*}
 \|\bar v(t,\theta_{-t}\omega, \bar v(0))\|^2
 & \leq  e^{\int_0^tC \left (\|A^\frac12\textbf{v}_1(\tau,\theta_{-t}\omega, \textbf{v}_{1,0})\|^2+|y_\delta(\theta_{\tau-t} \omega)|^2 \right)   \d \tau }
 \|\bar{v}(0)\|^2 \\
 &\leq e^{C  e^{ \left(\frac 4 \alpha -3  \right)\lambda t}  \left( e^{-\lambda t} \|\textbf{v}_{1,0}\|^2 +\zeta_1(\omega) \right) +C\int_{-t}^0|y_\delta(\theta_{\tau} \omega)|^2 \, \d \tau } \|\bar{v}(0)\|^2
\end{align*}
for any $t\geq T_1 (\omega)$.
Since $\textbf{v}_{1,0} $ belongs to $\mathfrak D(\theta_{-t}\omega) $ which is tempered, there exists a random variable $ t_{\mathfrak D} (\omega)  \geq T_1(\omega)$  such that
\begin{align*}
 e^{-\lambda t_{\mathfrak D}(\omega) } \|\textbf{v}_{1,0}\|^2 \leq  e^{-\lambda t_{\mathfrak D}(\omega) } \left \|  \, \mathfrak D \!  \left (\theta_{-t_{\mathfrak D}(\omega)} \omega \right ) \right\|^2 \leq 1,\quad \omega\in \Omega.
\end{align*}
Moreover, since $\mathfrak D$ is pullback absorbed by the absorbing set $\mathfrak B$,  this exists a large enough $t_{\mathfrak D}(\omega)$  such that
\begin{align}\label{mar19.2}
 \phi \left ( t_{\mathfrak D}(\omega) , \theta_{-t_{\mathfrak D}(\omega)}\omega, \mathfrak D(\theta_{- t_{\mathfrak D}(\omega) }\omega) \right ) \subset \mathfrak B(\omega), \quad \omega\in \Omega .
\end{align}
Therefore,  a random variable given by as follows
\begin{align*}
L_1({\mathfrak D}, \omega)=e^{C  e^{ \left(\frac 4 \alpha -3  \right)\lambda  t_{\mathfrak D}(\omega)}    ( 1+\zeta_1(\omega)  ) + C\int_{-t_{\mathfrak D}(\omega)}^0|y_\delta(\theta_{\tau} \omega)|^2 \, \d \tau } ,\quad \omega\in \Omega,
\end{align*}
then we have
\be \label{mar18.8}
\left \|\bar v  \! \left ( t_{\mathfrak D} (\omega) ,\theta_{-t_{\mathfrak D}(\omega)}\omega, \bar v(0) \right) \right \|^2
 \leq L_1({\mathfrak D}, \omega) \|\bar{v}(0)\|^2  ,\quad
\ee
whenever $\textbf{v}_{1,0},$ $\textbf{v}_{2,0}\in \mathfrak D \big ( \theta_{-t_{\mathfrak D}( \omega)}\omega \big),$  $\omega\in \Omega$. This completes the proof of Lemma \ref{lemma5.0}.
\end{proof}

\subsection{$(H,H^1)$-smoothing}

In this subsection, the $(H,H^1)$-smoothing will be shown  by two steps. Since  we have constructed  by   \eqref{4.19} an $H^1$ random absorbing set $\mathfrak B$, we begin with initial values lying  in the absorbing set $\mathfrak B$, and then consider the  initial values  in every tempered set $\mathfrak D $ in $\D_H$.

\begin{lemma}\label{lemma5.2}[$(H,H^1)$-smoothing on $\mathfrak B$]
 Let Assumption \ref{assum} hold and $f\in H$. For   the random absorbing set $\mathfrak B$ given by \eqref{4.19}     there are  random variables $\tau_\omega$ and $L_2(\mathfrak B,\omega) $ such that any
 two solutions $\textbf{v}_1$ and $\textbf{v}_2$ of the random anisotropic  NS equations \eqref{2.2} driven by colored noise corresponding to initial values  $\textbf{v}_{1,0},$ $\textbf{v}_{2,0}$ in $ \mathfrak B( \theta_{- \tau_\omega}\omega),$   respectively,  satisfy
\begin{align}\label{mar18.9}
 \left \| \textbf{v}_1  (  \tau_\omega,\theta_{-\tau_\omega}\omega,  \textbf{v}_{1,0} )
 - \textbf{v}_2   ( \tau_\omega,\theta_{-\tau_\omega}\omega, \textbf{v}_{2,0} ) \right \|^2_{H^1}
 \leq L_2({\mathfrak B}, \omega) \|\textbf{v}_{1,0}-\textbf{v}_{2,0}\|^2  ,\quad \omega \in \Omega.
\end{align}
\end{lemma}

\begin{proof}
Taking the inner product of the system   \eqref{5.1} with $A\bar{v}$ in $H$,  by integration by parts, we can deduce that
\begin{align}\label{5.2}
&\frac{1}{2}\frac{\d}{\d t}\|A^\frac12\bar{v}\|^2+ \nu \int_{\mathbb{T}^2}\partial_{yy}\bar{v}_1\Delta \bar{v}_1\d x+\nu \int_{\mathbb{T}^2}\partial_{xx}\bar{v}_2\Delta \bar{v}_2\d x\nonumber\\
&=- \big (B(\bar{v},\textbf{v}_1+hy_\delta(\theta_t\omega)),A\bar{v}\big )
-\big (B(\textbf{v}_2+hy_\delta(\theta_t\omega),\bar{v}),A\bar{v}\big ) .
\end{align}
By direct computations and similar method, we also have
\begin{align}
&\nu \int_{\mathbb{T}^2}\partial_{yy}\bar{v}_1\Delta \bar{v}_1\d x+\nu \int_{\mathbb{T}^2}\partial_{xx}\bar{v}_2\Delta \bar{v}_2\d x\nonumber\\
&\geq \frac12\nu \int_{\mathbb{T}^2}(\partial_{xx}\bar{v}_1+\partial_{yy}\bar{v}_1)^2\d x
+\frac12\nu \int_{\mathbb{T}^2}(\partial_{xx}\bar{v}_2+\partial_{yy}\bar{v}_2)^2\d x
\nonumber\\
&=\frac12\nu\|\Delta \bar{v}\|^2=\frac12\nu\|A \bar{v}\|^2.
\end{align}
By using the H\"{o}lder inequality, the Gagliardo-Nirenberg inequality and the Young inequality, we can get that
\begin{align}\label{5.3}
 \big| \big (B(\bar{v},\textbf{v}_1+hy_\delta(\theta_t\omega)),A\bar{v}\big )\big|
 &\leq \|\bar{v}\|_{L^\infty}\|\nabla(\textbf{v}_1+hy_\delta(\theta_t\omega))\|\|A\bar{v}\|\nonumber\\
&\leq C\|\bar{v}\|^\frac12\|A\bar{v}\|^\frac32
\left (\|A^\frac12\textbf{v}_1\|+\|A^\frac12hy_\delta(\theta_t\omega)\| \right)\nonumber\\
&\leq C\|A^\frac12\bar{v}\|^\frac12\|A\bar{v}\|^\frac32 \left (\|A^\frac12\textbf{v}_1\|+ |y_\delta(\theta_t\omega) | \right)\nonumber\\
&\leq \frac \nu 8\|A\bar{v}\|^2+C \left (\|A^\frac12\textbf{v}_1\|^4+|y_\delta(\theta_t\omega)|^4\right)\|A^\frac12\bar{v}\|^2,
\end{align}
and
\begin{align}\label{5.4}
 \big| \big (B(\textbf{v}_2+hy_\delta(\theta_t\omega),\bar{v}),A\bar{v}\big )\big| &\leq \|A^\frac12\bar{v}\|_{L^4}\|\textbf{v}_2+hy_\delta(\theta_t\omega)\|_{L^4}\|A\bar{v}\|\nonumber\\
&\leq C\|A^\frac12\bar{v}\|^\frac12\|A\bar{v}\|^\frac32 \left(\|A^\frac12\textbf{v}_2\|+\|A^\frac12hy_\delta(\theta_t\omega)\| \right)\nonumber\\
&\leq \frac \nu 8\|A\bar{v}\|^2+C\left (\|A^\frac12\textbf{v}_2\|^4+|y_\delta(\theta_t\omega)|^4\right)\|A^\frac12\bar{v}\|^2.
\end{align}
Inserting \eqref{5.3}-\eqref{5.4} into \eqref{5.2}, we can deduce that
\begin{align} \label{mar9.2}
\frac{\d}{\d t}\|A^\frac12\bar{v}\|^2+  \frac{\nu}{2} \|A\bar{v}\|^2
\leq C\left(\|A^\frac12\textbf{v}_1\|^4+\|A^\frac12\textbf{v}_2\|^4+|y_\delta(\theta_t\omega)|^4\right)\|A^\frac12\bar{v}\|^2.
\end{align}
By using Gronwall's lemma, we can get that for any $s\in(t-1, t-\frac 12)$, $t\geq 1$,
\begin{align}
 & \|A^\frac12\bar{v}(t)\|^2
 +  \frac{\nu }{2} \int_s^te^{\int_\eta^tC(\|A^\frac12\textbf{v}_1(\tau)\|^4+\|A^\frac12\textbf{v}_2(\tau)\|^4+|y_\delta(\theta_\tau\omega)|^4)\, \d \tau}\|A\bar{v}(\eta) \|^2 \, \d \eta
\nonumber\\
&\quad \leq e^{ \int_s^tC(\|A^\frac12\textbf{v}_1(\tau) \|^4+\|A^\frac12\textbf{v}_2(\tau)\|^4+|y_\delta(\theta_\tau\omega)|^4)\, \d \tau}\|A^\frac12\bar{v}(s)\|^2, \nonumber
\end{align}
and then integrating w.r.t. $s$ over $(t-1,t-\frac 12)$  yields
\begin{align*}
 & \|A^\frac12\bar{v}(t)\|^2
 + \frac{\nu}{2} \int_{t-\frac 12} ^te^{\int_\eta^tC(\|A^\frac12\textbf{v}_1(\tau)\|^4+\|A^\frac12\textbf{v}_2(\tau)\|^4+|y_\delta(\theta_\tau\omega)|^4)\, \d \tau}\|A\bar{v}(\eta) \|^2 \, \d \eta
\nonumber\\
&\quad \leq  2 \int_{t-1} ^{t-\frac 12}   e^{ \int_s^tC(\|A^\frac12\textbf{v}_1(\tau) \|^4+\|A^\frac12\textbf{v}_2(\tau)\|^4+|y_\delta(\theta_\tau\omega)|^4)\, \d \tau}\|A^\frac12\bar{v}(s)\|^2  \, \d s \\
& \quad \leq  2e^{ \int_{t-1}^tC(\|A^\frac12\textbf{v}_1(\tau) \|^4+\|A^\frac12\textbf{v}_2(\tau)\|^4+|y_\delta(\theta_\tau\omega)|^4)\, \d \tau}  \int_{t-1} ^{t-\frac 12} \|A^\frac12\bar{v}(s)\|^2 \, \d s
.
\end{align*}
Applying \eqref{mar9.1},  it yields
\begin{align}
 & \|A^\frac12\bar{v}(t)\|^2
 + \frac{\nu}{2}\int_{t-\frac 12} ^te^{\int_\eta^tC(\|A^\frac12\textbf{v}_1(\tau)\|^4+\|A^\frac12\textbf{v}_2(\tau)\|^4+|y_\delta(\theta_\tau\omega)|^4)\, \d \tau}\|A\bar{v}(\eta) \|^2 \, \d \eta \nonumber  \\
  &\quad  \leq  2e^{ \int_{0}^tC(\|A^\frac12\textbf{v}_1(\tau) \|^4+\|A^\frac12\textbf{v}_2(\tau)\|^4+|y_\delta(\theta_\tau\omega)|^4+1)\, \d \tau}
   \|\bar{v}(0)\|^2
,\quad t\geq 1. \label{mar19.1}
\end{align}

Since the absorbing set $\mathfrak B$ itself belongs to the attraction universe $\D_H$,  it pullback absorbs itself. Then  there exists a random variable $\tau_\omega \geq 1$ such that
\begin{align}\label{mar18.2}
 \phi \big(\tau_\omega,\theta_{-\tau_\omega}\omega, \B(\theta_{-\tau_\omega}\omega)\big)
 \subset \B(\omega), \quad \omega\in \Omega.
\end{align}
We replace $\omega$ with $\theta_{-\tau_\omega} \omega$ in \eqref{mar19.1}  and deduce  the estimate at $t=\tau_\omega$ that
\begin{align}
  \big \|
 A^{\frac 12}  \bar {v}(\tau_\omega,\theta_{-\tau_\omega}\omega,  \bar v(0))  \big \| {^2}
 \leq  2e^{ C\int_{0}^{\tau_\omega}  \sum_{i=1}^2\|A^\frac12\textbf{v}_i(s, \theta_{-\tau_\omega}\omega, v_i(0)) \|^4 \d s+ C\int^0_{-\tau_\omega} (|y_\delta(\theta_s\omega)|^4+1) \d s}
   \|\bar v(0) \|^2 .  \label{mar18.1}
\end{align}

 In order to prove the bound on the right hand side of \eqref{mar18.1}, by \eqref{4.27}, we can deduce that
\begin{align*}
\frac{\d}{\d t}\|A^{\frac12}\textbf{v}\|^2
&\leq    C\left ( 1 +  |y_\delta(\theta_t\omega)|^2 \right)  \|A^{\frac12}\textbf{v}\|^2  +
C \left(1 +|y_\delta(\theta_t\omega)|^6  \right) .
\end{align*}
Since the absorbing set $\mathfrak B $ is bounded and tempered in $H^1$, for  initial values $\textbf{v}(0)$ in   $\mathfrak B(\omega)$, applying  Gronwall's lemma, we can  deduce that
\begin{align*}
 \|A^{\frac12}\textbf{v}(s,\omega, \textbf{v}(0))\|^2
&\leq e^{ \int_0^s C(1+|y_\delta(\theta_\tau\omega)|^2)\, \d \tau}\|A^{\frac12}\textbf{v}( 0)\|^2
\\
&\quad +C\int_0^s e^{ \int_\eta^s C(1+ |y_\delta (\theta_\tau\omega)|^2)\, \d \tau} \left (1 +|y_\delta(\theta_\eta\omega)|^6 \right) \d \eta \\
&\leq e^{ \int_0^s C(1+|y_\delta(\theta_\tau\omega)|^2)\, \d \tau}
\left[ \zeta_2(\omega)+C\int_0^s   \left (1 +|y_\delta(\theta_\eta\omega)|^6 \right) \d \eta \right]
\end{align*}
for any $s\geq 0$ uniformly for $\textbf{v}(0)\in  \mathfrak B(\omega)$,   where $\zeta_2(\omega) $ is the square of the radius of $\mathfrak B(\omega)$ in $H^1$. As a consequence,  for any $s\in (0, \tau_\omega)$ we can deduce that
\begin{align*}
  \sup_{s\in (0, \tau_\omega)} \|A^{\frac12}\textbf{v}(s,\omega, \textbf{v}(0))\|^2
&\leq e^{ \int_0^{\tau_\omega}  C(1+|y_\delta(\theta_\tau\omega)|^2)\, \d \tau}
\left[\zeta_2(\omega)+C\int_0^{\tau_\omega}   \left (1 +|y_\delta(\theta_\eta\omega)|^6 \right) \d \eta \right] ,
\end{align*}
and then, we replace $\omega $ with $ \theta_{-\tau_\omega} \omega$ to get,
\begin{align*}
&  \sup_{s\in (0, \tau_\omega)} \|A^{\frac12}\textbf{v}(s, \theta_{-\tau_\omega} \omega, \textbf{v}(0))\|^2 \\
&\quad \leq e^{ \int^0_{-\tau_\omega}  C(1+|y_\delta(\theta_\tau\omega)|^2)\, \d \tau}
\left[ \zeta_2( \theta_{-\tau_\omega}\omega)+C\int^0_{-\tau_\omega}   \left (1 +|y_\delta(\theta_\eta\omega)|^6 \right) \d \eta \right] =: \zeta_7(\omega) .
\end{align*}
 Then, we have
\begin{align}\label{5.8}
    \int_{0}^{\tau_\omega}  \|A^\frac12\textbf{v}(s, \theta_{-\tau_\omega}\omega, \textbf{v}(0)) \|^4 \, \d s  \leq  \tau_\omega |\zeta_7(\omega)|^2
\end{align}
uniformly for any  $\textbf{v}(0) $ in $  \mathfrak B( \theta_{-\tau_\omega} \omega)$ and thus  for \eqref{mar18.1} we can get that
\begin{align*}
  \|
 A^{\frac 12} \bar {v}(\tau_\omega,\theta_{-\tau_\omega}\omega, \bar v(0))  \|^2
&  \leq  2e^{ C \tau_\omega  |\zeta_7( \omega)|^2 + C\int^0_{-\tau_\omega} (|y_\delta(\theta_s\omega)|^4+1) \d s}
   \| \bar v(0)  \|^2  \\
   &\leq 2e^{ C \tau_\omega  |\zeta_7( \omega)|^2 }
   \| \bar v(0)  \|^2.
\end{align*}
Hence, the random variable defined by
\begin{align*}
  L_2(\mathfrak B, \omega):=  2e^{ C \tau_\omega  |\zeta_7( \omega)|^2 }
,\quad \omega\in \Omega.
\end{align*}
This completes the proof of Lemma \ref{lemma5.2}.
\end{proof}

Nextly, we will prove  the   $(H,H^1)$-smoothing for initial values from any tempered set $\mathfrak D$ in $H$.
\begin{lemma} \label{lem:H1}
 Let Assumption \ref{assum} hold and $f\in H$. For   any tempered set   $\mathfrak D\in \D_H$ there exist random variables $ T_1(\mathfrak D ,\cdot )  $ and $L_3(\mathfrak D,\cdot) $ such that any
 two solutions $\textbf{v}_1$ and $\textbf{v}_2$ of  the random anisotropic  NS equations \eqref{2.2} driven by colored noise corresponding to initial values  $\textbf{v}_{1,0},$ $\textbf{v}_{2,0}$ in $ \mathfrak D( \theta_{-T_1(\mathfrak D,\omega)}\omega),$   respectively,  satisfy
\begin{align}\label{mar19.4}
  \left \| \textbf{v}_1  (T_1 ,\theta_{-T_1}\omega,  \textbf{v}_{1,0} )
 - \textbf{v}_2   ( T_1 ,\theta_{-T_1}\omega, \textbf{v}_{2,0} ) \right \|^2_{H^1}
 \leq L_3({\mathfrak D}, \omega) \|\textbf{v}_{1,0}-\textbf{v}_{2,0}\|^2  ,
\end{align}
where $T_1= T_1(\mathfrak D, \omega)$, $\omega\in \Omega$. In addition, $ v (T_1 ,\theta_{-T_1}\omega,  v_{0}) \in \mathfrak B(\omega)$ for any $v_0\in \mathfrak D(\theta_{-T_1}\omega) $.
\end{lemma}

\begin{proof}
By virtue of ${\mathfrak D} \in \D_H$   is pullback absorbed by $\mathfrak B$ and \eqref{mar19.2}, then  there exists  a random variable $  t_{\mathfrak D} (\cdot) $ such that
\begin{align}\label{mar19.5}
 \phi \left (  t_{\mathfrak D}(\omega) , \theta_{-  t_{\mathfrak D}(\omega)}\omega, \mathfrak D(\theta_{-  t_{\mathfrak D}(\omega)}\omega) \right) \subset \mathfrak B(\omega), \quad \omega\in \Omega.
\end{align}
Hence,    for any $ \omega\in \Omega $   we can deduce that
\begin{align*}
&   \left \|  \bar v  \!\left( \tau_\omega+   t_{\mathfrak D}(\theta_{-\tau_\omega} \omega)   , \, \theta_{-\tau_\omega  - t_{\mathfrak D}(\theta_{-\tau_\omega} \omega)}
 \omega, \, \bar v(0)\right)
   \right \|^2_{H^1}
    \nonumber \\
         &\quad = \left \| \bar v  \! \left( \tau_\omega  , \, \theta_{-\tau_\omega   }
 \omega, \,  \bar v \big(   t_{\mathfrak D}(\theta_{-\tau_\omega} \omega) ,\,  \theta_{-\tau_\omega- t_{\mathfrak D}(\theta_{-\tau_\omega} \omega)  }
 \omega,  \, \bar v(0)\big) \right)
   \right \|^2_{H^1} \nonumber
   \\
 &
  \quad \leq  L_2  (\mathfrak B,   \omega  )   \left \|  \bar v \!  \left( t_{\mathfrak D}(\theta_{-\tau_\omega} \omega),\,  \theta_{-\tau_\omega- t_{\mathfrak D}(\theta_{-\tau_\omega} \omega)    }
 \omega, \,  \bar v(0) \right)
   \right \|^2  \quad \text{(by \eqref{mar18.9})} \nonumber \\
 &\quad \leq     L_2  (\mathfrak B,   \omega  )   L_1  \big (\mathfrak D, \theta_{-\tau_\omega}\omega  \big)  \|\bar{v}(0)\|^2  \quad \text{(by \eqref{mar18.8})}
 \nonumber
\end{align*}
for any $\textbf{v}_{1,0}$, $\textbf{v}_{2,0}\in \mathfrak D \left (  \theta_{-\tau_\omega  - t_{\mathfrak D}(\theta_{-\tau_\omega} \omega)} \omega \right)$.  Let
\begin{align*}
 & T_1({\mathfrak D,\omega}) := \tau_\omega + t_{\mathfrak D}(\theta_{-\tau_\omega} \omega) ,  \\
 & L_3(\mathfrak D,\omega):=  L_2 \left (\mathfrak B,    \omega \right )   L_1 \left(\mathfrak D, \theta_{-\tau_\omega}\omega  \right) ,
\end{align*}
This completes the proof of \eqref{mar19.4}.

By the definition \eqref{mar19.5}  of $ t_{\mathfrak D}(\theta_{-\tau_\omega} \omega)$, we can deduce that, for any $\textbf{v}_0\in \mathfrak D(\theta_{-T_1}\omega) $,
 \begin{align*}
 y:= \textbf{v} \left(    t_{\mathfrak D}(\theta_{-\tau_\omega} \omega)   , \, \theta_{- t_{\mathfrak D}(\theta_{-\tau_\omega} \omega)} \circ \theta_{-\tau_\omega}
 \omega, \,  \textbf{v}(0)\right) \in \mathfrak B(\theta_{-\tau_\omega}\omega).
\end{align*}
Moreover, from the definition \eqref{mar18.2}  of $\tau_\omega$ we can deduce that
\begin{align*}
\textbf{ v}(\tau_\omega, \theta_{-\tau_\omega} \omega, y) \in \mathfrak B(\omega).
\end{align*}
Hence, we also have
\begin{align*}
    \textbf{v}\left (T_1 ,\theta_{-T_1}\omega,  \textbf{v}_{0}\right)
    & =  \textbf{v} \left( \tau_\omega+   t_{\mathfrak D}(\theta_{-\tau_\omega} \omega)   , \, \theta_{- t_{\mathfrak D}(\theta_{-\tau_\omega} \omega)}(\theta_{-\tau_\omega}
 \omega), \,  \textbf{v}(0)\right)\\
 &= \textbf{v}(\tau_\omega, \theta_{-\tau_\omega} \omega, y)
  \in \mathfrak B(\omega)
\end{align*}
 for any $v_0\in \mathfrak D(\theta_{-T_1}\omega) $, as desired. This completes the proof of Lemma \ref{lem:H1}.
 \end{proof}

\subsection{$(H , H^2)$-smoothing}

In this subsection, we will prove the $(H , H^2)$-smoothing property. For ease of analysis we  restrict again ourselves on the absorbing set $\mathfrak B $ first. We  first introduce  the following main estimate.

\begin{lemma}  Let Assumption \ref{assum} hold and $f\in H$. We say that
the solutions $\textbf{v}$ of the random anisotropic  NS equations \eqref{2.2} driven by colored noise corresponding to initial values  $\textbf{v}(0)$  satisfy the following estimate
\begin{align}
\sup_{s\in [t-1,t]}
 \|A\textbf{v}(s,\theta_{-t}\omega,  \textbf{v}(0)) \|^2
  \leq     2\rho(\omega) +2\|\mathcal A_0\|_{H^2}^2,
 \quad  t\geq T_{\mathfrak B} (\omega)+ 1,\label{mar9.4}
\end{align}
whenever  $\textbf{v}(0) \in\mathfrak B (\theta_{-t} \omega)$, where  $T_{\mathfrak B} $ is the random variable defined in   \eqref{timeB}.
\end{lemma}
\begin{proof}
For any $ t\geq T_{\mathfrak B} (\omega)+1$ and $s\in (t-1,t)$ so that $s\geq T_{\mathfrak B} (\omega)$, by
 Lemma \ref{lemma4.4} we can get that
\begin{align*}
 \|A\textbf{v}(s,\theta_{-t}\omega,  \textbf{v}(0)) \|^2    =
 \|A\textbf{v}(s,\theta_{-s-(t-s)}\omega,  \textbf{v}(0)) \|^2
 \leq     2\rho(\omega) +2\|\mathcal A_0\|_{H^2}^2
\end{align*}
uniformly for $s\in [t-1,t]$ and $v(0) \in\mathfrak B (\theta_{-t} \omega)$.
\end{proof}

\begin{lemma}[$(H^1, H^2)$-smoothing on $\mathfrak B$]\label{lemma5.1}  Let Assumption \ref{assum} hold and $f\in H$. There  exist  random variables $T_\omega$ and $ L_4(\mathfrak B, \omega )$ such that   any two solutions $\textbf{v}_1$ and $\textbf{v}_2$ of the random anisotropic  NS equations \eqref{2.2} driven by colored noise corresponding to initial values  $\textbf{v}_{1,0},$ $\textbf{v}_{2,0}$ in $\mathfrak{B}(\theta_{-T_\omega}\omega)$, respectively,   satisfy
\begin{align*}
 \big \|\textbf{v}_1(T_\omega,\theta_{-T_\omega}\omega,\textbf{v}_{1,0})-\textbf{v}_2(T_\omega,\theta_{-T_\omega}\omega,\textbf{v}_{2,0}) \big \|{^2_{H^2}}\leq  L_4(\mathfrak B, \omega)\|\textbf{v}_{1,0}-\textbf{v}_{2,0}\|^2_{H^1}, \quad \omega\in \Omega.
\end{align*}
\end{lemma}

\begin{proof} Taking the inner product of the system   \eqref{5.1} with $A^2\bar{v}$ in $H$,  by integration by parts, we can deduce that
\begin{align}\label{5.5}
&\frac{1}{2}\frac{\d}{\d t}\|A\bar{v}\|^2-\nu \int_{\mathbb{T}^2}\partial_{yy}\Delta \bar{v}_1\Delta\bar{v}_1dx-\nu \int_{\mathbb{T}^2}\partial_{xx}\Delta \bar{v}_2\Delta \bar{v}_2dx\nonumber\\
&=- \big (B(\bar{v},\textbf{v}_1+hy_\delta(\theta_t\omega)),\, A^2\bar{v}  \big )
- \big (B(\textbf{v}_2+hy_\delta(\theta_t\omega),\bar{v}),\, A^2\bar{v} \big ).
\end{align}
By direct computations and $\nabla\cdot \bar{v}=0$ and \eqref{3.9}, we also have
\begin{align}
&-\nu \int_{\mathbb{T}^2}\partial_{yy}\Delta \bar{v}_1\Delta\bar{v}_1dx-\nu \int_{\mathbb{T}^2}\partial_{xx}\Delta \bar{v}_2\Delta \bar{v}_2dx\nonumber\\
&= \nu \int_{\mathbb{T}^2}(\partial_{y}\Delta \bar{v}_1)^2dx
+\nu \int_{\mathbb{T}^2}(\partial_{x}\Delta \bar{v}_2)^2dx
\nonumber\\
&\geq\frac12\nu\|A^{\frac32} \bar{v}\|^2.
\end{align}
By using the H\"{o}lder inequality, the Gagliardo-Nirenberg inequality and the Young inequality, we can get that
\begin{align}\label{5.6}
 \big| \big(B(\bar{v},\textbf{v}_1+hy_\delta(\theta_t\omega)),\, A^2\bar{v}\big)
 \big| &\leq \|A^\frac12\bar{v}\|_{L^\infty}\|\nabla(\textbf{v}_1+hy_\delta(\theta_t\omega))\|\|A^\frac32\bar{v}\|  \nonumber\\
&\quad
 +\|\bar{v}\|_{L^\infty}\|A(\textbf{v}_1+hy_\delta(\theta_t\omega))\|\|A^\frac32\bar{v}\|
 \nonumber\\
&\leq C\|A\bar{v}\|^\frac12\|A^\frac32\bar{v}\|^\frac32
\left (\|A^\frac12\textbf{v}_1\|+\|A^\frac12h\||y_\delta(\theta_t\omega)| \right )\nonumber\\
&\quad +C\|\bar{v}\|^\frac12\|A\bar{v}\|^\frac12\|A^\frac32\bar{v}\| \big(\|A\textbf{v}_1\|+\|Ah\||y_\delta(\theta_t\omega)| \big)\nonumber\\
&\leq \frac  \nu 8\|A^\frac32\bar{v}\|^2+C\|A\bar{v}\|^2 \left (\|A^\frac12\textbf{v}_1\|^4+|y_\delta(\theta_t\omega)|^4 \right)\nonumber\\
&\quad +C\|A^\frac12\bar{v}\|^2 \left (\|A\textbf{v}_1\|^4+|y_\delta(\theta_t\omega)|^4 \right),
\end{align}
and 
\begin{align}\label{5.7}
\big| \big(B(\textbf{v}_2+hy_\delta(\theta_t\omega),\bar{v}),\, A^2\bar{v}\big)\big|
&\leq \|A^\frac12(\textbf{v}_2+hy_\delta(\theta_t\omega))\|\|\nabla\bar{v}\|_{L^\infty}\|A^\frac32\bar{v}\| \nonumber\\
&\quad
 +\|\textbf{v}_2+hy_\delta(\theta_t\omega)\|_{L^4}\|A\bar{v}\|_{L^4}\|A^\frac32\bar{v}\|\nonumber\\
&\leq C\|A\bar{v}\|^\frac12\|A^\frac32\bar{v}\|^\frac32
\left (\|A^\frac12\textbf{v}_2\|+\|A^\frac12h\||y_\delta(\theta_t\omega)| \right)\nonumber\\
&\quad +C \|A^\frac12(\textbf{v}_2+hy_\delta(\theta_t\omega)
)\|\|A\bar{v}\|^\frac12\|A^\frac32\bar{v}\|^\frac32 \nonumber\\
&\leq \frac  \nu 8\|A^\frac32\bar{v}\|^2+C\|A\bar{v}\|^2 \left (\|A^\frac12\textbf{v}_2\|^4+|y_\delta(\theta_t\omega)|^4 \right).
\end{align}
Inserting \eqref{5.6} and \eqref{5.7} into \eqref{5.5} yields
\begin{align}\label{7.1}
&\frac{\d}{\d t}\|A\bar{v}\|^2+\frac12\nu\|A^{\frac32} \bar{v}\|^2\nonumber\\
&\leq C\|A\bar{v}\|^2 \left ( \sum_{j=1}^2 \|A^\frac12\textbf{v}_j\|^4 +|y_\delta(\theta_t\omega)|^4 \right )   +C\|A^\frac12\bar{v}\|^2 \left (\|A\textbf{v}_1\|^4+|y_\delta(\theta_t\omega)|^4 \right).
\end{align}
For any $s\in(t-1,t)$ and $t\geq 1$, applying Gronwall's lemma, then  we can get that
\begin{align*}
&\|A\bar{v}(t)\|^2
-e^{\int_s^tC( \sum_{i=1}^2 \|A^\frac12\textbf{v}_i\|^4 +|y_\delta(\theta_\tau\omega)|^4)\, \d \tau}\|A\bar{v}(s)\|^2\nonumber\\
&\quad \leq \int_s^tCe^{\int_\eta^tC( \sum_{i=1}^2 \|A^\frac12\textbf{v}_i\|^4+|y_\delta(\theta_\tau\omega)|^4)\, \d \tau}
\|A^\frac12\bar{v}\|^2 \left(\|A\textbf{v}_1\|^4+|y_\delta(\theta_\eta\omega)|^4\right) \d \eta\nonumber\\
&\quad \leq Ce^{\int_{t-1}^tC( \sum_{i=1}^2 \|A^\frac12\textbf{v}_i\|^4+|y_\delta(\theta_\tau\omega)|^4)\, \d \tau}
\int_{t-1}^t\|A^\frac12\bar{v}\|^2\left(\|A\textbf{v}_1\|^4+|y_\delta(\theta_\eta\omega)|^4 \right) \d \eta. \nonumber
\end{align*}
Integrating w.r.t. $s$ on $(t-1,t)$ yields
\begin{align}\label{5.9}
&\|A\bar{v}(t)\|^2
-\int_{t-1}^te^{\int_s^tC( \sum_{i=1}^2 \|A^\frac12\textbf{v}_i\|^4+|y_\delta(\theta_\tau\omega)|^4)\, \d \tau}\|A\bar{v}(s)\|^2\, \d s\nonumber\\
&\quad \leq Ce^{\int_{t-1}^tC( \sum_{i=1}^2 \|A^\frac12\textbf{v}_i\|^4+|y_\delta(\theta_\tau\omega)|^4)\, \d \tau}
\int_{t-1}^t\|A^\frac12\bar{v}\|^2\left (\|A\textbf{v}_1\|^4+|y_\delta(\theta_\eta\omega)|^4\right) \d \eta.
\end{align}

Since our initial values $\textbf{v}_{1,0}$ and $\textbf{v}_{2,0}$   belong to the $H^1$ random  absorbing set  $\mathfrak B$,  applying Gronwall's lemma to \eqref{mar9.2}, then we can deduce    that
\begin{align}
 & \|A^\frac12\bar{v}(t)\|^2 + \frac{\nu}{2} \int_0^te^{\int_s^tC( \sum_{i=1}^2 \|A^\frac12\textbf{v}_i\|^4+|y_\delta(\theta_\tau\omega)|^4)\, \d \tau}\|A\bar{v}(s) \|^2\, \d s
\nonumber\\
&\quad \leq e^{\int_0^tC( \sum_{i=1}^2 \|A^\frac12\textbf{v}_i\|^4+|y_\delta(\theta_\tau\omega)|^4)\, \d \tau}\|A^\frac12\bar{v}(0)\|^2,\quad t\geq 1. \nonumber
\end{align}
Hence, we can get the  following inequalities
\begin{align}
& \frac{\nu}{2} \int_{t-1}^te^{\int_s^tC( \sum_{i=1}^2 \|A^\frac12\textbf{v}_i\|^4+|y_\delta(\theta_\tau\omega)|^4)\, \d \tau}\|A\bar{v}(s)\|^2\, \d s\nonumber\\
&\quad \leq e^{\int_{0}^tC( \sum_{i=1}^2 \|A^\frac12\textbf{v}_i\|^4+|y_\delta(\theta_\tau \omega)|^4)\, \d \tau}\|A^\frac12\bar{v}(0)\|^2, \nonumber
\end{align}
and
\begin{align}
&
\int_{t-1}^t\|A^\frac12\bar{v}\|^2 \left (\|A\textbf{v}_1\|^4+|y_\delta(\theta_\eta\omega)|^4\right) \d \eta\nonumber\\
&\quad \leq e^{\int_0^tC( \sum_{i=1}^2 \|A^\frac12\textbf{v}_i\|^4+|y_\delta(\theta_\tau\omega)|^4)\, \d \tau}\|A^\frac12\bar{v}(0)\|^2
\int_{t-1}^t \!  \left (\|A\textbf{v}_1\|^4+|y_\delta(\theta_\eta\omega)|^4\right) \d \eta.\nonumber
\end{align}
Inserting these inequalities  into \eqref{5.9}, we also can deduce that
\begin{align}
 \|A\bar{v}(t)\|^2
&\leq Ce^{\int_0^tC( \sum_{i=1}^2 \|A^\frac12\textbf{v}_i\|^4+|y_\delta(\theta_\tau\omega)|^4)\, \d \tau}\|A^\frac12\bar{v}(0)\|^2
\int_{t-1}^t  \! \left (1+\|A\textbf{v}_1\|^4+|y_\delta(\theta_\eta\omega)|^4\right)  \d \eta   \nonumber
\end{align}
for any $t\geq 1$.
We replace $\omega$ by $\theta_{-t}\omega$ to  get  for any $ t\geq 1$
\begin{align} \label{mar11.1}
 \|A\bar{v}(t,\theta_{-t}\omega,\bar{v}(0))\|^2
&\leq Ce^{\int_0^tC( \sum_{i=1}^2\|A^\frac12\textbf{v}_i(\tau,\theta_{-t}\omega,\textbf{v}_i(0))\|^4
+|y_\delta(\theta_{\tau-t}\omega)|^4)\, \d \tau}\|A^\frac12\bar{v}(0)\|^2\nonumber\\
&\quad  \times \int_{t-1}^t  \! \left ( \|A\textbf{v}_1(\eta,\theta_{-t}\omega,\textbf{v}_{1,0})\|^4+|y_\delta(\theta_{\eta-t}\omega)|^4+1\right) \d \eta.
\end{align}

By \eqref{mar9.4}, we can deduce that
\begin{align*}
\sup_{s\in [t-1,t]}
 \|A\textbf{v}(s,\theta_{-t}\omega,  \textbf{v}(0)) \|^2
  \leq     2\rho(\omega) +2\|\mathcal A_0\|_{H^2}^2,
 \quad  t\geq T_{\mathfrak B} (\omega)+1.
\end{align*}
As a consequence,  let
\begin{align*}
T_\omega :=T_{\mathfrak B} (\omega)+1, \quad \omega\in \Omega,
\end{align*}
we can get that
\begin{align}
 & \int_{T_\omega -1}^{T_\omega }  \left ( \|A\textbf{v}_1(\eta,\theta_{-T_\omega }\omega,\textbf{v}_{1,0})\|^4+|y_\delta(\theta_{\eta-T_\omega }\omega)|^4\right)  \d \eta  \nonumber\\
 &\quad
 \leq 4 \left( \rho(\omega) +\|\mathcal A_0\|_{H^2}^2 \right)^2 + \int_{-1}^0 |y_\delta(\theta_{\eta}\omega)|^4 \, \d \eta =: \zeta_8(\omega) . \label{mar11.2}
 \end{align}

On the other hand, by \eqref{4.27}, it yields
\begin{align*}
\frac{\d}{\d t}\|A^{\frac12}\textbf{v}\|^2 \leq C \left ( 1+ |y_\delta(\theta_t\omega)|^2 \right )\|A^{\frac12}\textbf{v}\|^2 + C \left (1 +|y_\delta(\theta_t\omega)|^6 \right) .
\end{align*}
we apply Gronwall's lemma to deduce
\begin{align*}
 \|A^{\frac 12} \textbf{v}(t,\omega, \textbf{v}(0))\|^2
  & \leq  e^{ \int^t_0 C(1+|y_\delta(\theta_\tau\omega)|^2) \, \d \tau} \|A^{\frac 12} \textbf{v}(0)\|^2  \\
 &\quad
 +\int_0^t Ce^{ \int^t_s C(1+|y_\delta(\theta_\tau\omega)|^2) \, \d \tau} \left (1+ |y_\delta(\theta_s\omega)|^6 \right)  \d s \\
 &\leq  e^{ \int^t_0 C(1+|y_\delta(\theta_\tau\omega)|^2) \, \d \tau} \left[ \|A^{\frac 12} \textbf{v}(0)\|^2
 +\int_0^t  C \left(1+ |y_\delta(\theta_s\omega)|^6 \right)  \d s \right]  .
\end{align*}
Since
\begin{align*}
\int_0^t  C \left (1+ |y_\delta(\theta_s\omega)|^6 \right)  \d s \leq e^{ \int^t_0 C(1+|y_\delta(\theta_\tau\omega)|^6) \, \d \tau},
\end{align*}
then the estimate is given by
\begin{align*}
 \|A^{\frac 12} \textbf{v}(t,\omega, \textbf{v}(0))\|^2
   \leq  \left( \|A^{\frac 12} \textbf{v}(0)\|^2 +1 \right) e^{ \int^t_0 C(1+|y_\delta(\theta_\tau\omega)|^6) \, \d \tau} ,
   \quad t>0.
\end{align*}
Hence,  for   $\textbf{v}(0) \in \mathfrak B(\theta_{-t}\omega)$ we can deduce that
 \begin{align*}
 \int_0^t
 \|A^{\frac 12} \textbf{v}(\eta, \theta_{-t} \omega, \textbf{v}(0))\|^4 \, \d \eta
 &   \leq  \left( \| \mathfrak B  (\theta_{-t}\omega)\|^2_{H^1} +1 \right)^2 \int_0^t e^{ \int^\eta_0 C(1+|y_\delta(\theta_{\tau-t} \omega)|^6) \, \d \tau}  \, \d \eta \nonumber \\
 &   \leq  \big( \zeta_2(\theta_{-t}\omega)  +C \big)^2  e^{ Ct+C \int^0_{-t}  |y_\delta(\theta_{\tau} \omega)|^6 \, \d \tau}   ,
   \quad t>0  , \nonumber
\end{align*}
and, particularly for $t=T_\omega$,
 \begin{align}  \label{mar11.3}
 & \int_0^{T_\omega}
 \|A^{\frac 12} \textbf{v}(\eta, \theta_{-{T_\omega}} \omega, \textbf{v}(0))\|^4 \, \d \eta \nonumber \\
 &   \quad  \leq  \big( \zeta_2(\theta_{-{T_\omega}}\omega)  +C \big)^2  e^{ C {T_\omega}+C \int^0_{-{T_\omega}}  |y_\delta(\theta_{\tau} \omega)|^6 \, \d \tau} =:\zeta_9(\omega).
\end{align}

Finally, inserting     \eqref{mar11.2}  and \eqref{mar11.3}  into \eqref{mar11.1}  we can get at $t=T_\omega$ that
\begin{align}
 \|A\bar{v}(T_\omega ,\theta_{-T_\omega }\omega,\bar{v}(0))\|^2
&\leq Ce^{ C \zeta_9(\omega) } \zeta_8(\omega) \|A^\frac12\bar{v}(0)\|^2. \nonumber
\end{align}
Hence, let
\begin{align*}
 L_4 (\mathfrak B, \omega) :=   Ce^{ C \zeta_9(\omega) } \zeta_8(\omega) ,\quad \omega\in \Omega. \qedhere
\end{align*}
This completes the proof of Lemma \ref{lemma5.1}.
\end{proof}

In this section, we next are ready to introduce the following main result.
\begin{theorem}[$(H, H^2)$-smoothing] \label{theorem5.6}
 Let Assumption \ref{assum} hold and $f\in H$. For   any tempered set $\mathfrak D \in \D_H$, there
  are  random variables $T_{{\mathfrak D}} (\cdot)  $   and $ L_{\mathfrak D}(\cdot )$ such that  any two solutions $\textbf{v}_1$ and $\textbf{v}_2$ of the random anisotropic  NS equations \eqref{2.2} driven by colored noise corresponding to initial values   $\textbf{v}_{1,0},$ $ \textbf{v}_{2,0}$ in $\mathfrak D \left (\theta_{-T_{\mathfrak D}(\omega)}\omega \right)$, respectively, satisfy
  \begin{align}
  &
  \big\|\textbf{v}_1 \!  \big (T_{\mathfrak D}(\omega),\theta_{-T_{\mathfrak D}(\omega)}\omega,\textbf{v}_{1,0}\big)
  -\textbf{v}_2 \big (T_{\mathfrak D}(\omega),\theta_{-T_{\mathfrak D}(\omega)}\omega,\textbf{v}_{2,0} \big) \big \|{^2_{H^2}}  \notag \\[0.8ex]
&\quad \leq  L_{\mathfrak D} (\omega)\|\textbf{v}_{1,0}-\textbf{v}_{2,0}\|^2,\quad \omega\in \Omega.\label{sep5.1}
\end{align}
\end{theorem}
\begin{proof}
By virtue of Lemma \ref{lem:H1}, for any the tempered set   $\mathfrak D\in \D_H$ there are random variables $ T_1(\mathfrak D ,\cdot )  $ and $L_3(\mathfrak D,\cdot) $ such that
  \begin{align} \label{mar19.6}
  \left \| \bar v\left (T_1 ,\theta_{-T_1}\omega,  \bar v(0)\right)
  \right \|^2_{H^1}
 \leq L_3({\mathfrak D}, \omega) \|\bar {v} (0)\|^2  ,
  \end{align}
where $T_1= T_1(\mathfrak D, \omega)$. Moreover, $ \textbf{v}\left (T_1 ,\theta_{-T_1}\omega,  \textbf{v}_{0}\right) \in \mathfrak B(\omega)$ for any $\textbf{v}_0\in \mathfrak D(\theta_{-T_1}\omega) $, $\omega\in \Omega$.
Hence, by Lemma \ref{lemma5.1} and \eqref{mar19.6}, we can deduce that
 \begin{align*}
  \left \| \bar v \big( T_\omega +T_1, \theta_{-T_\omega-T_1}
 \omega, \bar v(0) \big)
   \right \|^2_{H^2}
 &  =  \left \| \bar v \big( T_\omega  , \theta_{-T_\omega }
 \omega, \bar v  ( T_1, \theta_{-T_\omega-T_1}
 \omega, \bar v(0)  )\big)
   \right \|^2_{H^2}  \\
   &\leq  L_4(\mathfrak B, \omega) \left \|  \bar v \big  ( T_1, \theta_{-T_\omega-T_1}
 \omega, \bar v(0)  \big)
   \right \|^2_{H^1}  \  \text{(by Lemma \ref{lemma5.1})}\\
   &\leq L_4(\mathfrak B,\omega)  L_3 \big (\mathfrak D, \theta_{-T_\omega}\omega \big) \|\bar v(0)\|^2  \  \text{(by \eqref{mar19.6})} ,
 \end{align*}
where $T_1=T_1(\theta_{-T_\omega}  \omega)$,
uniformly for any $\textbf{v}_{1,0},$ $ \textbf{v}_{2,0}\in \B(\theta_{-T_\omega-T_1(\theta_{-T_\omega} \omega)}\omega)$, $  \omega\in \Omega$. Thence,
let  the random variables is given by
 \begin{align*}
  & T_{\mathfrak D}( \omega) := T_\omega +T_1(\theta_{-T_\omega} \omega) , \\
  &L_{\mathfrak D}(\omega) :=  L_4(\mathfrak B,\omega)   L_3 \big (\mathfrak D, \theta_{-T_\omega}\omega \big).
 \end{align*}
This completes the proof of Lemma \ref{theorem5.6}.
  \end{proof}

\section{The $(H,H^2)$-random attractor}\label{sec6}
In this section,  we will prove that the  random attractor  $\mathcal A$ of the random anisotropic  NS equations \eqref{2.2} driven by colored noise is in fact a finite-dimensional  $(H,H^2)$-random attractor. More precisely, Theorem \ref{theorem4.6} is now strengthened to as stated below.

\begin{theorem} \label{theorem5.1}
 Let Assumption \ref{assum} hold and $f\in H$. Then the  RDS $\phi$ generated by the the random anisotropic  NS equations \eqref{2.2} driven by colored noise  has  a tempered $(H,H^2)$-random attractor $\mathcal A $. In addition, $\mathcal A$ has finite fractal dimension in $H^2$: there exists   a constant $ d>0$ such that
\[
d_f^{H^2} \! \big ( \A(\omega) \big)  \leq d, \quad \omega \in \Omega.
\]
\end{theorem}
\begin{proof}
 By Theorem \ref{theorem4.1},  we have also constructed an $H^2$ random absorbing set $\mathfrak B_{H^2}$, which is tempered and closed in $H^2$.  By the  $(H,H^2)$-smoothing property \eqref{sep5.1},   we have proved the $(H,H^2)$-asymptotic compactness of $\phi$. Hence,  by Lemma \ref{lem:cui18} we can show that $\A$ is indeed an $(H,H^2)$-random attractor of $\phi$.
 Since the fractal dimension in $H$ of  $\A$ is finite  (see the Theorem \ref{theorem4.6}),    by the $(H,H^2)$-smoothing property \eqref{sep5.1}  and Lemma 5 in \cite{cui}, the   finite-dimensionality  in $H^2$  is proved. This completes the proof of Theorem \ref{theorem5.1}.
\end{proof}

 \begin{remark}
 The $(H,H^2)$-smoothing property \eqref{sep5.1}  and the finite fractal dimension in $H^2$ of the global attractor    are new even for     deterministic  anisotropic  NS equations.
 \end{remark}

\section{The $(H,H^3)$-random attractor}\label{sec7}
In this section, we first prove an  $H^3$ random absorbing set of system  \eqref{2.2} as $f\in H^2$. Nextly, we will prove the ($H,H^3$)-smoothing effect  of the random anisotropic  NS equations \eqref{2.2} driven by colored noise. This effect  is essentially a local $(H, H^3)$-Lipschitz continuity in initial values. Finally,   we will prove that the  random attractor  $\mathcal A$ of the random anisotropic  NS equations \eqref{2.2} driven by colored noise is in fact a finite-dimensional  $(H,H^3)$-random attractor.

\begin{lemma}\label{lemma7.1} [$H^2$ bound]
Let  $f\in H^1$ and Assumption \ref{assum1} hold.
Then there exists a random variable $ T_{\mathfrak B} (\omega)>T_1(\omega)$ such that,  for any $t\geq T_{\mathfrak B} (\omega)$,
\begin{align*}
\big\|A\textbf{v}(t,\theta_{-t}\omega,\textbf{v}(0))  \big\|^2 + \int_{t-\frac 12 }^t \|A^{\frac32}\textbf{v}(s,\theta_{-t}\omega,\textbf{v}(0))\|^2  \ \d s
\leq  \zeta_{10}(\omega),
\end{align*}
where $\zeta_{10}(\omega)$  is a tempered random variable given by \eqref{6.17} such that $\zeta_{10}(\omega) > \zeta_1(\omega)\geq 1$, $ \omega\in\Omega$.
\end{lemma}
\begin{proof}
Taking the inner product of the first equation   \eqref{2.2} with $A^2\textbf{v}$ in $H$, by integration by parts yields
\begin{align}\label{6.1}
&\frac12\frac{\d}{\d t}\|A\textbf{v}\|^2-\nu \int_{\mathbb{T}^2}\partial_{yy}\Delta v_1\Delta v_1dx-\nu \int_{\mathbb{T}^2}\partial_{xx}\Delta v_2\Delta v_2dx\nonumber\\
&=-\big(B(\textbf{v}+hy_\delta(\vartheta_t\omega)), A^2\textbf{v} \big)+(f,A^2\textbf{v})\nonumber\\
& +(-\nu A_1hy_\delta(\vartheta_t\omega),A^2\textbf{v})+(hy_\delta(\vartheta_t\omega),A^2\textbf{v})\nonumber\\
&=:K_1(t)+K_2(t)+K_3(t)+K_4(t).
\end{align}
By direct computations and $\nabla\cdot \textbf{v}=0$ and \eqref{3.9}, it also yields
\begin{align}
&-\nu \int_{\mathbb{T}^2}\partial_{yy}\Delta v_1\Delta v_1dx-\nu \int_{\mathbb{T}^2}\partial_{xx}\Delta v_2\Delta v_2dx\nonumber\\
&= \nu \int_{\mathbb{T}^2}(\partial_{y}\Delta v_1)^2dx
+\nu \int_{\mathbb{T}^2}(\partial_{x}\Delta v_2)^2dx
\nonumber\\
&\geq\frac\nu2\|A^{\frac32} \textbf{v}\|^2.
\end{align}
For the term $K_1(t)$,  by integration by parts, by using the H\"{o}lder inequality, and the Gagliardo-Nirenberg inequality and the Young inequality, we can get that for any $h\in H^3$
\begin{align}\label{6.2}
K_1(t)&\leq C\|A^{\frac{1}{2}}(\textbf{v}+hy_\delta(\vartheta_t\omega))\|_{L^4}\|\nabla(\textbf{v}+hy_\delta(\vartheta_t\omega))\|_{L^4}
\|A^{\frac{3}{2}}\textbf{v}\|\nonumber\\
&+C\|\textbf{v}+hy_\delta(\vartheta_t\omega)\|_{L^4}\|A(\textbf{v}+hy_\delta(\vartheta_t\omega))\|_{L^4}
\|A^{\frac{3}{2}}\textbf{v}\|\nonumber\\
&\leq C\|A^{\frac{1}{2}}(\textbf{v}+hy_\delta(\vartheta_t\omega))\|\|A(\textbf{v}+hy_\delta(\vartheta_t\omega))\|
\|A^{\frac{3}{2}}\textbf{v}\|\nonumber\\
&+C\|A^{\frac{1}{2}}(\textbf{v}+hy_\delta(\vartheta_t\omega))\|\|A(\textbf{v}+hy_\delta(\vartheta_t\omega))\|^\frac12
(\|A^{\frac{3}{2}}\textbf{v}\|^\frac12+\|A^{\frac{3}{2}}hy_\delta(\vartheta_t\omega)\|^\frac12)\|A^{\frac{3}{2}}\textbf{v}\|\nonumber\\
&\leq \frac\nu8\|A^{\frac{3}{2}}\textbf{v}\|^2
+C\|A^{\frac{1}{2}}(\textbf{v}+hy_\delta(\vartheta_t\omega))\|^2\|A(\textbf{v}+hy_\delta(\vartheta_t\omega))\|^2\nonumber\\
&+C\|A^{\frac{1}{2}}(\textbf{v}+hy_\delta(\vartheta_t\omega))\|^4\|A(\textbf{v}+hy_\delta(\vartheta_t\omega))\|^2\nonumber\\
&+C\|A^{\frac{1}{2}}(\textbf{v}+hy_\delta(\vartheta_t\omega))\|^2\|A(\textbf{v}+hy_\delta(\vartheta_t\omega))\|
\|A^{\frac{3}{2}}hy_\delta(\vartheta_t\omega)\|\nonumber\\
&\leq \frac\nu8\|A^{\frac{3}{2}}\textbf{v}\|^2
+C(\|A^{\frac{1}{2}}\textbf{v}\|^2+\|A^{\frac{1}{2}}hy_\delta(\vartheta_t\omega)\|^2
+\|A^{\frac{1}{2}}\textbf{v}\|^4+\|A^{\frac{1}{2}}hy_\delta(\vartheta_t\omega)\|^4)\|A\textbf{v}\|^2
\nonumber\\
&+C(\|A^{\frac{1}{2}}\textbf{v}\|^2+\|A^{\frac{1}{2}}hy_\delta(\vartheta_t\omega)\|^2
+\|A^{\frac{1}{2}}\textbf{v}\|^4+\|A^{\frac{1}{2}}hy_\delta(\vartheta_t\omega)\|^4)
(\|Ahy_\delta(\vartheta_t\omega)\|^2+\|A^\frac32hy_\delta(\vartheta_t\omega)\|^2)
\nonumber\\
&\leq \frac\nu8\|A^{\frac{3}{2}}\textbf{v}\|^2
+C(\|A^{\frac{1}{2}}\textbf{v}\|^2
+\|A^{\frac{1}{2}}\textbf{v}\|^4+|y_\delta(\vartheta_t\omega)|^2+|y_\delta(\vartheta_t\omega)|^4)\|A\textbf{v}\|^2
\nonumber\\
&+C(\|A^{\frac{1}{2}}\textbf{v}\|^4+|y_\delta(\vartheta_t\omega)|^2
+|y_\delta(\vartheta_t\omega)|^4)|y_\delta(\vartheta_t\omega)|^2\nonumber\\
&\leq \frac\nu8\|A^{\frac{3}{2}}\textbf{v}\|^2
+C(\|A^{\frac{1}{2}}\textbf{v}\|^2
+\|A^{\frac{1}{2}}\textbf{v}\|^4+|y_\delta(\vartheta_t\omega)|^2+|y_\delta(\vartheta_t\omega)|^4)\|A\textbf{v}\|^2
\nonumber\\
&+C(\|A^{\frac{1}{2}}\textbf{v}\|^6+|y_\delta(\vartheta_t\omega)|^4
+|y_\delta(\vartheta_t\omega)|^6).
\end{align}
For the last three terms $K_2(t)$-$K_{4}(t)$, by integration by parts, by using the H\"{o}lder inequality and the Young inequality, we can get that for $h\in H^3$ and $f\in H^1$
\begin{align}\label{6.3}
&K_{2}(t)+K_{3}(t)+K_{4}(t)\nonumber\\
&\leq\|A^\frac{1}{2}f\|\|A^\frac{3}{2}\textbf{v}\|+\nu\|A^\frac{3}{2}h\||y_\delta(\vartheta_t\omega)|\|A^\frac{3}{2}\textbf{v}\|
+\|A^\frac{1}{2}h\||y_\delta(\vartheta_t\omega)|\|A^\frac{3}{2}\textbf{v}\|\nonumber\\
&\leq \frac\nu8\|A^\frac{3}{2}\textbf{v}\|^2+C(\|A^\frac{1}{2}f\|^2+\|A^\frac{3}{2}h\|^2|y_\delta(\vartheta_t\omega)|^2
+\|A^\frac{1}{2}h\|^2|y_\delta(\vartheta_t\omega)|^2)\nonumber\\
&\leq \frac\nu8\|A^\frac{3}{2}\textbf{v}\|^2+C(\|A^\frac{1}{2}f\|^2+|y_\delta(\vartheta_t\omega)|^2
).
\end{align}
Inserting \eqref{6.2}-\eqref{6.3} into \eqref{6.1}, it yields
\begin{align}\label{6.4}
\frac{\d}{\d t}\|A\textbf{v}\|^2+\frac{\nu}{2}\|A^\frac{3}{2}\textbf{v}\|^2&\leq C(\|A^{\frac{1}{2}}\textbf{v}\|^2
+\|A^{\frac{1}{2}}\textbf{v}\|^4+|y_\delta(\vartheta_t\omega)|^2+|y_\delta(\vartheta_t\omega)|^4)\|A\textbf{v}\|^2\nonumber\\
&+C(\|A^\frac{1}{2}f\|^2+|y_\delta(\vartheta_t\omega)|^2+\|A^{\frac{1}{2}}\textbf{v}\|^6+|y_\delta(\vartheta_t\omega)|^4
+|y_\delta(\vartheta_t\omega)|^6)\nonumber\\
&\leq C(1+\|A^{\frac{1}{2}}\textbf{v}\|^4+|y_\delta(\vartheta_t\omega)|^4)\|A\textbf{v}\|^2\nonumber\\
&+C(1+\|A^{\frac{1}{2}}\textbf{v}\|^6+|y_\delta(\vartheta_t\omega)|^6).
\end{align}
For  $\eta\in(t-\frac{1}{2},t-\frac{1}{4})$,  applying Gronwall's lemma to \eqref{6.4} on $(\eta,t)$, we can get that
\begin{align*}
&\|A\textbf{v}(t)\|^2+\frac\nu2\int_\eta^te^{C\int_s^t(1+\|A^\frac12\textbf{v}\|^4+|y_\delta(\vartheta_\tau\omega)|^4)\d\tau}
\|A^\frac{3}{2}\textbf{v}(s)\|^2 \d s\nonumber\\
&\leq
e^{C\int_\eta^t(1+\|A^\frac12\textbf{v}\|^4+|y_\delta(\vartheta_\tau\omega)|^4)\d\tau}\|A\textbf{v}(\eta)\|^2\nonumber\\
&+C\int_\eta^te^{C\int_s^t(1+\|A^\frac12\textbf{v}\|^2+|y_\delta(\vartheta_\tau\omega)|^4)\d\tau}
(1+\|A^{\frac{1}{2}}\textbf{v}\|^6+|y_\delta(\vartheta_s\omega)|^6)\d s,
\end{align*}
and integrating over $\eta\in(t-\frac{1}{2},t-\frac{1}{4})$ yields
\begin{align}
&\frac{1}{4}\|A\textbf{v}(t)\|^2+\frac{\nu}{8}\int_{t-\frac{1}{4}}^te^{C\int_s^t(1+\|A^\frac12\textbf{v}\|^4+
|y_\delta(\vartheta_\tau\omega)|^4)\d\tau}
\|A^\frac{3}{2}\textbf{v}(s)\|^2 \d s\nonumber\\
&\leq
\int_{t-\frac{1}{2}}^{t-\frac{1}{4}}e^{C\int_\eta^t(1+\|A^\frac12\textbf{v}\|^4+|y_\delta(\vartheta_\tau\omega)|^4)\d\tau}
\|A\textbf{v}(\eta)\|^2\d \eta\nonumber\\
&+C\int_{t-\frac{1}{2}}^te^{C\int_s^t(1+\|A^\frac12\textbf{v}\|^4+|y_\delta(\vartheta_\tau\omega)|^4)\d\tau}
(1+\|A^\frac{1}{2}\textbf{v}(s)\|^6+|y_\delta(\vartheta_s\omega)|^6)\d s.
\end{align}
We replace $\omega$ with $\vartheta_{-t}\omega$ to deduce
\begin{align}\label{6.6}
&\|A\textbf{v}(t,\vartheta_{-t}\omega,\textbf{v}_0)\|^2+\frac{\nu}{2}\int_{t-\frac{1}{4}}^t
\|A^\frac{3}{2}\textbf{v}(s,\vartheta_{-t}\omega,\textbf{v}_0)\|^2 \d s\nonumber\\
&\leq\|A\textbf{v}(t,\vartheta_{-t}\omega,\textbf{v}_0)\|^2\nonumber\\
&+\frac{\nu}{2}\int_{t-\frac{1}{4}}^te^{C\int_s^t(1+\|A^\frac12\textbf{v}(\tau,\vartheta_{-t}\omega,\textbf{v}_0)\|^4)\d\tau+
C\int_{s-t}^0|y_\delta(\vartheta_\tau\omega)|^4\d\tau}
\|A^\frac{3}{2}\textbf{v}(s,\vartheta_{-t}\omega,\textbf{v}_0)\|^2 \d s\nonumber\\
&\leq
4\int_{t-\frac{1}{2}}^{t-\frac{1}{4}}e^{C\int_\eta^t(1+\|A^\frac12\textbf{v}(\tau,\vartheta_{-t}\omega,\textbf{v}_0)\|^4)\d\tau
+C\int_\eta^t|y_\delta(\vartheta_{\tau-t}\omega)|^4\d\tau}
\|A\textbf{v}(\eta,\vartheta_{-t}\omega,\textbf{v}_0)\|^2\d \eta\nonumber\\
&+C\int_{t-\frac{1}{2}}^te^{C\int_s^t(1+\|A^\frac12\textbf{v}(\tau,\vartheta_{-t}\omega,\textbf{v}_0)\|^4)\d\tau
+C\int_s^t|y_\delta(\vartheta_{\tau-t}\omega)|^4\d\tau}\nonumber\\
&(1+\|A^\frac{1}{2}\textbf{v}(s,\vartheta_{-t}\omega,\textbf{v}_0)\|^6+|y_\delta(\vartheta_{s-t}\omega)|^6)\d s\nonumber\\
&\leq Ce^{C\int_{t-\frac12}^t(1+\|A^\frac12\textbf{v}(\tau,\vartheta_{-t}\omega,\textbf{v}_0)\|^4)\d\tau
+C\int_{-\frac12}^0|y_\delta(\vartheta_{\tau}\omega)|^4\d\tau}
(\int_{t-\frac{1}{2}}^{t-\frac{1}{4}}\|A\textbf{v}(\eta,\vartheta_{-t}\omega,\textbf{v}_0)\|^2\d \eta\nonumber\\
&+\int_{t-\frac{1}{2}}^t(1+\|A^\frac{1}{2}\textbf{v}(s,\vartheta_{-t}\omega,\textbf{v}_0)\|^6+|y_\delta(\vartheta_{s-t}\omega)|^6)\d s)\nonumber\\
&\leq Ce^{C\int_{t-\frac12}^t(1+\|A^\frac12\textbf{v}(\tau,\vartheta_{-t}\omega,\textbf{v}_0)\|^4)\d\tau
+C\int_{-\frac12}^0|y_\delta(\vartheta_{\tau}\omega)|^4\d\tau}
(\int_{t-\frac{1}{2}}^{t}\|A\textbf{v}(\eta,\vartheta_{-t}\omega,\textbf{v}_0)\|^2\d \eta\nonumber\\
&+\int_{t-\frac{1}{2}}^t(1+\|A^\frac{1}{2}\textbf{v}(s,\vartheta_{-t}\omega,\textbf{v}_0)\|^6)\d s
+\int_{t-\frac{1}{2}}^t|y_\delta(\vartheta_{s-t}\omega)|^6\d s)\nonumber\\
&\leq Ce^{C\int_{t-\frac12}^t(1+\|A^\frac12\textbf{v}(\tau,\vartheta_{-t}\omega,\textbf{v}_0)\|^6)\d\tau
+C\int_{-\frac12}^0|y_\delta(\vartheta_{\tau}\omega)|^6\d\tau}
(\int_{t-\frac{1}{2}}^{t}\|A\textbf{v}(\eta,\vartheta_{-t}\omega,\textbf{v}_0)\|^2\d \eta\nonumber\\
&+\int_{t-\frac{1}{2}}^t\|A^\frac{1}{2}\textbf{v}(s,\vartheta_{-t}\omega,\textbf{v}_0)\|^6\d s
+\int_{-1}^0|y_\delta(\vartheta_{s}\omega)|^6\d s+1),
\end{align}
By \eqref{mar8.3}, we also can deduce that
\begin{align}
\int_{t-\frac{1}{2}}^{t}\|A\textbf{v}(\eta,\vartheta_{-t}\omega,\textbf{v}_0)\|^2\d \eta&\leq    \zeta_3(\omega) ,\label{6.7}\\
  \int_{t-1}^t\|A^{\frac12}\textbf{v}(\eta,\theta_{-t}\omega,\textbf{v}(0))\|^6 \, \d \eta
&  \leq    |\zeta_3(\omega)|^3 ,  \quad t\geq T_{\mathfrak B} (\omega) ,\label{6.5}
\end{align}
here    $ \zeta_3 $ is a tempered random variable defined in \eqref{mar8.3} and  a random variable $T_{\mathfrak B} (\omega)$ defined in \eqref{timeB}, respectively.

Inserting \eqref{6.7} and \eqref{6.5} into \eqref{6.6} yields
\begin{align}\label{6.17}
&\|A\textbf{v}(t,\vartheta_{-t}\omega,\textbf{v}_0)\|^2+\frac{\nu}{2}\int_{t-\frac{1}{4}}^t
\|A^\frac{3}{2}\textbf{v}(s,\vartheta_{-t}\omega,\textbf{v}_0)\|^2 \d s\nonumber\\
&\leq
e^{C
(1+|\zeta_3(\omega)|^3)+C
\int_{-\frac{1}{2}}^0
|y_\delta(\vartheta_\tau\omega)|^6 \d\tau}\left(C|\zeta_3(\omega)|^3+C\int_{-1}^0
|y_\delta(\vartheta_{s}\omega)|^6\d s+C\right)\nonumber\\
&:=\zeta_{10}(\omega).
\end{align}
This completes the proof of the Lemma \ref{lemma7.1}.
\end{proof}

\begin{lemma}\label{lemma7.2} [$H^3$ bound]
Let  $f\in H^2$ and Assumption \ref{assum1} hold.
Then there exists a random variable $ T_{\mathfrak B}^* (\omega)>T_1(\omega)$ such that,  for any $t\geq T_{\mathfrak B}^* (\omega)$,
\begin{align*}
\big\|A^\frac32\textbf{v}(t,\theta_{-t}\omega,\textbf{v}(0))  \big\|^2 + \int_{t-\frac 12 }^t \|A^{2}\textbf{v}(s,\theta_{-t}\omega,\textbf{v}(0))\|^2  \ \d s
\leq  \zeta_{12}(\omega),
\end{align*}
where $\zeta_{12}(\omega)$  is a tempered random variable given by \eqref{6.18} such that $\zeta_{12}(\omega) > \zeta_1(\omega)\geq 1$, $ \omega\in\Omega$.
\end{lemma}
\begin{proof}
Taking the inner product of the first equation   \eqref{2.2} with $A^3\textbf{v}$ in $H$, by integration by parts yields
\begin{align}\label{6.8}
&\frac12\frac{\d}{\d t}\|A^\frac32\textbf{v}\|^2+\nu \int_{\mathbb{T}^2}\partial_{yy}\Delta v_1\Delta^2 v_1dx+\nu \int_{\mathbb{T}^2}\partial_{xx}\Delta v_2\Delta^2 v_2dx\nonumber\\
&=-\big(B(\textbf{v}+hy_\delta(\vartheta_t\omega)), A^3\textbf{v} \big)+(f,A^3\textbf{v})\nonumber\\
& +(-\nu A_1hy_\delta(\vartheta_t\omega),A^3\textbf{v})+(hy_\delta(\vartheta_t\omega),A^3\textbf{v})\nonumber\\
&=:K_5(t)+K_6(t)+K_7(t)+K_8(t).
\end{align}
By direct computations and $\nabla\cdot \textbf{v}=0$ and \eqref{3.9}, it also yields
\begin{align}\label{6.11}
&\nu \int_{\mathbb{T}^2}\partial_{yy}\Delta v_1\Delta^2 v_1dx-\nu \int_{\mathbb{T}^2}\partial_{xx}\Delta v_2\Delta^2 v_2dx\nonumber\\
&= \nu \int_{\mathbb{T}^2}(\partial_{yy}\Delta v_1)^2dx+\nu \int_{\mathbb{T}^2}\partial_{xy}\Delta v_1\partial_{xy}\Delta v_1dx\nonumber\\
&+\nu \int_{\mathbb{T}^2}(\partial_{xx}\Delta v_2)^2dx+\nu \int_{\mathbb{T}^2}\partial_{xy}\Delta v_2\partial_{xy}\Delta v_2dx
\nonumber\\
&=\nu \int_{\mathbb{T}^2}(\partial_{yy}\Delta v_1)^2dx+\nu \int_{\mathbb{T}^2}(\partial_{yy}\Delta v_2)^2dx\nonumber\\
&+\nu \int_{\mathbb{T}^2}(\partial_{xx}\Delta v_2)^2dx+\nu \int_{\mathbb{T}^2}(\partial_{xx}\Delta v_1)^2dx
\nonumber\\
&\geq\frac\nu2\|A^{2} \textbf{v}\|^2.
\end{align}
For the term $K_5(t)$,  by integration by parts, applying the H\"{o}lder inequality, the Gagliardo-Nirenberg inequality and the Young inequality, we also can get that
\begin{align}\label{6.9}
K_5(t)&\leq C\|A(\textbf{v}+hy_\delta(\vartheta_t\omega))\|_{L^4}\|\nabla(\textbf{v}+hy_\delta(\vartheta_t\omega))\|_{L^4}
\|A^{2}\textbf{v}\|\nonumber\\
&+C\|\textbf{v}+hy_\delta(\vartheta_t\omega)\|_{L^4}\|A^\frac32(\textbf{v}+hy_\delta(\vartheta_t\omega))\|_{L^4}
\|A^{2}\textbf{v}\|\nonumber\\
&\leq C\|A^{\frac{3}{2}}(\textbf{v}+hy_\delta(\vartheta_t\omega))\|\|A(\textbf{v}+hy_\delta(\vartheta_t\omega))\|
\|A^{2}\textbf{v}\|\nonumber\\
&+C\|A^{\frac{1}{2}}(\textbf{v}+hy_\delta(\vartheta_t\omega))\|\|A^\frac32(\textbf{v}+hy_\delta(\vartheta_t\omega))\|^\frac12
(\|A^2\textbf{v}\|^\frac12+\|A^2hy_\delta(\vartheta_t\omega)\|^\frac12)\|A^{2}\textbf{v}\|\nonumber\\
&\leq \frac\nu8\|A^{2}\textbf{v}\|^2
+C\|A(\textbf{v}+hy_\delta(\vartheta_t\omega))\|^2\|A^{\frac{3}{2}}(\textbf{v}+hy_\delta(\vartheta_t\omega))\|^2\nonumber\\
&+C\|A^{\frac{1}{2}}(\textbf{v}+hy_\delta(\vartheta_t\omega))\|^4\|A^\frac32(\textbf{v}+hy_\delta(\vartheta_t\omega))\|^2
\nonumber\\
&+C\|A^{\frac{1}{2}}(\textbf{v}+hy_\delta(\vartheta_t\omega))\|^2\|A^\frac32(\textbf{v}+hy_\delta(\vartheta_t\omega))\|
\|A^2hy_\delta(\vartheta_t\omega)\|
\nonumber\\
&\leq \frac\nu8\|A^{2}\textbf{v}\|^2
+C(\|A\textbf{v}\|^2+\|Ahy_\delta(\vartheta_t\omega)\|^2
+\|A^{\frac{1}{2}}\textbf{v}\|^4+\|A^{\frac{1}{2}}hy_\delta(\vartheta_t\omega)\|^4)\|A^\frac32\textbf{v}\|^2
\nonumber\\
&+C(\|A\textbf{v}\|^2+\|Ahy_\delta(\vartheta_t\omega)\|^2
+\|A^{\frac{1}{2}}\textbf{v}\|^4+\|A^{\frac{1}{2}}hy_\delta(\vartheta_t\omega)\|^4)\|A^2hy_\delta(\vartheta_t\omega)\|^2
\nonumber\\
&\leq \frac\nu8\|A^{2}\textbf{v}\|^2
+C(\|A\textbf{v}\|^2
+\|A^{\frac{1}{2}}\textbf{v}\|^4+|y_\delta(\vartheta_t\omega)|^2+|y_\delta(\vartheta_t\omega)|^4)\|A^\frac32\textbf{v}\|^2
\nonumber\\
&+C(\|A\textbf{v}\|^2+\|A^{\frac{1}{2}}\textbf{v}\|^4+|y_\delta(\vartheta_t\omega)|^2
+|y_\delta(\vartheta_t\omega)|^4)|y_\delta(\vartheta_t\omega)|^2\nonumber\\
&\leq \frac\nu8\|A^{2}\textbf{v}\|^2
+C(1+\|A\textbf{v}\|^2
+\|A^{\frac{1}{2}}\textbf{v}\|^4+|y_\delta(\vartheta_t\omega)|^4)\|A^\frac32\textbf{v}\|^2
\nonumber\\
&+C(1+\|A\textbf{v}\|^4+\|A^{\frac{1}{2}}\textbf{v}\|^6
+|y_\delta(\vartheta_t\omega)|^6).
\end{align}
For the last three terms $K_6(t)$-$K_{8}(t)$, by integration by parts and the similar method, we also can get that for $h\in H^4$ and $f\in H^2$
\begin{align}\label{6.10}
&K_{6}(t)+K_{7}(t)+K_{8}(t)\nonumber\\
&\leq\|Af\|\|A^2\textbf{v}\|+\nu\|A^2h\||y_\delta(\vartheta_t\omega)|\|A^2\textbf{v}\|
+\|Ah\||y_\delta(\vartheta_t\omega)|\|A^2\textbf{v}\|\nonumber\\
&\leq \frac\nu8\|A^2\textbf{v}\|^2+C(\|Af\|^2+\|A^2h\|^2|y_\delta(\vartheta_t\omega)|^2
+\|Ah\|^2|y_\delta(\vartheta_t\omega)|^2)\nonumber\\
&\leq \frac\nu8\|A^2\textbf{v}\|^2+C(\|Af\|^2+|y_\delta(\vartheta_t\omega)|^2
).
\end{align}
Inserting \eqref{6.11}, \eqref{6.9} and \eqref{6.10} into \eqref{6.8}, it yields
\begin{align}\label{6.12}
\frac{\d}{\d t}\|A^\frac{3}{2}\textbf{v}\|^2+\frac{\nu}{2}\|A^2\textbf{v}\|^2&\leq C(1+\|A\textbf{v}\|^2
+\|A^{\frac{1}{2}}\textbf{v}\|^4+|y_\delta(\vartheta_t\omega)|^4)\|A^\frac{3}{2}v\|^2\nonumber\\
&+C(1+\|Af\|^2+\|A^{\frac{1}{2}}\textbf{v}\|^6+\|A\textbf{v}\|^4
+|y_\delta(\vartheta_t\omega)|^6)\nonumber\\
&\leq C(1+\|A\textbf{v}\|^4
+\|A^{\frac{1}{2}}\textbf{v}\|^6+|y_\delta(\vartheta_t\omega)|^6)\|A^\frac{3}{2}v\|^2\nonumber\\
&+C(1+\|A\textbf{v}\|^4+\|A^{\frac{1}{2}}\textbf{v}\|^6
+|y_\delta(\vartheta_t\omega)|^6)\nonumber\\
&\leq C(1+\|A\textbf{v}\|^6
+|y_\delta(\vartheta_t\omega)|^6)\|A^\frac{3}{2}v\|^2\nonumber\\
&+C(1+\|A\textbf{v}\|^6
+|y_\delta(\vartheta_t\omega)|^6).
\end{align}
For  $\eta\in(t-\frac{1}{2},t-\frac{1}{4})$,  applying Gronwall's lemma to \eqref{6.12} on $(\eta,t)$, we also can get that
\begin{align*}
&\|A^\frac{3}{2}\textbf{v}(t)\|^2+\frac\nu2\int_\eta^te^{C\int_s^t(1+\|A\textbf{v}\|^6+|y_\delta(\vartheta_\tau\omega)|^6)\d\tau}
\|A^2\textbf{v}(s)\|^2 \d s\nonumber\\
&\leq
e^{C\int_\eta^t(1+\|A\textbf{v}\|^6+|y_\delta(\vartheta_\tau\omega)|^6)\d\tau}\|A^\frac{3}{2}\textbf{v}(\eta)\|^2\nonumber\\
&+C\int_\eta^te^{C\int_s^t(1+\|A\textbf{v}\|^6+|y_\delta(\vartheta_\tau\omega)|^6)\d\tau}
(1+\|A\textbf{v}\|^6+|y_\delta(\vartheta_s\omega)|^6)\d s,
\end{align*}
and integrating over $\eta\in(t-\frac{1}{2},t-\frac{1}{4})$ yields
\begin{align}
&\frac{1}{4}\|A^\frac{3}{2}\textbf{v}(t)\|^2+\frac{\nu}{8}\int_{t-\frac{1}{4}}^te^{C\int_s^t(1+\|A\textbf{v}\|^6+
|y_\delta(\vartheta_\tau\omega)|^6)\d\tau}
\|A^2\textbf{v}(s)\|^2 \d s\nonumber\\
&\leq
\int_{t-\frac{1}{2}}^{t-\frac{1}{4}}e^{C\int_\eta^t(1+\|A\textbf{v}\|^6+|y_\delta(\vartheta_\tau\omega)|^6)\d\tau}
\|A^\frac{3}{2}\textbf{v}(\eta)\|^2\d \eta\nonumber\\
&+C\int_{t-\frac{1}{2}}^te^{C\int_s^t(1+\|A\textbf{v}\|^6+|y_\delta(\vartheta_\tau\omega)|^6)\d\tau}
(1+\|A\textbf{v}(s)\|^6+|y_\delta(\vartheta_s\omega)|^6)\d s.
\end{align}
We replace $\omega$ with $\vartheta_{-t}\omega$ to deduce
\begin{align}\label{6.13}
&\|A^\frac{3}{2}\textbf{v}(t,\vartheta_{-t}\omega,\textbf{v}_0)\|^2+\frac{\nu}{2}\int_{t-\frac{1}{4}}^t
\|A^2\textbf{v}(s,\vartheta_{-t}\omega,\textbf{v}_0)\|^2 \d s\nonumber\\
&\leq\|A^\frac{3}{2}\textbf{v}(t,\vartheta_{-t}\omega,\textbf{v}_0)\|^2\nonumber\\
&+\frac{\nu}{2}\int_{t-\frac{1}{4}}^te^{C\int_s^t(1+\|A\textbf{v}(\tau,\vartheta_{-t}\omega,\textbf{v}_0)\|^6)\d\tau+
C\int_{s-t}^0|y_\delta(\vartheta_\tau\omega)|^6\d\tau}
\|A^2\textbf{v}(s,\vartheta_{-t}\omega,\textbf{v}_0)\|^2 \d s\nonumber\\
&\leq
4\int_{t-\frac{1}{2}}^{t-\frac{1}{4}}e^{C\int_\eta^t(1+\|A\textbf{v}(\tau,\vartheta_{-t}\omega,\textbf{v}_0)\|^6)\d\tau
+C\int_\eta^t|y_\delta(\vartheta_{\tau-t}\omega)|^6\d\tau}
\|A^\frac{3}{2}\textbf{v}(\eta,\vartheta_{-t}\omega,\textbf{v}_0)\|^2\d \eta\nonumber\\
&+C\int_{t-\frac{1}{2}}^te^{C\int_s^t(1+\|A\textbf{v}(\tau,\vartheta_{-t}\omega,\textbf{v}_0)\|^6)\d\tau
+C\int_s^t|y_\delta(\vartheta_{\tau-t}\omega)|^6\d\tau}\nonumber\\
&(1+\|A\textbf{v}(s,\vartheta_{-t}\omega,\textbf{v}_0)\|^6+|y_\delta(\vartheta_{s-t}\omega)|^6)\d s\nonumber\\
&\leq Ce^{C\int_{t-\frac12}^t(1+\|A\textbf{v}(\tau,\vartheta_{-t}\omega,\textbf{v}_0)\|^6)\d\tau
+C\int_{-\frac12}^0|y_\delta(\vartheta_{\tau}\omega)|^6\d\tau}
(\int_{t-\frac{1}{2}}^{t-\frac{1}{4}}\|A^\frac{3}{2}\textbf{v}(\eta,\vartheta_{-t}\omega,\textbf{v}_0)\|^2\d \eta\nonumber\\
&+\int_{t-\frac{1}{2}}^t(1+\|A\textbf{v}(s,\vartheta_{-t}\omega,\textbf{v}_0)\|^6+|y_\delta(\vartheta_{s-t}\omega)|^6)\d s)\nonumber\\
&\leq Ce^{C\int_{t-\frac12}^t(1+\|A\textbf{v}(\tau,\vartheta_{-t}\omega,\textbf{v}_0)\|^6)\d\tau
+C\int_{-\frac12}^0|y_\delta(\vartheta_{\tau}\omega)|^6\d\tau}
(\int_{t-\frac{1}{2}}^{t}\|A^\frac{3}{2}\textbf{v}(\eta,\vartheta_{-t}\omega,\textbf{v}_0)\|^2\d \eta\nonumber\\
&+\int_{t-\frac{1}{2}}^t(1+\|A\textbf{v}(s,\vartheta_{-t}\omega,\textbf{v}_0)\|^6)\d s
+\int_{t-\frac{1}{2}}^t|y_\delta(\vartheta_{s-t}\omega)|^6\d s)\nonumber\\
&\leq Ce^{C\int_{t-\frac12}^t(1+\|A\textbf{v}(\tau,\vartheta_{-t}\omega,\textbf{v}_0)\|^6)\d\tau
+C\int_{-\frac12}^0|y_\delta(\vartheta_{\tau}\omega)|^6\d\tau}
(\int_{t-\frac{1}{2}}^{t}\|A^\frac{3}{2}\textbf{v}(\eta,\vartheta_{-t}\omega,\textbf{v}_0)\|^2\d \eta\nonumber\\
&+\int_{t-\frac{1}{2}}^t\|A\textbf{v}(s,\vartheta_{-t}\omega,\textbf{v}_0)\|^6\d s
+\int_{-1}^0|y_\delta(\vartheta_{s}\omega)|^6\d s+1),
\end{align}
By \eqref{6.4}, we also have
\begin{align}\label{6.14}
\frac{\d}{\d t}\|A\textbf{v}\|^2+\frac{\nu}{2}\|A^\frac{3}{2}\textbf{v}\|^2
&\leq C(1+\|A^{\frac{1}{2}}\textbf{v}\|^6+|y_\delta(\vartheta_t\omega)|^6)\|Av\|^2\nonumber\\
&+C(1+\|A^{\frac{1}{2}}\textbf{v}\|^6+|y_\delta(\vartheta_t\omega)|^6).
\end{align}
For $t>2$, integrating \eqref{6.14} over $(\tau,t)$ with $\tau\in (t-2,t-1)$, we can deduce that
\begin{align*}
&
\|A\textbf{v} (t) \|^2 -\|A\textbf{v}(\tau) \|^2  + \frac{\nu }{2}\int^t_{t-1} \|A^\frac32 \textbf{v}(s)\|^2 \, \d s \\
&\quad
 \leq  C \int^t_{t-2} \left(1+\|A^{\frac{1}{2}}\textbf{v}\|^6+  |y_\delta(\theta_s\omega)|^6 \right) \|A\textbf{v}(s) \|^2\, \d s
+ C\int^t_{t-2} \left(1+\|A^{\frac{1}{2}}\textbf{v}\|^6 +|y_\delta(\theta_s\omega)|^6 \right )  \d s,
\end{align*}
and then integrating over $\tau\in (t-2,t-1)$ yields
\begin{align*}
\|A\textbf{v} (t) \|^2+  \frac{\nu}{2} \int^t_{t-1} \|A^{\frac32} \textbf{v}(s)\|^2\, \d s
& \leq   C \int^t_{t-2} \left (\|A^{\frac{1}{2}}\textbf{v}\|^6+|y_\delta(\theta_s\omega)|^6+1\right)  \|A\textbf{v}(s) \|^2 \, \d s \\
&\quad + C\int^t_{t-2} \left (1+\|A^{\frac{1}{2}}\textbf{v}\|^6 +|y_\delta(\theta_s\omega)|^6 \right)  \d s ,\quad t> 2.
\end{align*}
For $\varepsilon\in  [0,2]$, by \eqref{mar8.3}, we replace $\omega$ with $\theta_{-t-\varepsilon} \omega$  to deduce
\begin{align*}
 & \|A\textbf{v} (t,\theta_{-t-\varepsilon} \omega, \textbf{v}(0)) \|^2 + \frac{ \nu}{2} \int_{t-1}^t  \|A^\frac32\textbf{v} (s,\theta_{-t-\varepsilon} \omega, \textbf{v}(0)) \|^2 \, \d s \nonumber \\
& \quad \leq   C \int^t_{t-2} \! \left (\|A^{\frac12}\textbf{v}(s,\theta_{-t-\varepsilon} \omega, \textbf{v}(0)) \|^6+|y_\delta(\theta_{s-t-\varepsilon} \omega)|^2+1\right)  \|A\textbf{v}(s,\theta_{-t-\varepsilon} \omega, \textbf{v}(0)) \|^2 \, \d s \nonumber\\
&\quad+\int^t_{t-2} \! \|A^{\frac12}\textbf{v}(s,\theta_{-t-\varepsilon} \omega, \textbf{v}(0)) \|^6\, \d s+ C\int^{-\varepsilon}_{-\varepsilon-2} \! \left(1 +|y_\delta(\theta_s\omega)|^6 \right )  \d s\nonumber\\
& \quad \leq   C\left (  |\zeta_3(\omega)|^3+ \sup_{\tau\in (-4,0)}|y_\delta(\theta_\tau\omega)|^6+1\right) \int^{t}_{t+\varepsilon-4}  \|A\textbf{v}(s,\theta_{-t-\varepsilon} \omega, \textbf{v}(0)) \|^2 \, \d s \nonumber\\
&\qquad+ |\zeta_3(\omega)|^3+ C\int^{0}_{-4}    \left(1 +|y_\delta(\theta_s\omega)|^6 \right)   \d s \nonumber\\
& \quad \leq   C\left (  |\zeta_3(\omega)|^3+ \sup_{\tau\in (-4,0)}|y_\delta(\theta_\tau\omega)|^6+1\right) \zeta_3(\omega) \nonumber\\
&\qquad+ |\zeta_3(\omega)|^3+ C\int^{0}_{-4}    \left(1 +|y_\delta(\theta_s\omega)|^6 \right)   \d s,\quad t>T_{\mathfrak B} (\omega) .
\end{align*}

 In other words,  for any  $t\geq T_{\mathfrak B}^* (\omega)= T_{\mathfrak B} (\omega) +5$ and $\eta\in [t-1,t ]$ $($so that $\eta\geq T_{\mathfrak B} (\omega)+4)$, it can obtain that
\begin{align*}
& \|A\textbf{v} (\eta ,\theta_{-t } \omega, \textbf{v}(0)) \|^2 + \frac{\nu}{2} \int_{\eta-1}^\eta  \|A^{\frac32}\textbf{v } (s,\theta_{-t} \omega, \textbf{v}(0)) \|^2 \, \d s\\
  & \quad =  \|A\textbf{v} (\eta ,\theta_{-\eta-(t-\eta) } \omega, \textbf{v}(0)) \|^2
  + \frac{\nu}{2}  \int_{\eta-1}^\eta  \|A^{\frac32}\textbf{v} (s,\theta_{-\eta-(t-\eta)} \omega, \textbf{v}(0)) \|^2 \, \d s \\
&\quad \leq  C\left (  |\zeta_3(\omega)|^3+ \sup_{\tau\in (-4,0)}|y_\delta(\theta_\tau\omega)|^6+1\right) \zeta_3(\omega) \nonumber\\
&\qquad+ |\zeta_3(\omega)|^3+ C\int^{0}_{-4}    \left(1 +|y_\delta(\theta_s\omega)|^6 \right)   \d s:=\zeta_{11}(\omega),\quad t>T_{\mathfrak B}^* (\omega) .
\end{align*}
Then we get the following main estimates
\begin{align}
\int_{t-\frac{1}{2}}^{t}\|A^\frac{3}{2}\textbf{v}(\eta,\vartheta_{-t}\omega,\textbf{v}_0)\|^2\d \eta&\leq \zeta_{11}(\omega)\label{6.15}\\
\int_{t-\frac{1}{2}}^t\|A\textbf{v}(s,\vartheta_{-t}\omega,\textbf{v}_0)\|^6 \d s&\leq |\zeta_{11}(\omega)|^3.\label{6.16}
\end{align}
Inserting \eqref{6.15} and \eqref{6.16} into \eqref{6.13}, we have that
\begin{align}\label{6.18}
&\|A^\frac{3}{2}\textbf{v}(t,\vartheta_{-t}\omega,\textbf{v}_0)\|^2+\frac{\nu}{2}\int_{t-\frac{1}{4}}^t
\|A^2\textbf{v}(s,\vartheta_{-t}\omega,\textbf{v}_0)\|^2 \d s\nonumber\\
&\leq
e^{C
(1+|\zeta_{11}(\omega)|^3)+C
\int_{-\frac{1}{2}}^0
|y_\delta(\vartheta_\tau\omega)|^6 \d\tau}\left(C|\zeta_{11}(\omega)|^3+C\int_{-1}^0
|y_\delta(\vartheta_{s}\omega)|^6\d s+C\right)\nonumber\\
&:=\zeta_{12}(\omega).
\end{align}
This completes the proof of the Lemma \ref{lemma7.2}.
\end{proof}
\subsection{$(H, H^3)$-smoothing}

Let $k=2$ or $k=3$. We introduce a random set given by
$\mathfrak{B}^{k}=\{\mathfrak{B}^k(\omega)\}_{\omega\in\Omega}$ in $H^k$  by
\begin{align}
\mathfrak{B}^{k}(\omega):= \left \{ v \in H^k:\|A^{\frac k2}v\|^2\leq \zeta_{12}(\omega)
\right \}, \quad \omega\in\Omega,
\end{align}
where $\zeta_{12}(\cdot)$ is the tempered random variable given by \eqref{6.18}. In order to prove the $( H, H^{3})$-Lipschtiz, we  need to prove the $( H, H^{2})$-Lipschtiz and $(H^2, H^{3})$-Lipschtiz. By the similar method in Theorem \ref{theorem5.6} and $f\in H^2$, we will obtain the $(H, H^2)$-smoothing. Moreover, we only prove the $(H^2, H^{3})$-Lipschtiz in this subsection.

\begin{lemma}[$(H^2, H^3)$-smoothing on $\mathfrak B^2$]\label{lemma7.3} Let Assumption \ref{assum1} hold and $f\in H^2$. There  exist  random variables $T_\omega^*$ and $ L_5(\mathfrak B^2, \omega )$ such that   any two solutions $\textbf{v}_1$ and $\textbf{v}_2$ of the random anisotropic  NS equations \eqref{2.2} driven by colored noise corresponding to initial values  $\textbf{v}_{1,0},$ $\textbf{v}_{2,0}$ in $\mathfrak{B}^2(\theta_{-T_\omega^*}\omega)$, respectively,   satisfy
\begin{align*}
 \big \|\textbf{v}_1(T_\omega^*,\theta_{-T_\omega^*}\omega,\textbf{v}_{1,0})-\textbf{v}_2(T_\omega^*,\theta_{-T_\omega^*}\omega,
 \textbf{v}_{2,0}) \big \|{^2_{H^3}}\leq  L_5(\mathfrak B^2, \omega)\|\textbf{v}_{1,0}-\textbf{v}_{2,0}\|^2_{H^2}, \quad \omega\in \Omega.
\end{align*}
\end{lemma}

\begin{proof} Taking the inner product of the system   \eqref{5.1} with $A^3\bar{v}$ in $H$,  by integration by parts, we can deduce that
\begin{align}\label{6.20}
&\frac{1}{2}\frac{\d}{\d t}\|A^{\frac32}\bar{v}\|^2+\nu \int_{\mathbb{T}^2}\partial_{yy}\Delta \bar{v}_1\Delta^2\bar{v}_1dx+\nu \int_{\mathbb{T}^2}\partial_{xx}\Delta \bar{v}_2\Delta^2 \bar{v}_2dx\nonumber\\
&=- \big (B(\bar{v},\textbf{v}_1+hy_\delta(\theta_t\omega)),\, A^3\bar{v}  \big )
- \big (B(\textbf{v}_2+hy_\delta(\theta_t\omega),\bar{v}),\, A^3\bar{v} \big ).
\end{align}
By direct computations and $\nabla\cdot \bar{v}=0$ and \eqref{3.9}, we also have
\begin{align}\label{6.19}
&\nu \int_{\mathbb{T}^2}\partial_{yy}\Delta \bar{v}_1\Delta^2 \bar{v}_1dx-\nu \int_{\mathbb{T}^2}\partial_{xx}\Delta \bar{v}_2\Delta^2 \bar{v}_2dx\nonumber\\
&= \nu \int_{\mathbb{T}^2}(\partial_{yy}\Delta \bar{v}_1)^2dx+\nu \int_{\mathbb{T}^2}\partial_{xy}\Delta \bar{v}_1\partial_{xy}\Delta \bar{v}_1dx\nonumber\\
&+\nu \int_{\mathbb{T}^2}(\partial_{xx}\Delta \bar{v}_2)^2dx+\nu \int_{\mathbb{T}^2}\partial_{xy}\Delta \bar{v}_2\partial_{xy}\Delta \bar{v}_2dx
\nonumber\\
&=\nu \int_{\mathbb{T}^2}(\partial_{yy}\Delta \bar{v}_1)^2dx+\nu \int_{\mathbb{T}^2}(\partial_{yy}\Delta \bar{v}_2)^2dx\nonumber\\
&+\nu \int_{\mathbb{T}^2}(\partial_{xx}\Delta \bar{v}_2)^2dx+\nu \int_{\mathbb{T}^2}(\partial_{xx}\Delta \bar{v}_1)^2dx
\nonumber\\
&\geq\frac\nu2\|A^2 \textbf{v}\|^2.
\end{align}
By using the H\"{o}lder inequality, the Gagliardo-Nirenberg inequality and the Young inequality, we can get that
\begin{align}\label{6.21}
 \big| \big(B(\bar{v},\textbf{v}_1+hy_\delta(\theta_t\omega)),\, A^3\bar{v}\big)
 \big| &\leq C\|A\bar{v}\|_{L^\infty}\|\nabla(\textbf{v}_1+hy_\delta(\theta_t\omega))\|\|A^2\bar{v}\|  \nonumber\\
&\quad
 +C\|\bar{v}\|_{L^\infty}\|A^\frac32(\textbf{v}_1+hy_\delta(\theta_t\omega))\|\|A^2\bar{v}\|
 \nonumber\\
&\leq C\|A^\frac32\bar{v}\|^\frac12\|A^2\bar{v}\|^\frac32
\left (\|A^\frac12\textbf{v}_1\|+\|A^\frac12h\||y_\delta(\theta_t\omega)| \right )\nonumber\\
&\quad +C\|A\bar{v}\|\|A^2\bar{v}\| \big(\|A^\frac32\textbf{v}_1\|+\|A^\frac32h\||y_\delta(\theta_t\omega)| \big)\nonumber\\
&\leq \frac  \nu 8\|A^2\bar{v}\|^2+C\|A^\frac32\bar{v}\|^2 \left (\|A^\frac12\textbf{v}_1\|^4+|y_\delta(\theta_t\omega)|^4 \right)\nonumber\\
&\quad +C\|A\bar{v}\|^2 \left (\|A^\frac32\textbf{v}_1\|^2+|y_\delta(\theta_t\omega)|^2 \right),
\end{align}
and
\begin{align}\label{6.22}
&\big| \big(B(\textbf{v}_2+hy_\delta(\theta_t\omega),\bar{v}),\, A^3\bar{v}\big)\big|\nonumber\\
&\leq C\|A(\textbf{v}_2+hy_\delta(\theta_t\omega))\|_{L^4}\|\nabla\bar{v}\|_{L^4}\|A^2\bar{v}\| \nonumber\\
&\quad
 +C\|\textbf{v}_2+hy_\delta(\theta_t\omega)\|_{L^4}\|A^\frac32\bar{v}\|_{L^4}\|A^2\bar{v}\|\nonumber\\
&\leq C\|A\bar{v}\|\|A^2\bar{v}\|
\left (\|A^\frac32\textbf{v}_2\|+\|A^\frac32h\||y_\delta(\theta_t\omega)| \right)\nonumber\\
&\quad +C \|A^\frac12(\textbf{v}_2+hy_\delta(\theta_t\omega)
)\|\|A^\frac32\bar{v}\|^\frac12\|A^2\bar{v}\|^\frac32 \nonumber\\
&\leq \frac  \nu 8\|A^2\bar{v}\|^2+C\|A^\frac32\bar{v}\|^2 \left (\|A^\frac12\textbf{v}_2\|^4+|y_\delta(\theta_t\omega)|^4\right)\nonumber\\
&\quad +C\|A\bar{v}\|^2 \left (\|A^\frac32\textbf{v}_2\|^2+|y_\delta(\theta_t\omega)|^2 \right).
\end{align}
Inserting \eqref{6.19}, \eqref{6.21} and \eqref{6.22} into \eqref{6.20} yields
\begin{align*}
\frac{\d}{\d t}\|A^\frac32\bar{v}\|^2
&\leq C\|A^\frac32\bar{v}\|^2 \left ( \sum_{j=1}^2 \|A^\frac12\textbf{v}_j\|^4 +|y_\delta(\theta_t\omega)|^4 \right )
\nonumber\\
 &+C\|A\bar{v}\|^2 \left (\sum_{j=1}^2\|A^\frac32\textbf{v}_j\|^2+|y_\delta(\theta_t\omega)|^2 \right).
\end{align*}
For any $s\in(t-1,t)$ and $t\geq 1$, applying Gronwall's lemma, then  we also can get that
\begin{align*}
&\|A^\frac32\bar{v}(t)\|^2
-e^{\int_s^tC( \sum_{i=1}^2 \|A^\frac12\textbf{v}_i\|^4 +|y_\delta(\theta_\tau\omega)|^4)\, \d \tau}\|A^\frac32\bar{v}(s)\|^2\nonumber\\
&\quad \leq \int_s^tCe^{\int_\eta^tC( \sum_{i=1}^2 \|A^\frac12\textbf{v}_i\|^4+|y_\delta(\theta_\tau\omega)|^4)\, \d \tau}
\|A\bar{v}\|^2 \left(\sum_{i=1}^2\|A^\frac32\textbf{v}_i\|^2+|y_\delta(\theta_\eta\omega)|^2\right) \d \eta\nonumber\\
&\quad \leq Ce^{\int_{t-1}^tC( \sum_{i=1}^2 \|A^\frac12\textbf{v}_i\|^2+|y_\delta(\theta_\tau\omega)|^4)\, \d \tau}
\int_{t-1}^t\|A\bar{v}\|^2\left(\sum_{i=1}^2\|A^\frac32\textbf{v}_i\|^2+|y_\delta(\theta_\eta\omega)|^2 \right) \d \eta. \nonumber
\end{align*}
Integrating w.r.t. $s$ on $(t-1,t)$ yields
\begin{align}\label{6.23}
&\|A^\frac32\bar{v}(t)\|^2
-\int_{t-1}^te^{\int_s^tC( \sum_{i=1}^2 \|A^\frac12\textbf{v}_i\|^4+|y_\delta(\theta_\tau\omega)|^4)\, \d \tau}\|A^\frac32\bar{v}(s)\|^2\, \d s\nonumber\\
&\quad \leq Ce^{\int_{t-1}^tC( \sum_{i=1}^2 \|A^\frac12\textbf{v}_i\|^4+|y_\delta(\theta_\tau\omega)|^4)\, \d \tau}
\int_{t-1}^t\|A\bar{v}\|^2\left (\sum_{i=1}^2\|A^\frac32\textbf{v}_i\|^2+|y_\delta(\theta_\eta\omega)|^2\right) \d \eta.
\end{align}

Since our initial values $\textbf{v}_{1,0}$ and $\textbf{v}_{2,0}$   belong to the $H^2$ random  absorbing set  $\mathfrak B^2$,  applying Gronwall's lemma to \eqref{7.1}, then we can deduce    that
\begin{align*}
 & \|A\bar{v}(t)\|^2 + \frac{\nu}{2} \int_0^te^{\int_s^tC( \sum_{i=1}^2 \|A^\frac12\textbf{v}_i\|^4+\|A\textbf{v}_1\|^4+|y_\delta(\theta_\tau\omega)|^4)\, \d \tau}\|A^\frac32\bar{v}(s) \|^2\, \d s
\nonumber\\
&\quad \leq e^{\int_0^tC( \sum_{i=1}^2 \|A^\frac12\textbf{v}_i\|^4+\|A\textbf{v}_1\|^4+|y_\delta(\theta_\tau\omega)|^4)\, \d \tau}\|A\bar{v}(0)\|^2.
\end{align*}
Hence, we can get the  following inequalities
\begin{align*}
& \frac{\nu}{2} \int_{t-1}^te^{\int_s^tC( \sum_{i=1}^2 \|A^\frac12\textbf{v}_i\|^4+|y_\delta(\theta_\tau\omega)|^4)\, \d \tau}\|A^\frac32\bar{v}(s)\|^2\, \d s\nonumber\\
&\quad  \leq\frac{\nu}{2} \int_0^te^{\int_s^tC( \sum_{i=1}^2 \|A^\frac12\textbf{v}_i\|^4+\|A\textbf{v}_1\|^4+|y_\delta(\theta_\tau\omega)|^4)\, \d \tau}\|A^\frac32\bar{v}(s) \|^2\, \d s\nonumber\\
&\quad \leq e^{\int_{0}^tC( \sum_{i=1}^2 \|A^\frac12\textbf{v}_i\|^4+\|A\textbf{v}_1\|^4+|y_\delta(\theta_\tau \omega)|^4)\, \d \tau}\|A\bar{v}(0)\|^2, \nonumber
\end{align*}
and
\begin{align}
&
\int_{t-1}^t\|A\bar{v}\|^2 \left (\sum_{i=1}^2\|A^\frac32\textbf{v}_i\|^2+|y_\delta(\theta_\eta\omega)|^2\right) \d \eta\nonumber\\
&\quad \leq e^{\int_0^tC( \sum_{i=1}^2 \|A^\frac12\textbf{v}_i\|^4+\|A\textbf{v}_1\|^4+|y_\delta(\theta_\tau\omega)|^4)\, \d \tau}\|A\bar{v}(0)\|^2
\int_{t-1}^t \!  \left (\sum_{i=1}^2\|A^\frac32\textbf{v}_i\|^2+|y_\delta(\theta_\eta\omega)|^2\right) \d \eta.\nonumber
\end{align}
Inserting these inequalities  into \eqref{6.23}, we also can deduce that
\begin{align}
 \|A^\frac32\bar{v}(t)\|^2
&\leq Ce^{\int_0^tC( \sum_{i=1}^2 \|A^\frac12\textbf{v}_i\|^4+\|A\textbf{v}_1\|^4+|y_\delta(\theta_\tau\omega)|^4)\, \d \tau}\|A\bar{v}(0)\|^2\nonumber\\
&\int_{t-1}^t  \! \left (1+\sum_{i=1}^2\|A^\frac32\textbf{v}_i\|^2+|y_\delta(\theta_\eta\omega)|^2\right)  \d \eta   \nonumber
\end{align}
for any $t\geq 1$.
We replace $\omega$ by $\theta_{-t}\omega$ to  get  for any $ t\geq 1$
\begin{align} \label{7.6}
 &\|A^\frac32\bar{v}(t,\theta_{-t}\omega,\bar{v}(0))\|^2\nonumber\\
&\leq Ce^{\int_0^tC( \sum_{i=1}^2\|A^\frac12\textbf{v}_i(\tau,\theta_{-t}\omega,\textbf{v}_i(0))\|^4
+\|A\textbf{v}_i(\tau,\theta_{-t}\omega,\textbf{v}_i(0))\|^4
+|y_\delta(\theta_{\tau-t}\omega)|^4)\, \d \tau}\|A\bar{v}(0)\|^2\nonumber\\
&\quad  \times \int_{t-1}^t  \! \left ( \sum_{i=1}^2\|A^\frac32\textbf{v}_i(\eta,\theta_{-t}\omega,\textbf{v}_{i}(0))\|^2+|y_\delta(\theta_{\eta-t}\omega)|^2+1\right) \d \eta.
\end{align}

By \eqref{6.18},   let
\begin{align*}
T_\omega^* :=T_{\mathfrak B}^* (\omega)+1, \quad \omega\in \Omega,
\end{align*}
we can get that
\begin{align}
 & \int_{T_\omega^* -1}^{T_\omega^* }  \left ( \sum_{i=1}^2\|A^\frac32\textbf{v}_i(\eta,\theta_{-T_\omega^* }\omega,\textbf{v}_{i}(0))\|^2+|y_\delta(\theta_{\eta-T_\omega^* }\omega)|^2\right)  \d \eta  \nonumber\\
 &\quad
 \leq  C\zeta_{12}(\omega) + \int_{-1}^0 |y_\delta(\theta_{\eta}\omega)|^2 \, \d \eta =: \zeta_{13}(\omega) . \label{7.2}
 \end{align}

By \eqref{mar11.3},  for   $\textbf{v}(0) \in \mathfrak B(\theta_{-t}\omega)$ we also can deduce that
 \begin{align*}
 \int_0^t
 \|A^{\frac 12} \textbf{v}(\eta, \theta_{-t} \omega, \textbf{v}(0))\|^4 \, \d \eta
 &   \leq  \left( \| \mathfrak B  (\theta_{-t}\omega)\|^2_{H^1} +1 \right)^2 \int_0^t e^{ \int^\eta_0 C(1+|y_\delta(\theta_{\tau-t} \omega)|^6) \, \d \tau}  \, \d \eta \nonumber \\
 &   \leq  \big( \zeta_2(\theta_{-t}\omega)  +C \big)^2  e^{ Ct+C \int^0_{-t}  |y_\delta(\theta_{\tau} \omega)|^6 \, \d \tau}   ,
   \quad t>0  , \nonumber
\end{align*}
and, particularly for $t=T_\omega^*$,
 \begin{align}\label{7.3}
 & \int_0^{T_\omega^*}
 \|A^{\frac 12} \textbf{v}(\eta, \theta_{-{T_\omega^*}} \omega, \textbf{v}(0))\|^4 \, \d \eta \nonumber \\
 &   \quad  \leq  \big( \zeta_2(\theta_{-{T_\omega^*}}\omega)  +C \big)^2  e^{ C {T_\omega^*}+C \int^0_{-{T_\omega^*}}  |y_\delta(\theta_{\tau} \omega)|^6 \, \d \tau} =:\zeta_9(\omega).
\end{align}

On the other hand, by \eqref{6.4}, it yields
\begin{align*}
\frac{\d}{\d t}\|A\textbf{v}\|^2&\leq  C(1+\|A^{\frac{1}{2}}\textbf{v}\|^4+|y_\delta(\vartheta_t\omega)|^4)\|A\textbf{v}\|^2\nonumber\\
&+C(1+\|A^{\frac{1}{2}}\textbf{v}\|^6+|y_\delta(\vartheta_t\omega)|^6).
\end{align*}
we apply Gronwall's lemma to deduce
\begin{align*}
 \|A \textbf{v}(t,\omega, \textbf{v}(0))\|^2
  & \leq  e^{ \int^t_0 C(1+\|A^{\frac{1}{2}}\textbf{v}\|^4+|y_\delta(\theta_\tau\omega)|^4) \, \d \tau} \|A \textbf{v}(0)\|^2  \\
 &\quad
 +\int_0^t Ce^{ \int^t_s C(1+\|A^{\frac{1}{2}}\textbf{v}\|^4+|y_\delta(\theta_\tau\omega)|^2) \, \d \tau} \left (1+\|A^{\frac{1}{2}}\textbf{v}\|^6+ |y_\delta(\theta_s\omega)|^6 \right)  \d s \\
 &\leq  e^{ \int^t_0 C(1+\|A^{\frac{1}{2}}\textbf{v}\|^4+|y_\delta(\theta_\tau\omega)|^4) \, \d \tau} \left[ \|A \textbf{v}(0)\|^2
 +\int_0^t  C \left(1+\|A^{\frac{1}{2}}\textbf{v}\|^6+ |y_\delta(\theta_s\omega)|^6 \right)  \d s \right]  .
\end{align*}
Since
\begin{align*}
\int_0^t  C \left (1+ |y_\delta(\theta_s\omega)|^6 \right)  \d s \leq e^{ \int^t_0 C(1+|y_\delta(\theta_\tau\omega)|^6) \, \d \tau},
\end{align*}
then the estimate is given by
\begin{align*}
 \|A \textbf{v}(t,\omega, \textbf{v}(0))\|^2
   \leq  \left( \|A \textbf{v}(0)\|^2 +1 \right) e^{ \int^t_0 C(1+\|A^{\frac{1}{2}}\textbf{v}\|^6+|y_\delta(\theta_\tau\omega)|^6) \, \d \tau} ,
   \quad t>0.
\end{align*}
For   $\textbf{v}(0) \in \mathfrak B^2(\theta_{-t}\omega)$ we also can deduce that
 \begin{align*}
 \int_0^t
 \|A \textbf{v}(\eta, \theta_{-t} \omega, \textbf{v}(0))\|^4 \, \d \eta
 &   \leq  \left( \| \mathfrak B  (\theta_{-t}\omega)\|^2_{H^2} +1 \right)^2 \int_0^t e^{ \int^\eta_0 C(1+|y_\delta(\theta_{\tau-t} \omega)|^6) \, \d \tau}  \, \d \eta \nonumber \\
 &   \leq  \big( \zeta_{12}(\theta_{-t}\omega)  +C \big)^2  e^{ Ct+C \int^0_{-t}  |y_\delta(\theta_{\tau} \omega)|^6 \, \d \tau}   ,
   \quad t>0  , \nonumber
\end{align*}
and, particularly for $t=T_\omega^*$,
 \begin{align}\label{7.4}
 & \int_0^{T_\omega^*}
 \|A \textbf{v}(\eta, \theta_{-{T_\omega^*}} \omega, \textbf{v}(0))\|^4 \, \d \eta \nonumber \\
 &   \quad  \leq  \big( \zeta_{12}(\theta_{-{T_\omega^*}}\omega)  +C \big)^3  e^{ C {T_\omega^*}+C \int^0_{-{T_\omega^*}}  |y_\delta(\theta_{\tau} \omega)|^6 \, \d \tau} =:\zeta_{14}(\omega).
\end{align}
Finally, inserting     \eqref{7.2}, \eqref{7.3}  and \eqref{7.4}  into \eqref{7.6}  we can get at $t=T_\omega^*$ that
\begin{align}
 \|A^\frac32\bar{v}(T_\omega^* ,\theta_{-T_\omega^* }\omega,\bar{v}(0))\|^2
&\leq Ce^{  C (\zeta_9(\omega)+\zeta_{14}(\omega)) } \zeta_{13}(\omega) \|A^\frac32\bar{v}(0)\|^2. \nonumber
\end{align}
Hence, let
\begin{align*}
 L_5 (\mathfrak B^2, \omega) :=   Ce^{ C (\zeta_9(\omega)+\zeta_{14}(\omega)) } \zeta_{13}(\omega) ,\quad \omega\in \Omega. \qedhere
\end{align*}
This completes the proof of Lemma \ref{lemma7.3}.
\end{proof}

In this section, we will introduce the following main result.
\begin{theorem}[$(H, H^3)$-smoothing] \label{theorem7.1}
 Let Assumption \ref{assum1} hold and $f\in H^2$. For   any tempered set $\mathfrak D \in \D_H$, there
  are  random variables $T_{{\mathfrak D}}^* (\cdot)  $   and $ L_{\mathfrak D}^*(\cdot )$ such that  any two solutions $\textbf{v}_1$ and $\textbf{v}_2$ of the random anisotropic  NS equations \eqref{2.2} driven by colored noise corresponding to initial values   $\textbf{v}_{1,0},$ $ \textbf{v}_{2,0}$ in $\mathfrak D \left (\theta_{-T_{\mathfrak D}^*(\omega)}\omega \right)$, respectively, satisfy
  \begin{align}
  &
  \big\|\textbf{v}_1 \!  \big (T_{\mathfrak D}^*(\omega),\theta_{-T_{\mathfrak D}^*(\omega)}\omega,\textbf{v}_{1,0}\big)
  -\textbf{v}_2 \big (T_{\mathfrak D}^*(\omega),\theta_{-T_{\mathfrak D}^*(\omega)}\omega,\textbf{v}_{2,0} \big) \big \|{^2_{H^3}}  \notag \\[0.8ex]
&\quad \leq  L_{\mathfrak D}^* (\omega)\|\textbf{v}_{1,0}-\textbf{v}_{2,0}\|^2,\quad \omega\in \Omega.\label{sep5.1}
\end{align}
\end{theorem}
\begin{proof}
By virtue of Theorem \ref{theorem5.6}, for any the tempered set   $\mathfrak D\in \D_H$ there are random variables $ T_{\mathfrak D}(\cdot )  $ and $L_{\mathfrak D}(\cdot) $ such that
  \begin{align} \label{7.5}
  \left \| \bar v\left (T_{\mathfrak D} ,\theta_{-T_{\mathfrak D}}\omega,  \bar v(0)\right)
  \right \|^2_{H^2}
 \leq L_{\mathfrak D}( \omega) \|\bar {v} (0)\|^2  ,
  \end{align}
where $T_{\mathfrak D}= T_{\mathfrak D}(\mathfrak D, \omega)$. Moreover, $ \textbf{v}\left (T_{\mathfrak D} ,\theta_{-T_{\mathfrak D}}\omega,  \textbf{v}_{0}\right) \in \mathfrak B(\omega)$ for any $\textbf{v}_0\in \mathfrak D(\theta_{-T_{\mathfrak D}}\omega) $, $\omega\in \Omega$.
Hence, by Lemma \ref{lemma7.3} and \eqref{7.5}, we also can deduce that
 \begin{align*}
  \left \| \bar v \big( T_\omega +T_{\mathfrak D}, \theta_{-T_\omega-T_{\mathfrak D}}
 \omega, \bar v(0) \big)
   \right \|^2_{H^3}
 &  =  \left \| \bar v \big( T_\omega  , \theta_{-T_\omega }
 \omega, \bar v  ( T_{\mathfrak D}, \theta_{-T_\omega-T_{\mathfrak D}}
 \omega, \bar v(0)  )\big)
   \right \|^2_{H^3}  \\
   &\leq  L_5(\mathfrak B^2, \omega) \left \|  \bar v \big  ( T_{\mathfrak D}, \theta_{-T_\omega-T_{\mathfrak D}}
 \omega, \bar v(0)  \big)
   \right \|^2_{H^2}  \  \text{(by Lemma \ref{lemma7.3})}\\
   &\leq L_5(\mathfrak B^2,\omega)  L_{\mathfrak D} \big (\mathfrak D, \theta_{-T_\omega}\omega \big) \|\bar v(0)\|^2  \  \text{(by \eqref{7.5})} ,
 \end{align*}
where $T_{\mathfrak D}=T_{\mathfrak D}(\theta_{-T_\omega}  \omega)$,
uniformly for any $\textbf{v}_{1,0},$ $ \textbf{v}_{2,0}\in \B(\theta_{-T_\omega-T_{\mathfrak D}(\theta_{-T_\omega} \omega)}\omega)$, $  \omega\in \Omega$. Thence,
let  the random variables is given by
 \begin{align*}
  & T_{\mathfrak D}^*( \omega) := T_\omega +T_{\mathfrak D}(\theta_{-T_\omega} \omega) , \\
  &L_{\mathfrak D}^*(\omega) :=  L_5(\mathfrak B^2,\omega)   L_{\mathfrak D} \big (\mathfrak D, \theta_{-T_\omega}\omega \big).
 \end{align*}
This completes the proof of Theorem \ref{theorem7.1}.
  \end{proof}
By virtue of the Theorem \ref{theorem5.1}, we use the similar method to get the following main result.
\begin{theorem} \label{theorem7.2}
 Let Assumption \ref{assum1} hold and $f\in H^2$. Then the  RDS $\phi$ generated by the the random anisotropic  NS equations \eqref{2.2} driven by colored noise  has  a tempered $(H,H^3)$-random attractor $\mathcal A $. In addition, $\mathcal A$ has finite fractal dimension in $H^3$: there exists   a constant $ d>0$ such that
\[
d_f^{H^3} \! \big ( \A(\omega) \big)  \leq d, \quad \omega \in \Omega.
\]
\end{theorem}

\section{Convergence of solutions}\label{sec8}

In this section, we introduce the following stochastic two-dimensional anisotropic Navier-Stokes (NS) equations driven by additive white noise  on  $\mathbb{T}^2=[0,L]^2$, $L>0$:
 \begin{align} \label{8.1}
 \left \{
  \begin{aligned}
 &  \frac{\partial \textbf{u}}{\partial t}-  \nu \left(
\begin{array}{ccc}
\partial_{yy}u\\ \partial_{xx}v
\end{array}\right)+(\textbf{u}\cdot\nabla)\textbf{u}+\nabla p=  f(\textbf{x}) +h(\textbf{x})\frac{\d W}{\d t},\\
  & \nabla\cdot \textbf{u}=0,
  \end{aligned} \right.
\end{align}
here, $\textbf{u}:=\textbf{u}_{h}$, $\textbf{u}_{h}$ given by in Section \ref{1}. We introduce the following transformation
 \begin{align*}
\textbf{u}(t)=\textbf{v}(t)+hz(\theta_t\omega),\quad t>0,\ \omega\in \Omega,
\end{align*}
then it yields
\begin{align}\label{8.5}
\left\{
\begin{aligned}
&
\frac{\d \textbf{v}}{\d t}+ \nu A_1\textbf{v}+B\big( \textbf{v}(t)+hz(\theta_t\omega)\big)=f(\textbf{x})-  \nu A_1hz(\theta_t\omega)+hz(\theta_t\omega), \\
& \textbf{v}(0)=\textbf{u}(0)-hz (\omega).
\end{aligned}
\right.
\end{align}

For $\delta>0$, we introduce the following pathwise random equations
 \begin{align} \label{8.2}
 \left \{
  \begin{aligned}
 &  \frac{\partial \textbf{u}_{\delta}}{\partial t}-  \nu \left(
\begin{array}{ccc}
\partial_{yy}u_{\delta}\\ \partial_{xx}v_{\delta}
\end{array}\right)+(\textbf{u}_{\delta}\cdot\nabla)\textbf{u}_{\delta}+\nabla p=  f(\textbf{x}) +h(\textbf{x})\zeta_{\delta}(\theta_t\omega),\\
  & \nabla\cdot \textbf{u}_{\delta}=0.
  \end{aligned} \right.
\end{align}
By the Proposition \ref{proposition3.5}, we introduce a new variable
 \begin{align}\label{8.3}
\textbf{u}_\delta(t)=\textbf{v}_\delta(t)+hy_\delta(\theta_t\omega),\quad t>0,\ \omega\in \Omega,
\end{align}
Applying \eqref{8.2} and \eqref{8.3}, it yields
\begin{align}\label{8.4}
\left\{
\begin{aligned}
&
\frac{\d \textbf{v}_\delta}{\d t}+ \nu A_1\textbf{v}_\delta+B\big( \textbf{v}_\delta(t)+hy_\delta(\theta_t\omega)\big)=f(\textbf{x})-  \nu A_1hy_\delta(\theta_t\omega)+hy_\delta (\theta_t\omega), \\
& \textbf{v}_\delta(0)=\textbf{u}(0)-hy_\delta (\omega).
\end{aligned}
\right.
\end{align}
Inspired by the Theorem 3.7 in \cite{gu2020} and the Theorem 4.6 in \cite{gu2019}, by the similar method and Lemma \ref{lemma3.4} and Proposition \ref{proposition3.5}, we introduce the upper semi-continuity of random attractors as $\delta\rightarrow0$.
\begin{theorem}
 Let Assumption \ref{assum} hold and $f\in H$. Then for $\omega\in\Omega$, we have
 \begin{align*}
 \lim\limits_{\delta\rightarrow 0} \dist_{H}(\A_{\delta}(\tau,\omega),\A_0(\tau,\omega))=0.
 \end{align*}
\end{theorem}

\subsection{$(H,H)$-convergence on $\mathfrak D$}
In this subsection, by \eqref{8.5} and \eqref{8.4}, we introduce the following system
\begin{align}\label{8.6}
&
\frac{\d (\textbf{v}_\delta-\textbf{v})}{\d t}+ \nu A_1(\textbf{v}_\delta-\textbf{v})+B\big( \textbf{v}_\delta(t)+hy_\delta(\theta_t\omega)\big)-B\big( \textbf{v}+hz(\theta_t\omega)\big)\nonumber\\
&=-  (\nu A_1hy_\delta(\theta_t\omega)-\nu A_1hz(\theta_t\omega))+hy_\delta (\theta_t\omega)-hz(\theta_t\omega),
\end{align}
here, we set $\bar{\textbf{v}}=\textbf{v}_\delta-\textbf{v}$.
\begin{theorem}\label{lemma8.1}
 Let Assumption \ref{assum} hold and $f\in H$.
 For  any tempered set $\mathfrak D\in \D_H $, then there exists random variables $t_{\mathfrak D}^1(\cdot) $ such that any
 two solutions $\textbf{u}$ and $\textbf{u}_\delta$ of the random anisotropic  NS equations \eqref{8.1} and \eqref{8.2} driven by  white noise  and colored noise corresponding to initial values $\textbf{u}_{0}$  in $ \mathfrak D( \theta_{-t_{\mathfrak D}^1(\omega)}\omega),$   respectively,  satisfy
\begin{align*}
\lim\limits_{\delta\rightarrow0}\left \| \textbf{u}_\delta \! \left ( t_{\mathfrak D} (\omega) ,\theta_{-t_{\mathfrak D}(\omega)}\omega,  \textbf{u}_{0}\right)
 - \textbf{u} \! \left ( t_{\mathfrak D} (\omega) ,\theta_{-t_{\mathfrak D}(\omega)}\omega, \textbf{u}_{0}\right) \right \|^2  =0,\quad \omega \in \Omega.
\end{align*}

\end{theorem}
\begin{proof}
Taking the inner product of the system   \eqref{8.6} with $\textbf{v}_\delta-\textbf{v}$ in $H$,  by integration by parts we can deduce that
\begin{align}\label{8.7}
&\frac{\d }{\d t}\|\textbf{v}_\delta-\textbf{v}\|^2+2\nu(A_1(\textbf{v}_\delta-\textbf{v}),\textbf{v}_\delta-\textbf{v})\nonumber\\
&=-2(B\big( \textbf{v}_\delta(t)+hy_\delta(\theta_t\omega)\big)-B\big( \textbf{v}+hz(\theta_t\omega)\big),\textbf{v}_\delta-\textbf{v})\nonumber\\
&-  2(\nu A_1hy_\delta(\theta_t\omega)-\nu A_1hz(\theta_t\omega),\textbf{v}_\delta-\textbf{v})
+2(hy_\delta (\theta_t\omega)-hz(\theta_t\omega),\textbf{v}_\delta-\textbf{v}).
\end{align}
By \eqref{3.9} and the similar method of \eqref{4.3} yields
\begin{align}\label{8.8}
\nu\|A^\frac12(\textbf{v}_\delta-\textbf{v})\|^2\leq2\nu(A_1(\textbf{v}_\delta-\textbf{v}),\textbf{v}_\delta-\textbf{v}).
\end{align}
For the first term on the right hand side of \eqref{8.7}, applying the H\"{o}lder inequality, the Gagliardo-Nirenberg inequality and the Young inequality, we can get that
\begin{align}\label{8.9}
&-2(B\big( \textbf{v}_\delta(t)+hy_\delta(\theta_t\omega)\big)-B\big( \textbf{v}+hz(\theta_t\omega)\big),\textbf{v}_\delta-\textbf{v})\nonumber\\
&=-2(B\big(\textbf{v}+hz(\theta_t\omega),\textbf{v}_\delta(t)-\textbf{v}+h(y_\delta(\theta_t\omega)-z(\theta_t\omega)))
,\textbf{v}_\delta-\textbf{v}\big)\nonumber\\
&-2(B\big( \textbf{v}_\delta(t)-\textbf{v}+h(y_\delta(\theta_t\omega)-z(\theta_t\omega)),\textbf{v}_\delta(t)+hy_\delta(\theta_t\omega))
,\textbf{v}_\delta-\textbf{v}\big)\nonumber\\
&=-2(B\big(\textbf{v}+hz(\theta_t\omega),h(y_\delta(\theta_t\omega)-z(\theta_t\omega)))
,\textbf{v}_\delta-\textbf{v}\big)\nonumber\\
&-2(B\big( \textbf{v}_\delta(t)-\textbf{v},\textbf{v}_\delta(t)+hy_\delta(\theta_t\omega))
,\textbf{v}_\delta-\textbf{v}\big)\nonumber\\
&-2(B\big( h(y_\delta(\theta_t\omega)-z(\theta_t\omega)),\textbf{v}_\delta(t)+hy_\delta(\theta_t\omega))
,\textbf{v}_\delta-\textbf{v}\big)\nonumber\\
&\leq 2\|\nabla(\textbf{v}_\delta(t)-\textbf{v})\|\|\textbf{v}(t)+hz(\theta_t\omega)\|_{L^4}
\|h(y_\delta(\theta_t\omega)-z(\theta_t\omega))\|_{L^4}\nonumber\\
&+2\|\textbf{v}_\delta(t)-\textbf{v}\|_{L^4}^2\|\nabla(\textbf{v}_\delta(t)+hy_\delta(\theta_t\omega))\|\nonumber\\
&+2\|\textbf{v}_\delta(t)-\textbf{v}\|_{L^4}\|\nabla(\textbf{v}_\delta(t)+hy_\delta(\theta_t\omega))\|
\|h(y_\delta(\theta_t\omega)-z(\theta_t\omega))\|_{L^4}\nonumber\\
&\leq C\|\nabla(\textbf{v}_\delta(t)-\textbf{v})\|
|y_\delta(\theta_t\omega)-z(\theta_t\omega)|(\|\nabla\textbf{v}\|+|z(\theta_t\omega)|)\nonumber\\
&+C\|\textbf{v}_\delta(t)-\textbf{v}\|\|A^\frac12(\textbf{v}_\delta(t)-\textbf{v})\|
(\|\nabla\textbf{v}_\delta(t)\|+|y_\delta(\theta_t\omega)|)
\nonumber\\
&+C\|\|A^\frac12(\textbf{v}_\delta(t)-\textbf{v})\||y_\delta(\theta_t\omega)-z(\theta_t\omega)|
(\|\nabla\textbf{v}_\delta(t)\|+|y_\delta(\theta_t\omega)|)
\nonumber\\
&\leq \frac{\nu}{4}\|A^\frac12(\textbf{v}_\delta(t)-\textbf{v})\|^2
+C(\|A^\frac12\textbf{v}_\delta(t)\|^2+|y_\delta(\theta_t\omega)|^2)\|\textbf{v}_\delta-\textbf{v}\|^2\nonumber\\
&+C(\|A^\frac12\textbf{v}_\delta(t)\|^2+\|A^\frac12\textbf{v}(t)\|^2+|y_\delta(\theta_t\omega)|^2
+|z(\theta_t\omega)|^2)|y_\delta(\theta_t\omega)-z(\theta_t\omega)|^2.
\end{align}
For the last two terms in the right hand side of \eqref{8.7}, we also can deduce that
\begin{align}\label{8.10}
&|2(hy_\delta (\theta_t\omega)-hz(\theta_t\omega),\textbf{v}_\delta-\textbf{v})-  2(\nu A_1hy_\delta(\theta_t\omega)-\nu A_1hz(\theta_t\omega),\textbf{v}_\delta-\textbf{v})|\nonumber\\
&\leq 2\|h\||y_\delta (\theta_t\omega)-z(\theta_t\omega)|\|\textbf{v}_\delta-\textbf{v}\|
+2\|A^\frac12 h\||y_\delta (\theta_t\omega)-z(\theta_t\omega)|\|A^\frac12(\textbf{v}_\delta-\textbf{v})\|\nonumber\\
&\leq\frac{\nu}{4}\|A^\frac12(\textbf{v}_\delta-\textbf{v})\|^2+
C\|A^\frac12 h\|^2|y_\delta (\theta_t\omega)-z(\theta_t\omega)|^2\nonumber\\
&\leq\frac{\nu}{4}\|A^\frac12(\textbf{v}_\delta-\textbf{v})\|^2+
C|y_\delta (\theta_t\omega)-z(\theta_t\omega)|^2.
\end{align}
Inserting \eqref{8.8}, \eqref{8.9} and \eqref{8.10} into \eqref{8.7}, it is easy to get that
\begin{align}\label{8.11}
&\frac{\d }{\d t}\|\bar{\textbf{v}}\|^2+\frac{\nu}{2}\|A^\frac12\bar{\textbf{v}}\|^2\nonumber\\
&\leq
C(\|A^\frac12\textbf{v}_\delta(t)\|^2+|y_\delta(\theta_t\omega)|^2)\|\bar{\textbf{v}}\|^2\nonumber\\
&+C(1+\|A^\frac12\textbf{v}_\delta(t)\|^2+\|A^\frac12\textbf{v}(t)\|^2+|y_\delta(\theta_t\omega)|^2
+|z(\theta_t\omega)|^2)|y_\delta(\theta_t\omega)-z(\theta_t\omega)|^2.
\end{align}

Applying Gronwall's lemma to \eqref{8.11}, we can  deduce that for $t>0$
\begin{align}\label{8.12}
 & \|\bar{\textbf{v}}(t)\|^2+  \frac{\nu}{2} \int_0^te^{\int_s^tC(\|A^\frac12\textbf{v}_\delta(\tau) \|^2+|y_\delta(\theta_\tau\omega)|^2)\, \d \tau}\|A^{\frac12}\bar{\textbf{v}} (s) \|^2\, \d s \nonumber \\
& \quad
\leq e^{\int_0^tC(\|A^\frac12\textbf{v}_\delta(\tau)\|^2+|y_\delta(\theta_\tau\omega)|^2)\, \d \tau}\|\bar{\textbf{v}}(0)\|^2\nonumber \\
& \quad
+ C\int_0^te^{\int_s^tC(\|A^\frac12\textbf{v}_\delta(\tau) \|^2+|y_\delta(\theta_\tau\omega)|^2)\, \d \tau}(1+\|A^\frac12\textbf{v}_\delta\|^2+\|A^\frac12\textbf{v}\|^2+|y_\delta(\theta_s\omega)|^2
+|z(\theta_s\omega)|^2)\nonumber \\
&\cdot|y_\delta(\theta_s\omega)-z(\theta_s\omega)|^2\d s\nonumber \\
& \quad
\leq e^{\int_0^tC(\|A^\frac12\textbf{v}_\delta(\tau)\|^2+|y_\delta(\theta_\tau\omega)|^2)\, \d \tau}(\|\bar{\textbf{v}}(0)\|^2\nonumber \\
& \quad
+ C\int_0^t(1+\|A^\frac12\textbf{v}_\delta\|^2+\|A^\frac12\textbf{v}\|^2+|y_\delta(\theta_s\omega)|^2
+|z(\theta_s\omega)|^2)|y_\delta(\theta_s\omega)-z(\theta_s\omega)|^2\d s).
\end{align}

By virtue of \eqref{4.29}, we   can deduce that
\begin{align*}
   \int_0^t \|A^\frac12\textbf{v}_\delta(\tau,\theta_{-t}\omega, \textbf{v}_{\delta,0})\|^2 \, \d \tau &\leq e^{ \left(\frac 4 \alpha -3  \right)\lambda t}  \int_0^t e^{ \left( \frac 4 \alpha -3  \right) \lambda (\tau-t)} \|A^\frac12\textbf{v}_\delta (\tau,\theta_{-t}\omega, \textbf{v}_{\delta,0})\|^2  \,  \d \tau  \\
 &   \leq \frac {2}{\alpha\nu} e^{ \left(\frac 4 \alpha -3  \right)\lambda t}  \left( e^{-\lambda t} \|\textbf{v}_{\delta,0}\|^2 +\zeta_1(\omega) \right) ,\quad t\geq T_1(\omega),
 \end{align*}
so for \eqref{8.12}, replacing $\omega$ by $\theta_{-t}\omega$, we also can  get that
\begin{align*}
 \|\bar{\textbf{v}}(t,\theta_{-t}\omega, \bar{\textbf{v}} (0))\|^2
 & \leq  e^{\int_0^tC \left (1+\|A^\frac12\textbf{v}_\delta(\tau,\theta_{-t}\omega, \textbf{v}_{\delta,0})\|^2+\|A^\frac12\textbf{v}(\tau,\theta_{-t}\omega, \textbf{v}_{\delta,0})\|^2+|y_\delta(\theta_{\tau-t} \omega)|^2+|z(\theta_{\tau-t} \omega)|^2 \right)   \d \tau }\nonumber\\
 &\cdot(\|\bar{\textbf{v}}(0)\|^2+\sup\limits_{s\in[0,t]}|y_\delta(\theta_s\omega)-z(\theta_s\omega)|^2) \\
 &\leq e^{C  e^{ \left(\frac 4 \alpha -3  \right)\lambda t}  \left( e^{-\lambda t} \|\textbf{v}_{\delta,0}\|^2+e^{-\lambda t} \|\textbf{v}_{0}\|^2 +\zeta_1(\omega) \right) +C\int_{-t}^0|y_\delta(\theta_{\tau} \omega)|^2 +C\int_{-t}^0|z(\theta_{\tau} \omega)|^2\, \d \tau } \nonumber\\
 &\cdot(\|\bar{\textbf{v}}(0)\|^2+\sup\limits_{s\in[0,t]}|y_\delta(\theta_s\omega)-z(\theta_s\omega)|^2).
\end{align*}
for any $t\geq T_1 (\omega)$.
Since $\textbf{v}_{\delta,0} $ and $\textbf{v}_{0} $ belong to $\mathfrak D(\theta_{-t}\omega) $ which is tempered, there exists a random variable $ t_{\mathfrak D}^1 (\omega)  \geq T_1(\omega)$  such that
\begin{align*}
 e^{-\lambda t_{\mathfrak D}^1(\omega) } (\|\textbf{v}_{\delta,0}\|^2+\|\textbf{v}_{0}\|^2) \leq  e^{-\lambda t_{\mathfrak D}^1(\omega) } 2\left \|  \, \mathfrak D \!  \left (\theta_{-t_{\mathfrak D}(\omega)} \omega \right ) \right\|^2 \leq 1,\quad \omega\in \Omega.
\end{align*}
Moreover, since $\mathfrak D$ is pullback absorbed by the absorbing set $\mathfrak B$,  this exists a large enough $t_{\mathfrak D}^1(\omega)$  such that
\begin{align}
 \phi \left ( t_{\mathfrak D}^1(\omega) , \theta_{-t_{\mathfrak D}^1(\omega)}\omega, \mathfrak D(\theta_{- t_{\mathfrak D}(\omega) }^1\omega) \right ) \subset \mathfrak B(\omega), \quad \omega\in \Omega .
\end{align}
Therefore,  a random variable given by as follows
\begin{align*}
\bar{L}_1({\mathfrak D}, \omega)=e^{C  e^{ \left(\frac 4 \alpha -3  \right)\lambda  t_{\mathfrak D}^1(\omega)}    ( 1+\zeta_1(\omega)  ) + C\int_{-t_{\mathfrak D}^1(\omega)}^0|y_\delta(\theta_{\tau} \omega)|^2+ C\int_{-t_{\mathfrak D}^1(\omega)}^0|z(\theta_{\tau} \omega)|^2 \, \d \tau } ,\quad \omega\in \Omega,
\end{align*}
then we have
\begin{align}
\left \|\bar{\textbf{v}} \! \left ( t_{\mathfrak D}^1 (\omega) ,\theta_{-t_{\mathfrak D}^1(\omega)}\omega, \bar{\textbf{v}}(0) \right) \right \|^2
 \leq \bar{L}_1({\mathfrak D}, \omega)\left (\|\bar{\textbf{v}}(0)\|^2+\sup\limits_{s\in[0,t_{\mathfrak D}^1]}|y_\delta(\theta_s\omega)-z(\theta_s\omega)|^2\right ) ,
\end{align}
whenever $\textbf{v}_{\delta,0},$ $\textbf{v}_{0}\in \mathfrak D \big ( \theta_{-t_{\mathfrak D}^1( \omega)}\omega \big),$  $\omega\in \Omega$. Then it yields
\begin{align*}
&\left \| \textbf{u}_\delta \! \left ( t_{\mathfrak D}^1 (\omega) ,\theta_{-t_{\mathfrak D}^1(\omega)}\omega,  \textbf{u}_{0}\right)
 - \textbf{u} \! \left ( t_{\mathfrak D}^1 (\omega) ,\theta_{-t_{\mathfrak D}^1(\omega)}\omega, \textbf{u}_{0}\right) \right \|^2
 \nonumber\\
& =\left \| \textbf{v}_\delta \! \left ( t_{\mathfrak D}^1 (\omega) ,\theta_{-t_{\mathfrak D}^1(\omega)}\omega,  \textbf{v}_{\delta,0}\right)
 - \textbf{v} \! \left ( t_{\mathfrak D}^1 (\omega) ,\theta_{-t_{\mathfrak D}^1(\omega)}\omega, \textbf{v}_{0}\right)+h(y_\delta(\omega)-z(\omega)) \right \|^2\nonumber\\
&\leq 2\left \| \textbf{v}_\delta \! \left ( t_{\mathfrak D}^1 (\omega) ,\theta_{-t_{\mathfrak D}^1(\omega)}\omega,  \textbf{v}_{\delta,0}\right)
 - \textbf{v} \! \left ( t_{\mathfrak D}^1 (\omega) ,\theta_{-t_{\mathfrak D}^1(\omega)}\omega, \textbf{v}_{0}\right)\right\|^2+2\left\|h\left (y_\delta(\omega)-z(\omega)\right ) \right \|^2,
\end{align*}
where $\textbf{v}_{\delta,0}=\textbf{u}_0-hy_\delta(\omega)$, $\textbf{v}_{0}=\textbf{u}_0-hz(\omega)$  and $\bar{\textbf{v}}(0)=-h(y_\delta(\omega)-z(\omega))$. Moreover, we also have
\begin{align*}
&\left \| \textbf{u}_\delta \! \left ( t_{\mathfrak D}^1 (\omega) ,\theta_{-t_{\mathfrak D}^1(\omega)}\omega,  \textbf{u}_{0}\right)
 - \textbf{u} \! \left ( t_{\mathfrak D}^1 (\omega) ,\theta_{-t_{\mathfrak D}^1(\omega)}\omega, \textbf{u}_{0}\right) \right \|^2\nonumber\\
 &\leq C\bar{L}_1({\mathfrak D}, \omega)\left (\|\bar{\textbf{v}}(0)\|^2+\sup\limits_{s\in[0,t_{\mathfrak D}^1]}|y_\delta(\theta_s\omega)-z(\theta_s\omega)|^2\right )\nonumber\\
 &\leq C\bar{L}_1({\mathfrak D}, \omega)\left (|y_\delta(\omega)-z(\omega)|^2+\sup\limits_{s\in[0,t_{\mathfrak D}^1]}|y_\delta(\theta_s\omega)-z(\theta_s\omega)|^2\right ).
\end{align*}
By \eqref{2.4} and \eqref{2.5}, this completes the proof of Theorem \ref{lemma8.1}.

\end{proof}
\subsection{$(H,H^1)$-convergence on $\mathfrak B$}
\begin{theorem}\label{lemma8.2}[$(H,H^1)$-convergence on $\mathfrak B$]
Let Assumption \ref{assum} hold and $f\in H$.
 For the random absorbing set $\mathfrak B$ given by \eqref{4.19}, then there exists random variables $\tau_\omega^1(\cdot) $ such that  any
 two solutions $\textbf{u}$ and $\textbf{u}_\delta$ of the random anisotropic  NS equations \eqref{8.1} and \eqref{8.2} driven by  white noise  and colored noise corresponding to initial values $\textbf{u}_{0}$  in $ \mathfrak D( \theta_{-\tau_\omega^1(\omega)}\omega),$   respectively,  satisfy
\begin{align*}
\lim\limits_{\delta\rightarrow0}\left \| \textbf{u}_\delta \! \left ( \tau_\omega^1 (\omega) ,\theta_{-\tau_\omega^1(\omega)}\omega,  \textbf{u}_{0}\right)
 - \textbf{u} \! \left ( \tau_\omega^1 (\omega) ,\theta_{-\tau_\omega^1(\omega)}\omega, \textbf{u}_{0}\right) \right \|^2_{H^1}  =0,\quad \omega \in \Omega.
\end{align*}
\end{theorem}

\begin{proof}
Taking the inner product of the system   \eqref{8.6} with $A\bar{\textbf{v}}$ in $H$,  by integration by parts, we can deduce that
\begin{align}\label{8.13}
&\frac{\d}{\d t}\|A^\frac12\bar{\textbf{v}}\|^2+ 2\nu \int_{\mathbb{T}^2}\partial_{yy}\bar{\textbf{v}}_1\Delta \bar{\textbf{v}}_1dx+2\nu \int_{\mathbb{T}^2}\partial_{xx}\bar{\textbf{v}}_2\Delta \bar{\textbf{v}}_2dx\nonumber\\
&=-2(B\big( \textbf{v}_\delta(t)+hy_\delta(\theta_t\omega)\big)-B\big( \textbf{v}+hz(\theta_t\omega)\big),A(\textbf{v}_\delta-\textbf{v}))\nonumber\\
&-  2(\nu A_1h(y_\delta(\theta_t\omega)-z(\theta_t\omega)),A(\textbf{v}_\delta-\textbf{v}))
+2(h(y_\delta (\theta_t\omega)-z(\theta_t\omega)),A(\textbf{v}_\delta-\textbf{v}))\nonumber\\
&=-2(B\big(\textbf{v}+hz(\theta_t\omega),\textbf{v}_\delta(t)-\textbf{v}+h(y_\delta(\theta_t\omega)-z(\theta_t\omega)))
,A(\textbf{v}_\delta-\textbf{v})\big)\nonumber\\
&-2(B\big( \textbf{v}_\delta(t)-\textbf{v}+h(y_\delta(\theta_t\omega)-z(\theta_t\omega)),\textbf{v}_\delta(t)+hy_\delta(\theta_t\omega))
,A(\textbf{v}_\delta-\textbf{v})\big)\nonumber\\
&-  2(\nu A_1h(y_\delta(\theta_t\omega)-z(\theta_t\omega)),A(\textbf{v}_\delta-\textbf{v}))
+2(h(y_\delta (\theta_t\omega)-z(\theta_t\omega)),A(\textbf{v}_\delta-\textbf{v})).
\end{align}
By direct computations and similar method, we also have
\begin{align}\label{8.14}
&\nu \int_{\mathbb{T}^2}\partial_{yy}\bar{\textbf{v}}_1\Delta \bar{\textbf{v}}_1dx+\nu \int_{\mathbb{T}^2}\partial_{xx}\bar{\textbf{v}}_2\Delta \bar{\textbf{v}}_2dx\nonumber\\
&\geq \frac12\nu \int_{\mathbb{T}^2}(\partial_{xx}\bar{\textbf{v}}_1+\partial_{yy}\bar{\textbf{v}}_1)^2dx
+\frac12\nu \int_{\mathbb{T}^2}(\partial_{xx}\bar{\textbf{v}}_2+\partial_{yy}\bar{\textbf{v}}_2)^2dx
\nonumber\\
&=\frac12\nu\|\Delta \bar{\textbf{v}}\|^2=\frac12\nu\|A \bar{\textbf{v}}\|^2.
    \end{align}
For the first term and the second term on the right hand side of \eqref{8.13}, by using the H\"{o}lder inequality, the Gagliardo-Nirenberg inequality and the Young inequality, we can get that
\begin{align}
 &\big| 2(B\big(\textbf{v}+hz(\theta_t\omega),\textbf{v}_\delta(t)-\textbf{v}+h(y_\delta(\theta_t\omega)-z(\theta_t\omega)))
,A(\textbf{v}_\delta-\textbf{v})\big)\big|\nonumber\\
 &\leq 2\|\textbf{v}+hz(\theta_t\omega)\|_{L^4}\|\nabla(\textbf{v}_\delta(t)-\textbf{v})\|_{L^4}\|A\bar{\textbf{v}}\|\nonumber\\
& +2\|\textbf{v}+hz(\theta_t\omega)\|_{L^4}\|\nabla(h(y_\delta(\theta_t\omega)-z(\theta_t\omega)))\|_{L^4}\|A\bar{\textbf{v}}\|
\nonumber\\
&\leq C\|A^\frac12(\textbf{v}+hz(\theta_t\omega))\|\|A^\frac12\bar{\textbf{v}}\|^\frac12\|A\bar{\textbf{v}}\|^\frac32\nonumber\\
& +C\|A^\frac12(\textbf{v}+hz(\theta_t\omega))\|\|Ah\||y_\delta(\theta_t\omega)-z(\theta_t\omega)|\|A\bar{\textbf{v}}\|\nonumber\\
&\leq\frac{\nu}{8}\|A\bar{\textbf{v}}\|^2+C(\|A^\frac12\textbf{v}\|^4+|z(\theta_t\omega)|^4)\|A^\frac12\bar{\textbf{v}}\|^2
\nonumber\\
& +C(\|A^\frac12\textbf{v}\|^2+|z(\theta_t\omega)|^2)|y_\delta(\theta_t\omega)-z(\theta_t\omega)|^2,
\end{align}
and
\begin{align}
 &\big|2(B\big( \textbf{v}_\delta(t)-\textbf{v}+h(y_\delta(\theta_t\omega)-z(\theta_t\omega)),\textbf{v}_\delta(t)+hy_\delta(\theta_t\omega))
,A(\textbf{v}_\delta-\textbf{v})\big)\big| \nonumber\\
&\leq 2\|\bar{\textbf{v}}\|_{L^\infty}\|\nabla(\textbf{v}_\delta(t)+hy_\delta(\theta_t\omega))\|\|A\bar{\textbf{v}}\|\nonumber\\
&+2\|h(y_\delta(\theta_t\omega)-z(\theta_t\omega))\|_{L^\infty}\|\nabla(\textbf{v}_\delta(t)+hy_\delta(\theta_t\omega))\|
\|A\bar{\textbf{v}}\|\nonumber\\
&\leq C\|A^\frac12\bar{\textbf{v}}\|^\frac12\|A\bar{\textbf{v}}\|^\frac32(\|A^\frac12\textbf{v}_\delta(t)\|+\|A^\frac12h\|
|y_\delta(\theta_t\omega)|)\nonumber\\
&+C\|Ah\||y_\delta(\theta_t\omega)-z(\theta_t\omega)|\|A\bar{\textbf{v}}\|(\|A^\frac12\textbf{v}_\delta(t)\|+\|A^\frac12h\|
|y_\delta(\theta_t\omega)|)
\nonumber\\
&\leq \frac{\nu}{8}\|A\bar{\textbf{v}}\|^2+C(\|A^\frac12\textbf{v}_\delta(t)\|^4+|y_\delta(\theta_t\omega)|^4)
\|A^\frac12\bar{\textbf{v}}\|^2
\nonumber\\
& +C(\|A^\frac12\textbf{v}_\delta(t)\|^2+
|y_\delta(\theta_t\omega)|^2)|y_\delta(\theta_t\omega)-z(\theta_t\omega)|^2.
\end{align}
For the last two terms on the right hand side of \eqref{8.13}, by using the similar method, we also can deduce that
\begin{align}\label{8.15}
&-  2(\nu A_1h(y_\delta(\theta_t\omega)-z(\theta_t\omega)),A(\textbf{v}_\delta-\textbf{v}))
+2(h(y_\delta (\theta_t\omega)-z(\theta_t\omega)),A(\textbf{v}_\delta-\textbf{v}))\nonumber\\
&\leq 2\nu\|Ah\||y_\delta(\theta_t\omega)-z(\theta_t\omega)|\|A\bar{\textbf{v}}\|
+2\|h\||y_\delta(\theta_t\omega)-z(\theta_t\omega)|\|A\bar{\textbf{v}}\|
\nonumber\\
&\leq \frac{\nu}{4}\|A\bar{\textbf{v}}\|^2+C|y_\delta(\theta_t\omega)-z(\theta_t\omega)|^2.
\end{align}

Inserting \eqref{8.14}-\eqref{8.15} into \eqref{8.13}, we can deduce that
\begin{align}\label{9.5}
&\frac{\d}{\d t}\|A^\frac12\bar{\textbf{v}}\|^2+  \frac{\nu}{2} \|A\bar{\textbf{v}}\|^2\nonumber\\
&\leq C(\|A^\frac12\textbf{v}\|^4+\|A^\frac12\textbf{v}_\delta\|^4+|y_\delta(\theta_t\omega)|^4+|z(\theta_t\omega)|^4)
\|A^\frac12\bar{\textbf{v}}\|^2\nonumber\\
&+ C(1+\|A^\frac12\textbf{v}\|^2+\|A^\frac12\textbf{v}_\delta\|^2+|y_\delta(\theta_t\omega)|^2+|z(\theta_t\omega)|^2)
|y_\delta(\theta_t\omega)-z(\theta_t\omega)|^2.
\end{align}
Applying Gronwall's lemma, we can get that for any $s\in(t-1, t-\frac 12)$, $t\geq 1$,
\begin{align}
 & \|A^\frac12\bar{\textbf{v}}(t)\|^2
 +  \frac{\nu }{2} \int_s^te^{\int_\eta^tC(\|A^\frac12\textbf{v}\|^4+\|A^\frac12\textbf{v}_\delta\|^4
 +|y_\delta(\theta_\tau\omega)|^4+|z(\theta_\tau\omega)|^4)\, \d \tau}\|A\bar{\textbf{v}}(\eta) \|^2 \, \d \eta
\nonumber\\
&\quad \leq e^{ \int_s^tC(\|A^\frac12\textbf{v}\|^4+\|A^\frac12\textbf{v}_\delta\|^4+|y_\delta(\theta_\tau\omega)|^4+|z(\theta_\tau\omega)|^4)\, \d \tau}\|A^\frac12\bar{\textbf{v}}(s)\|^2 \nonumber\\
&\quad +C\int_s^te^{\int_\eta^tC(\|A^\frac12\textbf{v}\|^4+\|A^\frac12\textbf{v}_\delta\|^4
 +|y_\delta(\theta_\tau\omega)|^4+|z(\theta_\tau\omega)|^4)\, \d \tau}(1+\|A^\frac12\textbf{v}\|^2+\|A^\frac12\textbf{v}_\delta\|^2\nonumber\\
 &\quad +|y_\delta(\theta_\eta\omega)|^2+|z(\theta_\eta\omega)|^2)
|y_\delta(\theta_\eta\omega)-z(\theta_\eta\omega)|^2 \, \d \eta\nonumber\\
&\quad \leq e^{ \int_s^tC(\|A^\frac12\textbf{v}\|^4+\|A^\frac12\textbf{v}_\delta\|^4+|y_\delta(\theta_\tau\omega)|^4+|z(\theta_\tau\omega)|^4)\, \d \tau}(\|A^\frac12\bar{\textbf{v}}(s)\|^2 \nonumber\\
&\quad +C\int_s^t(1+\|A^\frac12\textbf{v}\|^2+\|A^\frac12\textbf{v}_\delta\|^2 +|y_\delta(\theta_\eta\omega)|^2+|z(\theta_\eta\omega)|^2)
|y_\delta(\theta_\eta\omega)-z(\theta_\eta\omega)|^2 \, \d \eta),
\end{align}
and then integrating w.r.t. $s$ over $(t-1,t-\frac 12)$  yields
\begin{align*}
 & \|A^\frac12\bar{\textbf{v}}(t)\|^2
 + \frac{\nu}{2} \int_{t-\frac 12} ^te^{\int_\eta^tC(\|A^\frac12\textbf{v}\|^4+\|A^\frac12\textbf{v}_\delta\|^4
 +|y_\delta(\theta_\tau\omega)|^4+|z(\theta_\tau\omega)|^4)\, \d \tau}\|A\bar{\textbf{v}}(\eta) \|^2 \, \d \eta
\nonumber\\
&\quad \leq  2  e^{ \int_{t-1}^tC(\|A^\frac12\textbf{v}\|^4+\|A^\frac12\textbf{v}_\delta\|^4+|y_\delta(\theta_\tau\omega)|^4+|z(\theta_\tau\omega)|^4)\, \d \tau}( \int_{t-1} ^{t-\frac 12}\|A^\frac12\bar{\textbf{v}}(s)\|^2\, \d s \nonumber\\
&\quad +C\int_{t-1}^t(1+\|A^\frac12\textbf{v}\|^2+\|A^\frac12\textbf{v}_\delta\|^2 +|y_\delta(\theta_\eta\omega)|^2+|z(\theta_\eta\omega)|^2)
|y_\delta(\theta_\eta\omega)-z(\theta_\eta\omega)|^2 \, \d \eta) .
\end{align*}
Applying \eqref{8.12},  it yields
\begin{align}\label{8.16}
 & \|A^\frac12\bar{\textbf{v}}(t)\|^2
 + \frac{\nu}{2}\int_{t-\frac 12} ^te^{\int_\eta^tC(\|A^\frac12\textbf{v}\|^4+\|A^\frac12\textbf{v}_\delta\|^4
 +|y_\delta(\theta_\tau\omega)|^4+|z(\theta_\tau\omega)|^4)\, \d \tau}\|A\bar{\textbf{v}}(\eta) \|^2 \, \d \eta \nonumber  \\
  &\quad  \leq  Ce^{ \int_{0}^tC(\|A^\frac12\textbf{v}\|^4+\|A^\frac12\textbf{v}_\delta\|^4
 +|y_\delta(\theta_\tau\omega)|^4+|z(\theta_\tau\omega)|^4+1)\, \d \tau}\nonumber  \\
   &\cdot(\|\bar{\textbf{v}}(0)\|^2+\sup\limits_{s\in[t-1,t]}(|y_\delta(\theta_s\omega)-z(\theta_s\omega)|^2 ))
,\quad t\geq 1.
\end{align}

Since the absorbing set $\mathfrak B$ itself belongs to the attraction universe $\D_H$,  it pullback absorbs itself. Then  there exists a random variable $\tau_\omega^1 \geq 1$ such that
\begin{align}
 \phi \big(\tau_\omega^1,\theta_{-\tau_\omega^1}\omega, \B(\theta_{-\tau_\omega^1}\omega)\big)
 \subset \B(\omega), \quad \omega\in \Omega.
\end{align}
We replace $\omega$ with $\theta_{-\tau_\omega^1} \omega$ in \eqref{8.16}  and deduce  the estimate at $t=\tau_\omega^1$ that
\begin{align}\label{8.17}
  &\big \|
 A^{\frac 12}  \bar {\textbf{v}}(\tau_\omega^1,\theta_{-\tau_\omega^1}\omega,  \bar{\textbf{v}}(0))  \big \| {^2}\nonumber\\
 &\leq  Ce^{ C\int_{0}^{\tau_\omega^1} (\|A^\frac12\textbf{v}_\delta(s, \theta_{-\tau_\omega^1}\omega, \textbf{v}_\delta(0)) \|^4 +\|A^\frac12\textbf{v}(s, \theta_{-\tau_\omega^1}\omega, \textbf{v}(0)) \|^4)\d s+ C\int^0_{-\tau_\omega^1} (|y_\delta(\theta_s\omega)|^4+|z(\theta_s\omega)|^4+1) \d s}\nonumber\\
 &  \cdot(\|\bar{\textbf{v}}(0)\|^2+\sup\limits_{s\in[-1,0]}(|y_\delta(\theta_s\omega)-z(\theta_s\omega)|^2 )).
\end{align}

 In order to prove the bound on the right hand side of \eqref{8.17}, by \eqref{5.8}, we can deduce that
\begin{align*}
    \int_{0}^{\tau_\omega^1}  \|A^\frac12\textbf{v}_\delta(s, \theta_{-\tau_\omega^1}\omega, \textbf{v}_\delta(0)) \|^4 \, \d s  \leq  \tau_\omega^1 |\zeta_7(\omega)|^2
\end{align*}
and similarly, we also have
\begin{align*}
    \int_{0}^{\tau_\omega^1}  \|A^\frac12\textbf{v}(s, \theta_{-\tau_\omega^1}\omega, \textbf{v}(0)) \|^4 \, \d s  \leq  \tau_\omega^1 |\zeta_7(\omega)|^2
\end{align*}
uniformly for any  $\textbf{v}_\delta(0) $ and $\textbf{v}(0) $ in $  \mathfrak B( \theta_{-\tau_\omega^1} \omega)$ and thus  for \eqref{8.17} we can get that
\begin{align*}
  \|
 &A^{\frac 12} \bar {\textbf{v}}(\tau_\omega^1,\theta_{-\tau_\omega^1}\omega, \bar{\textbf{v}}(0))  \|^2\nonumber\\
&  \leq  Ce^{ C \tau_\omega^1  |\zeta_7( \omega)|^2 + C\int^0_{-\tau_\omega^1} (|y_\delta(\theta_s\omega)|^4+|z(\theta_s\omega)|^4+1) \d s}
   (\|\bar{\textbf{v}}(0)\|^2+\sup\limits_{s\in[-1,0]}(|y_\delta(\theta_s\omega)-z(\theta_s\omega)|^2 ))  \\
   &\leq Ce^{ C \tau_\omega^1  |\zeta_7( \omega)|^2 }(\|\bar{\textbf{v}}(0)\|^2+\sup\limits_{s\in[-1,0]}(|y_\delta(\theta_s\omega)-z(\theta_s\omega)|^2 )).
\end{align*}
Hence, the random variable defined by
\begin{align*}
  \bar{L}_2(\mathfrak B, \omega):=  Ce^{ C \tau_\omega^1  |\zeta_7( \omega)|^2 }
,\quad \omega\in \Omega,
\end{align*}
then we have
\begin{align*}
  \|
 A^{\frac 12} \bar {\textbf{v}}(\tau_\omega^1,\theta_{-\tau_\omega^1}\omega, \bar{\textbf{v}}(0))  \|^2
  \leq   \bar{L}_2(\mathfrak B, \omega)(\|\bar{\textbf{v}}(0)\|^2+\sup\limits_{s\in[-1,0]}
  (|y_\delta(\theta_s\omega)-z(\theta_s\omega)|^2 )),
\end{align*}
whenever $\textbf{v}_{\delta,0},$ $\textbf{v}_{0}\in \mathfrak D \big ( \theta_{-\tau_\omega^1( \omega)}\omega \big),$  $\omega\in \Omega$. Then it yields
\begin{align*}
&\left \| \textbf{u}_\delta \! \left ( \tau_\omega^1 (\omega) ,\theta_{-\tau_\omega^1(\omega)}\omega,  \textbf{u}_{0}\right)
 - \textbf{u} \! \left ( \tau_\omega^1 (\omega) ,\theta_{-\tau_\omega^1(\omega)}\omega, \textbf{u}_{0}\right) \right \|^2_{H^1}
 \nonumber\\
& =\left \| \textbf{v}_\delta \! \left ( \tau_\omega^1 (\omega) ,\theta_{-\tau_\omega^1(\omega)}\omega,  \textbf{v}_{\delta,0}\right)
 - \textbf{v} \! \left ( \tau_\omega^1 (\omega) ,\theta_{-\tau_\omega^1(\omega)}\omega, \textbf{v}_{0}\right)+h(y_\delta(\omega)-z(\omega)) \right \|^2_{H^1}\nonumber\\
&\leq 2\left \| \textbf{v}_\delta \! \left ( \tau_\omega^1 (\omega) ,\theta_{-\tau_\omega^1(\omega)}\omega,  \textbf{v}_{\delta,0}\right)
 - \textbf{v} \! \left ( \tau_\omega^1 (\omega) ,\theta_{-\tau_\omega^1(\omega)}\omega, \textbf{v}_{0}\right)\right\|^2_{H^1}+2\left\|h\left (y_\delta(\omega)-z(\omega)\right ) \right \|^2_{H^1},
\end{align*}
where $\textbf{v}_{\delta,0}=\textbf{u}_0-hy_\delta(\omega)$, $\textbf{v}_{0}=\textbf{u}_0-hz(\omega)$  and $\bar{\textbf{v}}(0)=-h(y_\delta(\omega)-z(\omega))$. Moreover, we also have
\begin{align*}
&\left \| \textbf{u}_\delta \! \left ( \tau_\omega^1 (\omega) ,\theta_{-\tau_\omega^1(\omega)}\omega,  \textbf{u}_{0}\right)
 - \textbf{u} \! \left ( \tau_\omega^1 (\omega) ,\theta_{-\tau_\omega^1(\omega)}\omega, \textbf{u}_{0}\right) \right \|^2_{H^1}\nonumber\\
 &\leq C\bar{L}_2({\mathfrak B}, \omega)\left (\|\bar{\textbf{v}}(0)\|^2+\sup\limits_{s\in[-1,0]}|y_\delta(\theta_s\omega)-z(\theta_s\omega)|^2\right )\nonumber\\
 &\leq C\bar{L}_2({\mathfrak B}, \omega)\left (|y_\delta(\omega)-z(\omega)|^2+\sup\limits_{s\in[-1,0]}|y_\delta(\theta_s\omega)-z(\theta_s\omega)|^2\right ).
\end{align*}
By \eqref{2.4} and \eqref{2.5}, this completes the proof of Theorem \ref{lemma8.2}.

\end{proof}

\subsection{$(H^1, H^2)$-convergence on $\mathfrak B$}
\begin{theorem}[$(H^1, H^2)$-convergence on $\mathfrak B$]\label{lemma8.3}  Let Assumption \ref{assum1} hold and $f\in H^1$. There  exists a  random variable $T_\omega^1$  such that   any
 two solutions $\textbf{u}$ and $\textbf{u}_\delta$ of the random anisotropic  NS equations \eqref{8.1} and \eqref{8.2} driven by  white noise  and colored noise corresponding to initial values  $\textbf{v}_{\delta,0},$ $\textbf{v}_{0}$ in $\mathfrak{B}(\theta_{-T_\omega^1}\omega)$, respectively,   satisfy
\begin{align*}
 \lim\limits_{\delta\rightarrow0}\|\textbf{u}_\delta(T_\omega^1,\theta_{-T_\omega^1}\omega,\textbf{u}_{0})
 -\textbf{u}(T_\omega^1,\theta_{-T_\omega^1}\omega,\textbf{u}_{0}) \big \|^2_{H^2}=0, \quad \omega\in \Omega.
\end{align*}
\end{theorem}

\begin{proof} Taking the inner product of the system   \eqref{8.6} with $A^2\bar{\textbf{v}}$ in $H$,  by integration by parts, we can deduce that
\begin{align}\label{8.18}
&\frac{\d}{\d t}\|A\bar{\textbf{v}}\|^2-2\nu \int_{\mathbb{T}^2}\partial_{yy}\Delta \bar{\textbf{v}}_1\Delta\bar{\textbf{v}}_1dx-2\nu \int_{\mathbb{T}^2}\partial_{xx}\Delta \bar{\textbf{v}}_2\Delta \bar{\textbf{v}}_2dx\nonumber\\
&=-2(B\big( \textbf{v}_\delta(t)+hy_\delta(\theta_t\omega)\big)-B\big( \textbf{v}+hz(\theta_t\omega)\big),A^2(\textbf{v}_\delta-\textbf{v}))\nonumber\\
&-  2(\nu A_1h(y_\delta(\theta_t\omega)-z(\theta_t\omega)),A^2(\textbf{v}_\delta-\textbf{v}))
+2(h(y_\delta (\theta_t\omega)-z(\theta_t\omega)),A^2(\textbf{v}_\delta-\textbf{v}))\nonumber\\
&=-2(B\big(\textbf{v}+hz(\theta_t\omega),\textbf{v}_\delta(t)-\textbf{v}+h(y_\delta(\theta_t\omega)-z(\theta_t\omega)))
,A^2(\textbf{v}_\delta-\textbf{v})\big)\nonumber\\
&-2(B\big( \textbf{v}_\delta(t)-\textbf{v}+h(y_\delta(\theta_t\omega)-z(\theta_t\omega)),\textbf{v}_\delta(t)+hy_\delta(\theta_t\omega))
,A^2(\textbf{v}_\delta-\textbf{v})\big)\nonumber\\
&-  2(\nu A_1h(y_\delta(\theta_t\omega)-z(\theta_t\omega)),A^2(\textbf{v}_\delta-\textbf{v}))
+2(h(y_\delta (\theta_t\omega)-z(\theta_t\omega)),A^2(\textbf{v}_\delta-\textbf{v})).
\end{align}
By direct computations and $\nabla\cdot \bar{\textbf{v}}=0$ and \eqref{3.9}, we also have
\begin{align}\label{9.1}
&-2\nu \int_{\mathbb{T}^2}\partial_{yy}\Delta \bar{\textbf{v}}_1\Delta\bar{\textbf{v}}_1dx-2\nu \int_{\mathbb{T}^2}\partial_{xx}\Delta \bar{\textbf{v}}_2\Delta \bar{\textbf{v}}_2dx\nonumber\\
&= 2\nu \int_{\mathbb{T}^2}(\partial_{y}\Delta \bar{\textbf{v}}_1)^2dx
+2\nu \int_{\mathbb{T}^2}(\partial_{x}\Delta \bar{\textbf{v}}_2)^2dx
\nonumber\\
&\geq\nu\|A^{\frac32} \bar{\textbf{v}}\|^2.
\end{align}
By using the H\"{o}lder inequality, the Gagliardo-Nirenberg inequality and the Young inequality, we can get that
\begin{align}\label{9.2}
 &\big| 2(B\big(\textbf{v}+hz(\theta_t\omega),\textbf{v}_\delta(t)-\textbf{v}+h(y_\delta(\theta_t\omega)-z(\theta_t\omega)))
,A^2\bar{\textbf{v}}\big)\big|\nonumber\\
 &\leq 2\|A^\frac12(\textbf{v}+hz(\theta_t\omega))\|_{L^4}\|A^\frac12\bar{\textbf{v}}\|_{L^4}\|A^\frac32\bar{\textbf{v}}\|
 +2\|\textbf{v}+hz(\theta_t\omega)\|_{L^\infty}\|A\bar{\textbf{v}}\|\|A^\frac32\bar{\textbf{v}}\|\nonumber\\
 &+2\|A^\frac12(\textbf{v}+hz(\theta_t\omega))\|_{L^4}\|A^\frac12h(y_\delta(\theta_t\omega)-z(\theta_t\omega))
 \|_{L^4}\|A^\frac32\bar{\textbf{v}}\|\nonumber\\
 &+2\|\textbf{v}+hz(\theta_t\omega)\|_{L^\infty}\|Ah(y_\delta(\theta_t\omega)-z(\theta_t\omega))\|\|A^\frac32\bar{\textbf{v}}\|
 \nonumber\\
 &\leq C(\|A\textbf{v}\|+\|Ah\||z(\theta_t\omega)|)\|A\bar{\textbf{v}}\|\|A^\frac32\bar{\textbf{v}}\|\nonumber\\
 &+C\|A(\textbf{v}+hz(\theta_t\omega))\|\|A\bar{\textbf{v}}\|\|A^\frac32\bar{\textbf{v}}\|\nonumber\\
 &+C(\|A\textbf{v}\|+\|Ah\||z(\theta_t\omega)|)\|Ah\||y_\delta(\theta_t\omega)-z(\theta_t\omega)|
 \|A^\frac32\bar{\textbf{v}}\|\nonumber\\
 &+C\|A(\textbf{v}+hz(\theta_t\omega))\|\|Ah\||y_\delta(\theta_t\omega)-z(\theta_t\omega)|\|A^\frac32\bar{\textbf{v}}\|\nonumber\\
 &\leq \frac{\nu}{8}\|A^\frac32\bar{\textbf{v}}\|^2+C(\|A\textbf{v}\|^2+|z(\theta_t\omega)|^2)\|A\bar{\textbf{v}}\|^2\nonumber\\
 &+C(\|A\textbf{v}\|^2+|z(\theta_t\omega)|^2)|y_\delta(\theta_t\omega)-z(\theta_t\omega)|^2,
\end{align}
and 
\begin{align}\label{9.3}
&\big| -2(B\big( \textbf{v}_\delta(t)-\textbf{v}+h(y_\delta(\theta_t\omega)-z(\theta_t\omega)),\textbf{v}_\delta(t)+hy_\delta(\theta_t\omega))
,A^2(\textbf{v}_\delta-\textbf{v})\big)\big|\nonumber\\
&\leq 2\|A^\frac12\bar{\textbf{v}}\|_{L^4}\|A^\frac12(\textbf{v}_\delta(t)+hy_\delta(\theta_t\omega))\|_{L^4}
\|A^\frac32\bar{\textbf{v}}\|\nonumber\\
&+2\|\bar{\textbf{v}}\|_{L^\infty}\|A(\textbf{v}_\delta(t)+hy_\delta(\theta_t\omega))\|
\|A^\frac32\bar{\textbf{v}}\|\nonumber\\
&+2\|A^\frac12h(y_\delta(\theta_t\omega)-z(\theta_t\omega))\|_{L^4}\|A^\frac12(\textbf{v}_\delta(t)+hy_\delta(\theta_t\omega))\|_{L^4}
\|A^\frac32\bar{\textbf{v}}\|\nonumber\\
&+2\|h(y_\delta(\theta_t\omega)-z(\theta_t\omega))\|_{L^\infty}\|A(\textbf{v}_\delta(t)+hy_\delta(\theta_t\omega))\|
\|A^\frac32\bar{\textbf{v}}\|\nonumber\\
&\leq C(\|A\textbf{v}_\delta(t)\|+\|Ah\||y_\delta(\theta_t\omega)|)\|A\bar{\textbf{v}}\|
\|A^\frac32\bar{\textbf{v}}\|\nonumber\\
&+C\|A\bar{\textbf{v}}\|\|A(\textbf{v}_\delta(t)+hy_\delta(\theta_t\omega))\|
\|A^\frac32\bar{\textbf{v}}\|\nonumber\\
&+C\|Ah\||y_\delta(\theta_t\omega)-z(\theta_t\omega)|\|A(\textbf{v}_\delta(t)+hy_\delta(\theta_t\omega))\|
\|A^\frac32\bar{\textbf{v}}\|\nonumber\\
&+C(\|A\textbf{v}_\delta(t)\|+\|Ah\||y_\delta(\theta_t\omega)|)\|Ah\||y_\delta(\theta_t\omega)-z(\theta_t\omega)|
\|A^\frac32\bar{\textbf{v}}\|\nonumber\\
&\leq \frac{\nu}{8}\|A^\frac32\bar{\textbf{v}}\|^2
+C(\|A\textbf{v}_\delta(t)\|^2+|y_\delta(\theta_t\omega)|^2)\|A\bar{\textbf{v}}\|^2\nonumber\\
&+C(\|A\textbf{v}_\delta(t)\|^2+|y_\delta(\theta_t\omega)|^2)|y_\delta(\theta_t\omega)-z(\theta_t\omega)|^2.
\end{align}
For the last two terms on the right hand side of \eqref{8.18}, by using the similar method, we also can deduce that
\begin{align}\label{9.4}
&-  2(\nu A_1h(y_\delta(\theta_t\omega)-z(\theta_t\omega)),A^2(\textbf{v}_\delta-\textbf{v}))
+2(h(y_\delta (\theta_t\omega)-z(\theta_t\omega)),A^2(\textbf{v}_\delta-\textbf{v}))\nonumber\\
&\leq 2\nu\|A^\frac32h\||y_\delta(\theta_t\omega)-z(\theta_t\omega)|\|A^\frac32\bar{\textbf{v}}\|
+2\|A^\frac32h\||y_\delta(\theta_t\omega)-z(\theta_t\omega)|\|A^\frac32\bar{\textbf{v}}\|
\nonumber\\
&\leq \frac{\nu}{4}\|A^\frac32\bar{\textbf{v}}\|^2+C|y_\delta(\theta_t\omega)-z(\theta_t\omega)|^2.
\end{align}
Inserting \eqref{9.1}, \eqref{9.2}, \eqref{9.3} and \eqref{9.4} into \eqref{8.18} yields
\begin{align}\label{9.15}
&\frac{\d}{\d t}\|A\bar{\textbf{v}}\|^2+\frac12\nu\|A^{\frac32} \bar{\textbf{v}}\|^2\nonumber\\
&\leq C\|A\bar{\textbf{v}}\|^2 \left ( \|A\textbf{v}\|^2+\|A\textbf{v}_\delta(t)\|^2+|y_\delta(\theta_t\omega)|^2+|z(\theta_t\omega)|^2 \right ) \nonumber\\
 &+C(1+\|A\textbf{v}\|^2+\|A\textbf{v}_\delta(t)\|^2+|y_\delta(\theta_t\omega)|^2+|z(\theta_t\omega)|^2
 )|y_\delta(\theta_t\omega)-z(\theta_t\omega)|^2.
\end{align}
For any $s\in(t-1,t)$ and $t\geq 1$, applying Gronwall's lemma, then  we can get that
\begin{align*}
&\|A\bar{\textbf{v}}(t)\|^2\nonumber\\
&\quad-e^{\int_s^tC( \|A\textbf{v}\|^2+\|A\textbf{v}_\delta\|^2+|y_\delta(\theta_\tau\omega)|^2+|z(\theta_\tau\omega)|^2)\, \d \tau}\|A\bar{\textbf{v}}(s)\|^2\nonumber\\
&\quad \leq \int_s^te^{\int_\eta^tC( \|A\textbf{v}\|^2+\|A\textbf{v}_\delta\|^2+|y_\delta(\theta_\tau\omega)|^2+|z(\theta_\tau\omega)|^2)\, \d \tau}
C(1+\|A\textbf{v}\|^2+\|A\textbf{v}_\delta\|^2\nonumber\\
&\quad +|y_\delta(\theta_\eta\omega)|^2+|z(\theta_\eta\omega)|^2
 )|y_\delta(\theta_\eta\omega)-z(\theta_\eta\omega)|^2 \d \eta\nonumber\\
&\quad \leq Ce^{\int_{t-1}^tC( \|A\textbf{v}\|^2+\|A\textbf{v}_\delta\|^2+|y_\delta(\theta_\tau\omega)|^2+|z(\theta_\tau\omega)|^2)\, \d \tau}
\int_{t-1}^t C(1+\|A\textbf{v}\|^2+\|A\textbf{v}_\delta\|^2\nonumber\\
&\quad +|y_\delta(\theta_\eta\omega)|^2+|z(\theta_\eta\omega)|^2
 )|y_\delta(\theta_\eta\omega)-z(\theta_\eta\omega)|^2 \d \eta.
\end{align*}
Integrating w.r.t. $s$ on $(t-1,t)$ yields
\begin{align}\label{9.6}
&\|A\bar{\textbf{v}}(t)\|^2\nonumber\\
&\quad \leq Ce^{\int_{t-1}^tC( \|A\textbf{v}\|^2+\|A\textbf{v}_\delta\|^2+|y_\delta(\theta_\tau\omega)|^2+|z(\theta_\tau\omega)|^2)\, \d \tau}
[\int_{t-1}^t\|A\bar{\textbf{v}}(s)\|^2\, \d s+\int_{t-1}^t C(1+\|A\textbf{v}\|^2\nonumber\\
&\quad +\|A\textbf{v}_\delta\|^2+|y_\delta(\theta_\eta\omega)|^2+|z(\theta_\eta\omega)|^2
 )|y_\delta(\theta_\eta\omega)-z(\theta_\eta\omega)|^2 \d \eta].
\end{align}

Since our initial values $\textbf{v}_{\delta,0}$ and $\textbf{v}_{0}$   belong to the $H^1$ random  absorbing set  $\mathfrak B$,  applying Gronwall's lemma to \eqref{9.5}, then we can deduce    that
\begin{align}
 & \|A^\frac12\bar{\textbf{v}}(t)\|^2 + \frac{\nu}{2} \int_0^te^{\int_s^tC( \|A^\frac12\textbf{v}\|^4+\|A^\frac12\textbf{v}_\delta\|^4+|y_\delta(\theta_\tau\omega)|^4+|z(\theta_\tau\omega)|^4)\, \d \tau}\|A\bar{\textbf{v}}(s) \|^2\, \d s
\nonumber\\
&\quad \leq e^{\int_0^tC( \|A^\frac12\textbf{v}\|^4+\|A^\frac12\textbf{v}_\delta\|^4+|y_\delta(\theta_\tau\omega)|^4+|z(\theta_\tau\omega)|^4)\, \d \tau} (\|A^\frac12\bar{v}(0)\|^2\nonumber\\
&\quad +C\int_0^t(1+\|A^\frac12\textbf{v}\|^2+\|A^\frac12\textbf{v}_\delta\|^2 +|y_\delta(\theta_\eta\omega)|^2+|z(\theta_\eta\omega)|^2)
|y_\delta(\theta_\eta\omega)-z(\theta_\eta\omega)|^2 \, \d \eta)
.
\end{align}
Hence, we can get the  following inequalities
\begin{align*}
& \frac{\nu}{2} \int_{t-1}^t\|A\bar{\textbf{v}}(s)\|^2\, \d s\nonumber\\
& \quad \leq\frac{\nu}{2} \int_{t-1}^te^{\int_s^tC( \|A^\frac12\textbf{v}\|^4+\|A^\frac12\textbf{v}_\delta\|^4+|y_\delta(\theta_\tau\omega)|^4+|z(\theta_\tau\omega)|^4)\, \d \tau}\|A\bar{\textbf{v}}(s)\|^2\, \d s\nonumber\\
&\quad \leq e^{\int_0^tC( \|A^\frac12\textbf{v}\|^4+\|A^\frac12\textbf{v}_\delta\|^4+|y_\delta(\theta_\tau\omega)|^4+|z(\theta_\tau\omega)|^4)\, \d \tau} (\|A^\frac12\bar{v}(0)\|^2\nonumber\\
&\quad +C\int_0^t(1+\|A^\frac12\textbf{v}\|^2+\|A^\frac12\textbf{v}_\delta\|^2 +|y_\delta(\theta_\eta\omega)|^2+|z(\theta_\eta\omega)|^2)
|y_\delta(\theta_\eta\omega)-z(\theta_\eta\omega)|^2 \, \d \eta)
.
\end{align*}

Inserting these inequalities  into \eqref{9.6}, we also can deduce that
\begin{align}
 \|A\bar{\textbf{v}}(t)\|^2
&\leq Ce^{\int_0^tC( 1+\|A^\frac12\textbf{v}\|^4+\|A^\frac12\textbf{v}_\delta\|^4+|y_\delta(\theta_\tau\omega)|^4+|z(\theta_\tau\omega)|^4)\, \d \tau}e^{\int_{t-1}^tC( \|A\textbf{v}\|^2+\|A\textbf{v}_\delta\|^2)\, \d \tau}\nonumber\\
&\left(\|A^\frac12\bar{v}(0)\|^2 + \sup\limits_{s\in[0,t]}|y_\delta(\theta_s\omega)-z(\theta_s\omega)|^2\right).
\end{align}
for any $t\geq 1$.
We replace $\omega$ by $\theta_{-t}\omega$ to  get  for any $ t\geq 1$
\begin{align}\label{9.7}
& \|A\bar{\textbf{v}}(t,\theta_{-t}\omega,\bar{v}(0))\|^2\nonumber\\
&\quad \leq Ce^{\int_0^tC( 1+\|A^\frac12\textbf{v}(\tau,\theta_{-t}\omega,\textbf{v}(0))\|^4
+\|A^\frac12\textbf{v}_\delta(\tau,\theta_{-t}\omega,\textbf{v}_\delta(0))\|^4
+|y_\delta(\theta_{\tau-t}\omega)|^4+|z(\theta_{\tau-t}\omega)|^4)\, \d \tau}\nonumber\\
&\quad e^{\int_{t-1}^tC( \|A\textbf{v}(\tau,\theta_{-t}\omega,\textbf{v}(0))\|^2+\|A\textbf{v}_\delta(\tau,\theta_{-t}\omega,\textbf{v}_\delta(0))
\|^2)\, \d \tau} \nonumber\\
&\left(\|A^\frac12\bar{\textbf{v}}(0)\|^2+ \sup\limits_{s\in[0,t]}|y_\delta(\theta_{s-t}\omega)-z(\theta_{s-t}\omega)|^2\right).
\end{align}
By \eqref{mar11.2} and \eqref{mar9.4}, for $\omega\in \Omega$, there exists $T_\omega^1 \geq T_\omega$
such that
\begin{align}\label{9.8}
  \int_{T_\omega^1 -1}^{T_\omega^1 }  \left ( \|A^\frac12\textbf{v}_\delta(\eta,\theta_{-T_\omega^1 }\omega,\textbf{v}_{\delta,0})\|^2+\|A^\frac12\textbf{v}(\eta,\theta_{-T_\omega^1 }\omega,\textbf{v}_{0})\|^2\right)  \d \eta
 \leq  2\zeta_8(\omega) .
 \end{align}
By \eqref{mar11.3} and the similar method, for   $\textbf{v}(0), \textbf{v}_\delta(0)\in \mathfrak B(\theta_{-t}\omega)$ we can deduce that
 \begin{align*}
 &\int_0^t
 (\|A^{\frac 12} \textbf{v}_\delta(\eta, \theta_{-t} \omega, \textbf{v}_\delta(0))\|^4+\|A^{\frac 12} \textbf{v}(\eta, \theta_{-t} \omega, \textbf{v}(0))\|^4) \, \d \eta\nonumber\\
 &   \leq  2\left( \| \mathfrak B  (\theta_{-t}\omega)\|^2_{H^1} +1 \right)^2 \int_0^t e^{ \int^\eta_0 C(1+|y_\delta(\theta_{\tau-t} \omega)|^6+|z(\theta_{\tau-t} \omega)|^6) \, \d \tau}  \, \d \eta \nonumber \\
 &   \leq  2\big( \zeta_2(\theta_{-t}\omega)  +C \big)^2  e^{ Ct+C \int^0_{-t} ( |y_\delta(\theta_{\tau} \omega)|^6 ++|z(\theta_{\tau} \omega)|^6)\, \d \tau}   ,
   \quad t>0  , \nonumber
\end{align*}
and, particularly for $t=T_\omega^1$,
 \begin{align}\label{9.9}
 & \int_0^{T_\omega^1}
 (\|A^{\frac 12} \textbf{v}_\delta(\eta, \theta_{-{T_\omega^1}} \omega, \textbf{v}_\delta(0))\|^4+\|A^{\frac 12} \textbf{v}(\eta, \theta_{-{T_\omega^1}} \omega, \textbf{v}(0))\|^4 )\, \d \eta \nonumber \\
 &   \quad  \leq  2\big( \zeta_2(\theta_{-{T_\omega^1}}\omega)  +C \big)^2  e^{ C {T_\omega^1}+C \int^0_{-{T_\omega^1}}  (|y_\delta(\theta_{\tau} \omega)|^6+|z(\theta_{\tau} \omega)|^6 )\, \d \tau} =\zeta_{15}(\omega).
\end{align}
Finally, inserting     \eqref{9.8}  and \eqref{9.9}  into \eqref{9.7}  we can get at $t=T_\omega^1$ that
\begin{align*}
 \|A\bar{\textbf{v}}(T_\omega^1 ,\theta_{-T_\omega^1 }\omega,\bar{\textbf{v}}(0))\|^2
&\leq Ce^{ C \zeta_{15}(\omega)+C\zeta_8(\omega) } \left(\|A^\frac12\bar{\textbf{v}}(0)\|^2+ \sup\limits_{s\in[-T_\omega^1,0]}|y_\delta(\theta_{s}\omega)-z(\theta_{s}\omega)|^2\right). \nonumber
\end{align*}
Hence, let
\begin{align*}
 \bar{L}_3 (\mathfrak B, \omega) :=   Ce^{ C \zeta_{15}(\omega)+C\zeta_8(\omega) }  ,\quad \omega\in \Omega. \qedhere
\end{align*}
Then it yields
\begin{align*}
&\left \| \textbf{u}_\delta \! \left ( T_\omega^1 (\omega) ,\theta_{-T_\omega^1(\omega)}\omega,  \textbf{u}_{0}\right)
 - \textbf{u} \! \left ( T_\omega^1 (\omega) ,\theta_{-T_\omega^1(\omega)}\omega, \textbf{u}_{0}\right) \right \|^2_{H^2}
 \nonumber\\
& =\left \| \textbf{v}_\delta \! \left ( T_\omega^1 (\omega) ,\theta_{-T_\omega^1}\omega,  \textbf{v}_{\delta,0}\right)
 - \textbf{v} \! \left ( T_\omega^1 (\omega) ,\theta_{-T_\omega^1(\omega)}\omega, \textbf{v}_{0}\right)+h(y_\delta(\omega)-z(\omega)) \right \|^2_{H^2}\nonumber\\
&\leq 2\left \| \textbf{v}_\delta \! \left ( T_\omega^1 (\omega) ,\theta_{-T_\omega^1(\omega)}\omega,  \textbf{v}_{\delta,0}\right)
 - \textbf{v} \! \left ( T_\omega^1 (\omega) ,\theta_{-T_\omega^1(\omega)}\omega, \textbf{v}_{0}\right)\right\|^2_{H^2}+2\left\|h\left (y_\delta(\omega)-z(\omega)\right ) \right \|^2_{H^2},
\end{align*}
where $\textbf{v}_{\delta,0}=\textbf{u}_0-hy_\delta(\omega)$, $\textbf{v}_{0}=\textbf{u}_0-hz(\omega)$  and $\bar{\textbf{v}}(0)=-h(y_\delta(\omega)-z(\omega))$. Moreover, we also have
\begin{align*}
&\left \| \textbf{u}_\delta \! \left ( T_\omega^1 (\omega) ,\theta_{-T_\omega^1(\omega)}\omega,  \textbf{u}_{0}\right)
 - \textbf{u} \! \left ( T_\omega^1(\omega) ,\theta_{-T_\omega^1(\omega)}\omega, \textbf{u}_{0}\right) \right \|^2_{H^2}\nonumber\\
 &\leq C\bar{L}_3({\mathfrak B}, \omega)\left (\|A^\frac12\bar{\textbf{v}}(0)\|^2+\sup\limits_{s\in[-1,0]}|y_\delta(\theta_s\omega)-z(\theta_s\omega)|^2\right )\nonumber\\
 &\leq C\bar{L}_3({\mathfrak B}, \omega)\left (|y_\delta(\omega)-z(\omega)|^2+\sup\limits_{s\in[-1,0]}|y_\delta(\theta_s\omega)-z(\theta_s\omega)|^2\right ).
\end{align*}
By \eqref{2.4} and \eqref{2.5}, this completes the proof of Theorem \ref{lemma8.3}.
\end{proof}
\subsection{$(H^2, H^3)$-convergence on $\mathfrak B^2$}
\begin{theorem}[$(H^2, H^3)$-convergence on $\mathfrak B^2$]\label{lemma8.4}  Let Assumption \ref{assum1} hold and $f\in H^2$. There  exists a  random variable $\bar{T}_\omega^*$  such that    any
 two solutions $\textbf{u}$ and $\textbf{u}_\delta$ of the random anisotropic  NS equations \eqref{8.1} and \eqref{8.2} driven by  white noise  and colored noise corresponding to initial values  $\textbf{u}_{0}$  in $\mathfrak{B}^2(\theta_{-\bar{T}_\omega^*})$, respectively,   satisfy
\begin{align*}
 \lim\limits_{\delta\rightarrow0}\|\textbf{u}_\delta(\bar{T}_\omega^*,\theta_{-\bar{T}_\omega^*}\omega,\textbf{u}_{0})
 -\textbf{u}(\bar{T}_\omega^*,\theta_{-\bar{T}_\omega^*}\omega,\textbf{u}_{0}) \big \|^2_{H^3}=0, \quad \omega\in \Omega.
\end{align*}
\end{theorem}
\begin{proof} Taking the inner product of the system   \eqref{8.6} with $A^3\bar{\textbf{v}}$ in $H$,  by integration by parts, we can deduce that
\begin{align}\label{9.10}
&\frac{\d}{\d t}\|A^{\frac32}\bar{\textbf{v}}\|^2+2\nu \int_{\mathbb{T}^2}\partial_{yy}\Delta \bar{\textbf{v}}_1\Delta^2\bar{\textbf{v}}_1dx+2\nu \int_{\mathbb{T}^2}\partial_{xx}\Delta \bar{\textbf{v}}_2\Delta^2 \bar{\textbf{v}}_2dx\nonumber\\
&=-2(B\big( \textbf{v}_\delta(t)+hy_\delta(\theta_t\omega)\big)-B\big( \textbf{v}+hz(\theta_t\omega)\big),A^3(\textbf{v}_\delta-\textbf{v}))\nonumber\\
&-  2(\nu A_1h(y_\delta(\theta_t\omega)-z(\theta_t\omega)),A^3(\textbf{v}_\delta-\textbf{v}))
+2(h(y_\delta (\theta_t\omega)-z(\theta_t\omega)),A^3(\textbf{v}_\delta-\textbf{v}))\nonumber\\
&=-2(B\big(\textbf{v}+hz(\theta_t\omega),\textbf{v}_\delta(t)-\textbf{v}+h(y_\delta(\theta_t\omega)-z(\theta_t\omega)))
,A^3(\textbf{v}_\delta-\textbf{v})\big)\nonumber\\
&-2(B\big( \textbf{v}_\delta(t)-\textbf{v}+h(y_\delta(\theta_t\omega)-z(\theta_t\omega)),\textbf{v}_\delta(t)+hy_\delta(\theta_t\omega))
,A^3(\textbf{v}_\delta-\textbf{v})\big)\nonumber\\
&-  2(\nu A_1h(y_\delta(\theta_t\omega)-z(\theta_t\omega)),A^3(\textbf{v}_\delta-\textbf{v}))
+2(h(y_\delta (\theta_t\omega)-z(\theta_t\omega)),A^3(\textbf{v}_\delta-\textbf{v})).
\end{align}
By direct computations and $\nabla\cdot \bar{\textbf{v}}=0$ and \eqref{3.9}, we also have
\begin{align}\label{9.11}
&\nu \int_{\mathbb{T}^2}\partial_{yy}\Delta \bar{\textbf{v}}_1\Delta^2 \bar{\textbf{v}}_1dx-\nu \int_{\mathbb{T}^2}\partial_{xx}\Delta \bar{\textbf{v}}_2\Delta^2 \bar{\textbf{v}}_2dx\nonumber\\
&= \nu \int_{\mathbb{T}^2}(\partial_{yy}\Delta \bar{\textbf{v}}_1)^2dx+\nu \int_{\mathbb{T}^2}\partial_{xy}\Delta \bar{\textbf{v}}_1\partial_{xy}\Delta \bar{\textbf{v}}_1dx\nonumber\\
&+\nu \int_{\mathbb{T}^2}(\partial_{xx}\Delta \bar{\textbf{v}}_2)^2dx+\nu \int_{\mathbb{T}^2}\partial_{xy}\Delta \bar{\textbf{v}}_2\partial_{xy}\Delta \bar{\textbf{v}}_2dx
\nonumber\\
&=\nu \int_{\mathbb{T}^2}(\partial_{yy}\Delta \bar{\textbf{v}}_1)^2dx+\nu \int_{\mathbb{T}^2}(\partial_{yy}\Delta \bar{\textbf{v}}_2)^2dx\nonumber\\
&+\nu \int_{\mathbb{T}^2}(\partial_{xx}\Delta \bar{\textbf{v}}_2)^2dx+\nu \int_{\mathbb{T}^2}(\partial_{xx}\Delta \bar{\textbf{v}}_1)^2dx
\nonumber\\
&\geq\frac\nu2\|A^2 \bar{\textbf{v}}\|^2.
\end{align}
By using the H\"{o}lder inequality, the Gagliardo-Nirenberg inequality and the Young inequality, we can get that
\begin{align}\label{9.12}
 &\big| 2(B\big(\textbf{v}+hz(\theta_t\omega),\textbf{v}_\delta(t)-\textbf{v}+h(y_\delta(\theta_t\omega)-z(\theta_t\omega)))
,A^3\bar{\textbf{v}}\big)\big|\nonumber\\
 &\leq C\|A(\textbf{v}+hz(\theta_t\omega))\|\|A^\frac12\bar{\textbf{v}}\|_{L^\infty}\|A^2\bar{\textbf{v}}\|
 +C\|\textbf{v}+hz(\theta_t\omega)\|_{L^\infty}\|A^\frac32\bar{\textbf{v}}\|\|A^2\bar{\textbf{v}}\|\nonumber\\
 &+C\|A(\textbf{v}+hz(\theta_t\omega))\|\|A^\frac12h(y_\delta(\theta_t\omega)-z(\theta_t\omega))
 \|_{L^\infty}\|A^2\bar{\textbf{v}}\|\nonumber\\
 &+C\|\textbf{v}+hz(\theta_t\omega)\|_{L^\infty}\|A^\frac32h(y_\delta(\theta_t\omega)-z(\theta_t\omega))\|\|A^2\bar{\textbf{v}}\|
 \nonumber\\
 &\leq C(\|A\textbf{v}\|+\|Ah\||z(\theta_t\omega)|)\|A^\frac32\bar{\textbf{v}}\|\|A^2\bar{\textbf{v}}\|\nonumber\\
 &+C\|A(\textbf{v}+hz(\theta_t\omega))\|\|A^\frac32\bar{\textbf{v}}\|\|A^2\bar{\textbf{v}}\|\nonumber\\
 &+C(\|A\textbf{v}\|+\|Ah\||z(\theta_t\omega)|)\|A^\frac32h\||y_\delta(\theta_t\omega)-z(\theta_t\omega)|
 \|A^2\bar{\textbf{v}}\|\nonumber\\
 &+C\|A(\textbf{v}+hz(\theta_t\omega))\|\|A^\frac32h\||y_\delta(\theta_t\omega)-z(\theta_t\omega)|
 \|A^2\bar{\textbf{v}}\|\nonumber\\
 &\leq \frac{\nu}{8}\|A^2\bar{\textbf{v}}\|^2+C(\|A\textbf{v}\|^2+|z(\theta_t\omega)|^2)\|A^\frac32\bar{\textbf{v}}\|^2\nonumber\\
 &+C(\|A\textbf{v}\|^2+|z(\theta_t\omega)|^2)|y_\delta(\theta_t\omega)-z(\theta_t\omega)|^2,
\end{align}
and 
\begin{align}\label{9.13}
&\big| -2(B\big( \textbf{v}_\delta(t)-\textbf{v}+h(y_\delta(\theta_t\omega)-z(\theta_t\omega)),\textbf{v}_\delta(t)+hy_\delta(\theta_t\omega))
,A^3(\textbf{v}_\delta-\textbf{v})\big)\big|\nonumber\\
&\leq C\|A\bar{\textbf{v}}\|_{L^4}\|A^\frac12(\textbf{v}_\delta(t)+hy_\delta(\theta_t\omega))\|_{L^4}
\|A^2\bar{\textbf{v}}\|\nonumber\\
&+C\|\bar{\textbf{v}}\|_{L^\infty}\|A^\frac32(\textbf{v}_\delta(t)+hy_\delta(\theta_t\omega))\|
\|A^2\bar{\textbf{v}}\|\nonumber\\
&+C\|Ah(y_\delta(\theta_t\omega)-z(\theta_t\omega))\|_{L^4}\|A^\frac12(\textbf{v}_\delta(t)+hy_\delta(\theta_t\omega))\|_{L^4}
\|A^2\bar{\textbf{v}}\|\nonumber\\
&+C\|h(y_\delta(\theta_t\omega)-z(\theta_t\omega))\|_{L^\infty}\|A^\frac32(\textbf{v}_\delta(t)+hy_\delta(\theta_t\omega))\|
\|A^2\bar{\textbf{v}}\|\nonumber\\
&\leq C(\|A\textbf{v}_\delta(t)\|+\|Ah\||y_\delta(\theta_t\omega)|)\|A^\frac32\bar{\textbf{v}}\|
\|A^2\bar{\textbf{v}}\|\nonumber\\
&+C\|A^\frac32\bar{\textbf{v}}\|\|A(\textbf{v}_\delta(t)+hy_\delta(\theta_t\omega))\|
\|A^2\bar{\textbf{v}}\|\nonumber\\
&+C\|A^\frac32h\||y_\delta(\theta_t\omega)-z(\theta_t\omega)|\|A(\textbf{v}_\delta(t)+hy_\delta(\theta_t\omega))\|
\|A^2\bar{\textbf{v}}\|\nonumber\\
&+C(\|A^\frac32\textbf{v}_\delta(t)\|+\|A^\frac32h\||y_\delta(\theta_t\omega)|)\|Ah\||y_\delta(\theta_t\omega)-z(\theta_t\omega)|
\|A^2\bar{\textbf{v}}\|\nonumber\\
&\leq \frac{\nu}{8}\|A^2\bar{\textbf{v}}\|^2
+C(\|A\textbf{v}_\delta(t)\|^2+|y_\delta(\theta_t\omega)|^2)\|A^\frac32\bar{\textbf{v}}\|^2\nonumber\\
&+C(\|A^\frac32\textbf{v}_\delta(t)\|^2+|y_\delta(\theta_t\omega)|^2)|y_\delta(\theta_t\omega)-z(\theta_t\omega)|^2.
\end{align}
For the last two terms on the right hand side of \eqref{9.10}, by using the similar method, we also can deduce that for $h\in H^4$
\begin{align}\label{9.14}
&-  2(\nu A_1h(y_\delta(\theta_t\omega)-z(\theta_t\omega)),A^3(\textbf{v}_\delta-\textbf{v}))
+2(h(y_\delta (\theta_t\omega)-z(\theta_t\omega)),A^3(\textbf{v}_\delta-\textbf{v}))\nonumber\\
&\leq 2\nu\|A^2h\||y_\delta(\theta_t\omega)-z(\theta_t\omega)|\|A^2\bar{\textbf{v}}\|
+2\|Ah\||y_\delta(\theta_t\omega)-z(\theta_t\omega)|\|A^2\bar{\textbf{v}}\|
\nonumber\\
&\leq \frac{\nu}{4}\|A^2\bar{\textbf{v}}\|^2+C|y_\delta(\theta_t\omega)-z(\theta_t\omega)|^2.
\end{align}
Inserting \eqref{9.11}, \eqref{9.12}, \eqref{9.13} and \eqref{9.14} into \eqref{9.10} yields
\begin{align*}
\frac{\d}{\d t}\|A^\frac32\bar{\textbf{v}}\|^2
&\leq C\|A^\frac32\bar{\textbf{v}}\|^2 \left ( \|A\textbf{v}\|^2+\|A\textbf{v}_\delta\|^2+|y_\delta(\theta_t\omega)|^2+|z(\theta_t\omega)|^2 \right )
\nonumber\\
 &+C(1+\|A\textbf{v}\|^2+\|A^\frac32\textbf{v}_\delta(t)\|^2+|y_\delta(\theta_t\omega)|^2+
 |z(\theta_t\omega)|^2)|y_\delta(\theta_t\omega)-z(\theta_t\omega)|^2.
\end{align*}
For any $s\in(t-1,t)$ and $t\geq 1$, applying Gronwall's lemma, then  we also can get that
\begin{align*}
&\|A^\frac32\bar{\textbf{v}}(t)\|^2
-e^{\int_s^tC(\|A\textbf{v}\|^2+\|A\textbf{v}_\delta\|^2+|y_\delta(\theta_\tau\omega)|^2+|z(\theta_\tau\omega)|^2 )\, \d \tau}\|A^\frac32\bar{\textbf{v}}(s)\|^2\nonumber\\
&\quad \leq \int_s^tCe^{\int_\eta^tC( \|A\textbf{v}\|^2+\|A\textbf{v}_\delta\|^2+|y_\delta(\theta_\tau\omega)|^2+|z(\theta_\tau\omega)|^2)\, \d \tau}
 |y_\delta(\theta_\eta\omega)-z(\theta_\eta\omega)|^2\nonumber\\
  &\quad \left(1+\|A\textbf{v}\|^2
 +\|A^\frac32\textbf{v}_\delta\|^2+|y_\delta(\theta_\eta\omega)|^2+
 |z(\theta_\eta\omega)|^2\right) \d \eta\nonumber\\
&\quad \leq Ce^{\int_{t-1}^tC( \|A\textbf{v}\|^2+\|A\textbf{v}_\delta\|^2+|y_\delta(\theta_\tau\omega)|^2+|z(\theta_\tau\omega)|^2)\, \d \tau}
\int_{t-1}^t|y_\delta(\theta_\eta\omega)-z(\theta_\eta\omega)|^2\nonumber\\
  &\quad \left(1+\|A\textbf{v}\|^2
 +\|A^\frac32\textbf{v}_\delta\|^2+|y_\delta(\theta_\eta\omega)|^2+
 |z(\theta_\eta\omega)|^2\right) \d \eta.
\end{align*}
Integrating w.r.t. $s$ on $(t-1,t)$ yields
\begin{align}\label{9.16}
&\|A^\frac32\bar{\textbf{v}}(t)\|^2
-\int_{t-1}^te^{\int_s^tC( \|A\textbf{v}\|^2+\|A\textbf{v}_\delta\|^2+|y_\delta(\theta_\tau\omega)|^2+|z(\theta_\tau\omega)|^2)\, \d \tau}\|A^\frac32\bar{\textbf{v}}(s)\|^2\, \d s\nonumber\\
&\quad \leq Ce^{\int_{t-1}^tC( \|A\textbf{v}\|^2+\|A\textbf{v}_\delta\|^2+|y_\delta(\theta_\tau\omega)|^2+|z(\theta_\tau\omega)|^2)\, \d \tau}
\int_{t-1}^t|y_\delta(\theta_\eta\omega)-z(\theta_\eta\omega)|^2\nonumber\\
  &\quad \left(1+\|A\textbf{v}\|^2
 +\|A^\frac32\textbf{v}_\delta\|^2+|y_\delta(\theta_\eta\omega)|^2+
 |z(\theta_\eta\omega)|^2\right) \d \eta.
\end{align}

Since our initial values $\textbf{v}_{\delta,0}$ and $\textbf{v}_{0}$   belong to the $H^2$ random  absorbing set  $\mathfrak B^2$,  applying Gronwall's lemma to \eqref{9.15}, then we can deduce    that
\begin{align*}
&\|A\bar{\textbf{v}}(t)\|^2+\frac{\nu}{2} \int_0^te^{\int_s^tC( \|A\textbf{v}\|^2+\|A\textbf{v}_\delta\|^2+|y_\delta(\theta_\tau\omega)|^2+|z(\theta_\tau\omega)|^2)\, \d \tau}\|A^\frac32\bar{\textbf{v}}(s)\|^2\, \d s\nonumber\\
&\quad \leq Ce^{\int_{0}^tC( \|A\textbf{v}\|^2+\|A\textbf{v}_\delta\|^2+|y_\delta(\theta_\tau\omega)|^2+|z(\theta_\tau\omega)|^2)\, \d \tau}
[\|A\bar{\textbf{v}}(0)\|^2+\int_{0}^t C(1+\|A\textbf{v}\|^2+\|A\textbf{v}_\delta\|^2\nonumber\\
&\quad +|y_\delta(\theta_\eta\omega)|^2+|z(\theta_\eta\omega)|^2
 )|y_\delta(\theta_\eta\omega)-z(\theta_\eta\omega)|^2 \d \eta].
\end{align*}
Hence, we can get the  following inequalities
\begin{align*}
& \frac{\nu}{2} \int_{t-1}^te^{\int_s^tC( \|A\textbf{v}\|^2+\|A\textbf{v}_\delta\|^2+|y_\delta(\theta_\tau\omega)|^2+|z(\theta_\tau\omega)|^2)\, \d \tau}\|A^\frac32\bar{\textbf{v}}(s)\|^2\, \d s\nonumber\\
&\quad  \leq\frac{\nu}{2} \int_0^te^{\int_s^tC( \|A\textbf{v}\|^2+\|A\textbf{v}_\delta\|^2+|y_\delta(\theta_\tau\omega)|^2+|z(\theta_\tau\omega)|^2)\, \d \tau}\|A^\frac32\bar{\textbf{v}}(s)\|^2\, \d s\nonumber\\
&\quad \leq Ce^{\int_{0}^tC( \|A\textbf{v}\|^2+\|A\textbf{v}_\delta\|^2+|y_\delta(\theta_\tau\omega)|^2+|z(\theta_\tau\omega)|^2)\, \d \tau}
[\|A\bar{\textbf{v}}(0)\|^2+\int_{0}^t C(1+\|A\textbf{v}\|^2+\|A\textbf{v}_\delta\|^2\nonumber\\
&\quad +|y_\delta(\theta_\eta\omega)|^2+|z(\theta_\eta\omega)|^2
 )|y_\delta(\theta_\eta\omega)-z(\theta_\eta\omega)|^2 \d \eta].
\end{align*}
Inserting these inequalities  into \eqref{9.16}, we also can deduce that
\begin{align*}
 &\|A^\frac32\bar{\textbf{v}}(t)\|^2\nonumber\\
&\quad \leq Ce^{\int_{0}^tC( 1+ \|A\textbf{v}\|^2+\|A\textbf{v}_\delta\|^2+|y_\delta(\theta_\tau\omega)|^2+|z(\theta_\tau\omega)|^2)\, \d \tau}
[\|A\bar{\textbf{v}}(0)\|^2\nonumber\\
&\quad+\sup\limits_{s\in[t-1,t]}|y_\delta(\theta_s\omega)-z(\theta_s\omega)|^2+\int_{t-1}^t C\|A^\frac32\textbf{v}_\delta\|^2|y_\delta(\theta_\eta\omega)-z(\theta_\eta\omega)|^2 \d \eta],
\end{align*}
for any $t\geq 1$.
We replace $\omega$ by $\theta_{-t}\omega$ to  get  for any $ t\geq 1$
\begin{align} \label{9.20}
 &\|A^\frac32\bar{\textbf{v}}(t,\theta_{-t}\omega,\bar{\textbf{v}}(0))\|^2\nonumber\\
&\leq Ce^{\int_0^tC( 1+\|A\textbf{\textbf{v}}_\delta(\tau,\theta_{-t}\omega,\textbf{\textbf{v}}_\delta(0))\|^2
+\|A\textbf{\textbf{v}}(\tau,\theta_{-t}\omega,\textbf{\textbf{v}}(0))\|^2
+|y_\delta(\theta_{\tau-t}\omega)|^2+|z(\theta_{\tau-t}\omega)|^2)\, \d \tau}[\|A\bar{\textbf{v}}(0)\|^2\nonumber\\
&+\sup\limits_{s\in[-1,0]}|y_\delta(\theta_s\omega)-z(\theta_s\omega)|^2\left(1+C\int_{t-1}^t \|A^\frac32\textbf{v}_\delta(\eta,\theta_{-t}\omega,\textbf{v}_\delta(0))\|^2
\d \eta\right) ].
\end{align}

By \eqref{6.18},   let
\begin{align*}
\bar{T}_\omega^* :=T_\omega^* (\omega)+1, \quad \omega\in \Omega,
\end{align*}
we can get that
\begin{align}\label{9.17}
  \int_{\bar{T}_\omega^* -1}^{\bar{T}_\omega^* }  \|A^\frac32\textbf{v}_\delta(\eta,\theta_{\bar{T}_\omega^* }\omega,\textbf{v}_\delta(0))\|^2  \d \eta
 \leq  C\zeta_{12}(\omega)  =: \zeta_{16}(\omega) .
 \end{align}

By \eqref{7.4}, for   $\textbf{v}(0), \textbf{v}_\delta(0) \in \mathfrak B^2(\theta_{-t}\omega)$ we also can deduce that
 \begin{align*}
 &\int_0^t
 (\|A \textbf{v}(\eta, \theta_{-t} \omega, \textbf{v}(0))\|^2+\|A \textbf{v}_\delta(\eta, \theta_{-t} \omega, \textbf{v}_\delta(0))\|^2) \, \d \eta\nonumber\\
 &   \leq  C\left( \| \mathfrak B  (\theta_{-t}\omega)\|_{H^2}^2 +1 \right) \int_0^t e^{ \int^\eta_0 C(1+|y_\delta(\theta_{\tau-t} \omega)|^6+|z(\theta_{\tau-t} \omega)|^6) \, \d \tau}  \, \d \eta \nonumber \\
 &   \leq  C\big( \zeta_{12}(\theta_{-t}\omega)  +C \big)  e^{ Ct+C \int^0_{-t}  |y_\delta(\theta_{\tau} \omega)|^6+C \int^0_{-t}  |z(\theta_{\tau} \omega)|^6 \, \d \tau}   ,
   \quad t>0  , \nonumber
\end{align*}
and, particularly for $t=\bar{T}_\omega^*$,
 \begin{align}\label{9.18}
 & \int_0^{\bar{T}_\omega^*}
 (\|A \textbf{v}(\eta, \theta_{-{\bar{T}_\omega^*}} \omega, \textbf{v}(0))\|^2+ \|A \textbf{v}_\delta(\eta, \theta_{-{\bar{T}_\omega^*}} \omega, \textbf{v}_\delta(0))\|^2)\, \d \eta \nonumber \\
 &   \quad  \leq  \big( \zeta_{12}(\theta_{-{\bar{T}_\omega^*}}\omega)  +C \big)^3  e^{ C {\bar{T}_\omega^*}+C \int^0_{-{\bar{T}_\omega^*}} ( |y_\delta(\theta_{\tau} \omega)|^6+ |z(\theta_{\tau} \omega)|^6)\, \d \tau} =:\zeta_{17}(\omega).
\end{align}
Finally, inserting     \eqref{9.17}  and \eqref{9.18}  into \eqref{9.20}  we can get at $t=\bar{T}_\omega^*$ that
\begin{align}
 &\|A^\frac32\bar{\textbf{v}}(\bar{T}_\omega^* ,\theta_{-\bar{T}_\omega^* }\omega,\bar{\textbf{v}}(0))\|^2\nonumber\\
&\leq Ce^{  C (\zeta_{17}(\omega)+\zeta_{16}(\omega)) } \zeta_{16}(\omega)
\left(\|A\bar{\textbf{v}}(0)\|^2+\sup\limits_{s\in[-1,0]}|y_\delta(\theta_s\omega)-z(\theta_s\omega)|^2\right).
\end{align}
Hence, let
\begin{align*}
 \bar{L}_4 (\mathfrak B^2, \omega) :=   Ce^{ C (\zeta_{17}(\omega)+\zeta_{16}(\omega)) } \zeta_{16}(\omega) ,\quad \omega\in \Omega.
\end{align*}
Then it yields
\begin{align*}
&\left \| \textbf{u}_\delta \! \left ( \bar{T}_\omega^* (\omega) ,\theta_{-\bar{T}_\omega^*(\omega)}\omega,  \textbf{u}_{0}\right)
 - \textbf{u} \! \left ( \bar{T}_\omega^* (\omega) ,\theta_{-\bar{T}_\omega^*(\omega)}\omega, \textbf{u}_{0}\right) \right \|^2_{H^3}
 \nonumber\\
& =\left \| \textbf{v}_\delta \! \left (\bar{T}_\omega^* (\omega) ,\theta_{-\bar{T}_\omega^*}\omega,  \textbf{v}_{\delta,0}\right)
 - \textbf{v} \! \left ( \bar{T}_\omega^* (\omega) ,\theta_{-\bar{T}_\omega^*(\omega)}\omega, \textbf{v}_{0}\right)+h(y_\delta(\omega)-z(\omega)) \right \|^2_{H^3}\nonumber\\
&\leq 2\left \| \textbf{v}_\delta \! \left ( \bar{T}_\omega^* (\omega) ,\theta_{-\bar{T}_\omega^*(\omega)}\omega,  \textbf{v}_{\delta,0}\right)
 - \textbf{v} \! \left ( \bar{T}_\omega^* (\omega) ,\theta_{-\bar{T}_\omega^*(\omega)}\omega, \textbf{v}_{0}\right)\right\|^2_{H^3}+2\left\|h\left (y_\delta(\omega)-z(\omega)\right ) \right \|^2_{H^3},
\end{align*}
where $\textbf{v}_{\delta,0}=\textbf{u}_0-hy_\delta(\omega)$, $\textbf{v}_{0}=\textbf{u}_0-hz(\omega)$  and $\bar{\textbf{v}}(0)=-h(y_\delta(\omega)-z(\omega))$. Moreover, we also have
\begin{align*}
&\left \| \textbf{u}_\delta \! \left ( \bar{T}_\omega^* (\omega) ,\theta_{-\bar{T}_\omega^*(\omega)}\omega,  \textbf{u}_{0}\right)
 - \textbf{u} \! \left ( \bar{T}_\omega^*(\omega) ,\theta_{-\bar{T}_\omega^*(\omega)}\omega, \textbf{u}_{0}\right) \right \|^2_{H^3}\nonumber\\
 &\leq C\bar{L}_4({\mathfrak B^2}, \omega)\left (\|A\bar{\textbf{v}}(0)\|^2+\sup\limits_{s\in[-1,0]}|y_\delta(\theta_s\omega)-z(\theta_s\omega)|^2\right )\nonumber\\
 &\leq C\bar{L}_4({\mathfrak B^2}, \omega)\left (|y_\delta(\omega)-z(\omega)|^2+\sup\limits_{s\in[-1,0]}|y_\delta(\theta_s\omega)-z(\theta_s\omega)|^2\right ).
\end{align*}
This completes the proof of Theorem \ref{lemma8.4}.
\end{proof}

\section*{Acknowledgements}
  H. Liu was supported by the National Natural Science Foundation of China (No. 12271293) and the project of Youth Innovation Team of Universities of Shandong Province (No. 2023KJ204), the Natural Science Foundation of Shandong Province (No. ZR2024MA069).  C.F. Sun was supported by the National Natural Science Foundation of China (No. 11701269) and the Natural Science Foundation of Jiangsu Province (No. BK20231301). J. Xin was supported  by the Natural Science Foundation of Shandong Province (No. ZR2023MA002).

\section*{Conflict of interest}
The authors declare that there are no conflicts of interest in this work.


\end{document}